\documentclass{article}

\usepackage{fullpage}

\usepackage{amssymb}                    
\usepackage{mathrsfs}                   
\usepackage{amsmath}                    
\usepackage{amsfonts}                   
\usepackage{amsthm}                     
\usepackage{mathtools}                  
\usepackage{stmaryrd}                   
\usepackage{makeidx}
\usepackage[usenames,dvipsnames]{xcolor}
\usepackage{graphics}
\usepackage[
  colorlinks=true,
  linkcolor=darkgreen,
  citecolor=darkgreen,
  urlcolor=darkgreen,
  bookmarksdepth=subsection,
  final
]{hyperref}
\usepackage{rotating}

\usepackage{color}
\definecolor{darkgreen}{RGB}{20,140,10}

\usepackage[all]{xy}

\usepackage{tikz}
\usetikzlibrary{
  cd,
  calc,
  arrows.meta,
  backgrounds,
  decorations,
  decorations.pathmorphing, 
}

\tikzcdset{
  arrow style=tikz, 
  diagrams={
    trim left=
    {([xshift=5pt]current bounding box.west)},
    trim right=
    {([xshift=-5pt]current bounding box.east)},
    baseline=
    {([yshift=-3pt]current bounding box.center)},
    >={
      Computer Modern Rightarrow[
        length=4pt, width=4pt
      ]
    },
    hook/.style={
      {Hooks[right, length=2pt, width=8pt]}->
    },
    hook'/.style={
      {Hooks[left, length=2pt, width=8pt]}->
    },
    shorten=0pt,
  }
}

\tikzset{
  uphordown/.style={
    rounded corners,
    to path={
      let \p1 = (\tikztostart.north),
          \p2 = (\tikztotarget.north),
          \n1 = {max(\y1,\y2) + 8pt}
      in
      (\p1)
      -- (\x1, \n1)
      -- (\x2, \n1) \tikztonodes 
      -- (\p2)
    }
  }
}

\newcommand{\inlinetikzcd}[1]{\begin{tikzcd}[sep=small, ampersand replacement=\&]#1\end{tikzcd}}

\theoremstyle{plain}
\newtheorem{theorem}{Theorem}[section]
\newtheorem{lemma}[theorem]{Lemma}

\newtheorem{proposition}[theorem]{Proposition}
\newtheorem{corollary}[theorem]{Corollary}

\theoremstyle{definition}
\newtheorem{definition}[theorem]{Definition}
\newtheorem{example}[theorem]{Example}
\newtheorem{examples}{Examples}

\newtheorem{remark}[theorem]{Remark}

\renewcommand{\(}{\begin{equation}}
\renewcommand{\)}{\end{equation}}
\newcommand{\bea}{\begin{eqnarray}}
\newcommand{\eea}{\end{eqnarray}}

\newcommand{\Mfd}{\mathrm{SmthMfd}}

\newcommand{\oo}{\infty}
\renewcommand{\d}{\mathrm{d}}

\begin{document}

\title{Synthetic Geometry of Differential Equations:\\
	I. Jets and Comonad Structure}
   \author{
     \def\arraystretch{.9}
     \begin{tabular}{c}
      Igor Khavkine\thanks{Institute of Mathematics, Czech Academy of Sciences, \v{Z}itn\'a 25, 110 00 Prague 1, Czech Republic}
      \\
      \footnotesize{\tt khavkine@math.cas.cz}
     \end{tabular}
	\;\;\;and\;\;\;
    \def\arraystretch{.9}
	\begin{tabular}{c}
      Urs Schreiber\thanks{Division of Science \& Center for Quantum and Topological Systems, New York University, Abu Dhabi, UAE}
      \\
      \footnotesize{\tt us13@nyu.edu}
    \end{tabular}
  }
\maketitle

\begin{abstract}
	We give an abstract formulation of the formal theory of partial
	differential equations (PDEs) in synthetic differential geometry, one
	that would seamlessly generalize the traditional theory to a range of
	enhanced contexts, such as super-geometry, higher (stacky)
	differential geometry, or even a combination of both. A motivation for
	such a level of generality is the eventual goal of solving the open
	problem of covariant geometric pre-quantization of locally variational
	field theories, which may include fermions and (higher) gauge fields.

	A remarkable observation of Marvan~\cite{Marvan86} is that the jet
	bundle construction in ordinary differential geometry has the
	structure of a comonad, whose (Eilenberg-Moore) category of coalgebras
	is equivalent to Vinogradov's category of PDEs. We give a synthetic
	generalization of the jet bundle construction and exhibit it as the base
	change comonad along the unit of the ``infinitesimal shape'' functor,
	the differential geometric analog of Simpson's ``de~Rham shape''
	operation in algebraic geometry. This comonad structure coincides with
	Marvan's on ordinary manifolds. This suggests to consider PDE theory
	in the more general context of any topos equipped with an
	``infinitesimal shape'' monad (a ``differentially cohesive'' topos).
	We give a new natural definition of a category of formally integrable
	PDEs at this level of generality and prove that it is always
	equivalent to the Eilenberg-Moore category over the synthetic jet
	comonad. When restricted to ordinary manifolds, Marvan's result shows
	that our definition of the category of PDEs coincides with
	Vinogradov's, meaning that it is a sensible generalization in the
	synthetic context.

	Finally, we observe that whenever the base space $\Sigma$ is formally smooth,
    the category of formally integrable PDEs with independent
	variables ranging in $\Sigma$ is also equivalent simply to the slice
	category over $\Im\Sigma$. This yields in particular a convenient site
	presentation of the categories of PDEs in general contexts.
\end{abstract}

\newpage

\tableofcontents

\newpage

\section*{Notation}

\begin{tabular}{lll}
	$\mathcal{C}_{/c}$
	& slice category  & \ref{SliceCategory}
    \\
    $\mathcal{C}/X$ & slice site & \ref{SliceSite}
    \\
	$[Z \to X]$
	& object in slice
   \\
	$\mathrm{FormalSmoothSet} = \mathrm{Sh}(\mathrm{FormalCartSp})$
	& Cahiers topos & \ref{FormalSmoothSets}
    \\
    $\mathbf{H}$ & arbitrary differentially cohesive topos & \ref{DifferentialCohesion}
    \\
    $\mathrm{SmthMfd} \hookrightarrow \mathrm{FormalSmoothSet}$ & smooth manifolds &  \ref{SmoothManifolds}
    \\
    $\mathrm{DiffSp} \hookrightarrow \mathrm{FormalSmoothSet}$ & diffeological spaces & \ref{DiffeologicalSpace}
    \\
    $\mathrm{SmoothSet} \hookrightarrow \mathrm{FormalSmoothSet}$ & smooth sets & \ref{SmoothManifolds}
    \\
    $\mathrm{LocProMfd} \hookrightarrow \mathrm{FormalSmoothSet}$ & smooth locally pro-manifolds & \ref{LocallyProManifold}
    \\
    $\mathrm{LocProMfd}_{\downarrow \Sigma} \hookrightarrow \mathrm{LocProMfd}_{/\Sigma}$ & fibered manifolds & \ref{FiberedManifold}
   \\
     $\xymatrix{X \ar[r]^-{\mathrm{et}} & Y}$ & formally {\'e}tale morphism & \ref{FormallyEtaleMorphism}
   \\
   $\xymatrix{V \ar@{<-}[r]^-{\mathrm{et}} & U \ar@{->>}[r]^-{\mathrm{et}} & X }$ & $V$-atlas for $V$-manifold $X$ & \ref{VManifold}
   \\
	$\mathbb{D}^n(k) \hookrightarrow \mathbb{R}^n$
	& standard infinitesimal $n$-disk of order $k$  & \ref{StandardInfinitesimalnDisk}, \ref{InfinitesimalDiskInManifold} \\
	$\mathbb{D}_x \hookrightarrow X$ & abstract infinitesimal disk around $x$ in $X$
	& \ref{FormalDiskAbstractly} \\
	$T^\oo \Sigma$
	& formal disk bundle of $\Sigma$ & \ref{InfinitesimalDiskBundle}  \\
	$T^\infty_\Sigma : \mathbf{H}_{/\Sigma} \to \mathbf{H}_{/\Sigma}$ &
	formal disk bundle functor & --\\
	$\pi\colon T^\oo \Sigma \to \Sigma$
	& sends each formal disk to its origin & -- \\
	$\mathrm{ev}\colon T^\oo \Sigma \to \Sigma$
	& sends each formal disk to itself & -- \\
	$\eta_E : E \to T^\infty_\Sigma E$
	& unit of $T^\oo_\Sigma$ at $E$
		& \ref{MonadOperationsOnInfinitesimalDiskBundles}  \\
	$\nabla_E$
	& product of $T^\oo_\Sigma$ at $E$
		& -- \\
	$J^\oo_\Sigma : \mathbf{H}_{/\Sigma} \to \mathbf{H}_{/\Sigma}$
	& jet bundle functor & \ref{jetcomonad}\\
	$\epsilon_E$
	& counit of $J^\oo_\Sigma$ at $E$ & -- \\
	$\Delta_E$
	& coproduct of $J^\oo_\Sigma$ at $E$ & -- \\
	$L \dashv R$ & adjoint functors
		& \ref{AdjointFunctor} \\
	$\widetilde{\sigma} \colon E \to J^\oo_\Sigma Y$
	& $(T^\oo_\Sigma \dashv J^\oo_\Sigma)$-adjunct morphism of $\sigma \colon T^\oo_\Sigma E \to Y$
		& \ref{jetcomonad} \\
	$\overline{\tau} \colon T^\oo_\Sigma E \to Y$
	& $(T^\oo_\Sigma \dashv J^\oo_\Sigma)$-adjunct morphism of $\tau \colon E \to J^\oo_\Sigma Y$
		& -- \\
	$j^\oo \sigma\colon \Sigma \to J^\oo_\Sigma E$
	& jet prolongation of section $\sigma\colon \Sigma \to E$
		& \ref{jetprolongation} \\
	$\hat{D}\colon \Gamma_\Sigma(E) \to \Gamma_\Sigma(F)$
	& differential operator induced by $D\colon J^\oo E \to F$
		& \ref{DifferentialOperatorsAsMapsOutOfJetBundle} \\
	$D = \underline{\hat{D}}$
	& formal differential operator underlying $\hat{D}$
		& -- \\
	$p^\oo D\colon J^\oo_\Sigma E \to J^\oo_\Sigma F$
	& prolongation of formal differential operator $D$
		& \ref{InfiniteProlongation} \\
	$\mathrm{DiffOp}(\mathcal{C}) \simeq \mathrm{Kl}(J^\oo_\Sigma|_\mathcal{C})$
	& \begin{tabular}{l}
			differential operators in $\mathcal{C} \hookrightarrow \mathbf{H}_{/\Sigma}$, \\
			co-Kleisli category of $J^\oo_\Sigma$ \end{tabular}
		& \ref{DifferentialOperatorsCategory}\\
	$\mathrm{DiffOp}_{\downarrow\Sigma}(\mathrm{LocProMfd})$
	& $\mathcal{C} = \mathrm{LocProMfd}_{\downarrow\Sigma}$, in fibered manifolds
		& \ref{DifferentialOperatorsCategory} \\
	$\mathrm{DiffOp}_{/\Sigma}(\mathbf{H})$
	& $\mathcal{C} = \mathbf{H}_{/\Sigma}$, generalized differential operators
		& \ref{DifferentialOperatorGeneralized} \\
	$\mathcal{E} \hookrightarrow J^\infty_\Sigma Y$ & generalized PDE & \ref{generalpdes} \\
	$\mathcal{E}^\oo \hookrightarrow \mathcal{E}$
	& canonical inclusion of prolongation of PDE $\mathcal{E}$
		& \ref{FormallyIntegrable}\\
	$\mathrm{PDE}(\mathcal{C}) \simeq \mathrm{EM}(J^\oo_\Sigma|_\mathcal{C})$
	& \begin{tabular}{l}
			formally integrable PDEs in $\mathcal{C} \hookrightarrow \mathbf{H}_{/\Sigma}$, \\
			coalgebras over $J^\oo_\Sigma$ \end{tabular}
		& \ref{FormallyIntegrable}, \ref{PDEIsEM} \\
	$\mathrm{PDE}_{\downarrow\Sigma}(\mathrm{LocProMfd})$
	& $\mathcal{C} = \mathrm{LocProMfd}_{\downarrow\Sigma}$, in fibered manifolds
		& \ref{RemarksOnPDECategory} \\
	$\mathrm{PDE}_{/\Sigma}(\mathbf{H})$
	& $\mathcal{C} = \mathbf{H}_{/\Sigma}$, generalized PDEs
		& -- \\
	$\rho_{\mathcal{E}} \colon \mathcal{E} \to J^\oo_\Sigma \mathcal{E}$
	& jet coalgebra structure & \ref{InputForPDEisEM}
\end{tabular}

\newpage

\section{Introduction and Summary}

Local variational calculus---equivalently, local Lagrangian classical
field theory---concerns the analysis of variational partial differential
equations (PDEs), their symmetries and other properties.
It is most usefully formulated in the language
of jet bundles (e.g.,~\cite{Vinogradov86,Olver93,VinKras,Anderson-book}). Much of the time, in this
formalism, it is sufficient to work in local coordinates, which leads to
very useful but quite intricate formulas. Unfortunately, this means that
often certain technical aspects go ignored or unnoticed, which includes
the consequences of the global topology of the spaces of either dependent
or independent variables, as well as aspects of analysis on infinite
dimensional jet bundles. Of course, in order to consider some
generalizations of this framework (to, for instance, supermanifolds,
differentiable spaces more singular than manifolds, or even the more
general spaces of higher geometry), it is desirable to have a
comprehensive discussion of these issues based on a precise and
sufficiently flexible technical foundation.

In this work, we achieve this goal by giving a succinct and transparent
formalization of jet bundles and PDEs (eventually also laying the ground
for a formalization of local variational calculus) within a ``convenient
category'' for differential geometry. A key example of such a category
is Dubuc's ``Cahiers topos'' \cite{Dubuc79} (reviewed below in
section~\ref{FormalSmoothSet}), which was originally introduced as a
well-adapted model of the Kock-Lawvere axioms for \emph{synthetic
differential geometry}~\cite{KockBookA}.

Here ``synthetic'' means that the axioms formulate natural properties of
objects (for instance that to every object $X$ there is naturally
associated an object $T X$ that behaves like the tangent bundle of $X$),
instead of prescribing the objects themselves. Synthetic axioms instead
prescribe what the \emph{category} of objects is to be like (for
instance that it carries an endomorphism $T$ with a natural
transformation $T \to \mathrm{id}$). While our formulation of
differential geometry is ``synthetic'' in this sense, we do not actually use the
Kock-Lawvere axiom scheme, but another axiom called ``differential
cohesion'' in \cite{dcct,Wellen17,Cherubini2024} or ``elasticity'' in \cite{SS20-Orb} (inspired by the ``axiomatic cohesion'' of \cite{Lawvere07}, going back to \cite{Lawvere94}), see proposition
\ref{diffcohesionofFormalSmoothSet} and definition
\ref{ReImEtAdjunctionOnFormalSmoothSets} below.

The Cahiers topos is the category of sheaves on formal Cartesian spaces
(meaning $\mathbb{R}^n$'s, possibly ``thickened'' by infinitesimal
directions). From the perspective of the foundations of what is called
the \emph{formal theory} of PDEs, having access to spaces with actual
infinitesimal directions turns out to be of great practical and
conceptual help. Traditionally, the formal (meaning infinitesimal)
aspects of PDE theory have only been treated informally (meaning
heuristically). But in the Cahiers topos, such heuristic arguments can
actually be made precise. In a technical sense, this category is
``convenient'' because it fully-faithfully includes ordinary manifolds%
	\footnote{This statement may require some interpretation for those not
	familiar with the use of sheaves in category theory. A manifold $X$
	may be characterized by its coordinate atlas, which technically forms
	a sheaf on the category of all Cartesian spaces $\mathbb{R}^n$, with
	$n=\dim X$, and local diffeomorphisms as arrows. The Cahier topos
	generalizes this correspondence by dropping the dimensionality
	condition on $\mathbb{R}^n$ and the smooth invertibility condition on
	arrows. In addition the Cahiers topos contains infinitesimal spaces.
	\label{sheaves-ftn}} 
but also ensures the existence of objects resulting from constructions
with intersections, quotients and limits, which may be too singular or
too infinite dimensional to correspond to ordinary manifolds. In
particular the categories of infinite-dimensional Fr{\'e}chet manifolds
and of smooth locally pro-manifolds, $\mathrm{LocProMfd}$ (introduced in
section~\ref{FrechetManifolds} below), embed fully faithfully into the
Cahiers topos, this we discuss in section~\ref{FrechetManifolds}. Locally
pro-manifolds constitute a slight extension of the category of smooth
manifolds that admits spaces locally modeled on projective limits of
smooth manifolds. This class of spaces has been used extensively, though
only semi-explicitly, as a minimalistic setting for discussing
analytical aspects of infinite jet
bundles~\cite{Takens79,Michor80,Saunders89,Marvan86,GuPf13}. By giving
it a precise definition and by embedding it fully faithfully in the
Cahiers topos, we finally provide a precise and flexible categorical
framework for locally pro-manifolds.

Another ``convenience,'' which will be of more use in followup works,
after we lay down the groundwork here, is the
existence of moduli spaces (or more loosely \emph{classifying spaces})
of differential form data as bona fide objects in this category.

In \ref{FormalDisks} we discuss how a differentially cohesive topos
$\mathbf{H}$, such as the Cahiers topos, comes equipped with a natural
and intrinsic notion of formal infinitesimal neighborhoods, which are
formalized with the help of a monad functor $\Im\colon \mathbf{H} \to
\mathbf{H}$ with the property that the following pullback diagram in
$\mathbf{H}$ defines the formal infinitesimal neighborhood bundle
$T^\infty \Sigma \to \Sigma$
$$
  \begin{gathered}
  \xymatrix{
    T^\infty \Sigma \ar[r]_-{\ }="s" \ar[d] & \Sigma \ar[d]^{\eta_\Sigma}
    \\
    \Sigma \ar[r]_{\eta_\Sigma}^-{\ }="t" & \Im \Sigma
    \ar@{}|{\mbox{\tiny (pb)}} "s"; "t"
  }
  \end{gathered}
  \,,
$$
where $\eta_\Sigma \colon \Sigma\to \Im \Sigma$ the unit (natural
transformation) of the monad. We may call $\Im \Sigma$ the ``de~Rham
shape'' (or ``infinitesimal shape'') of $\Sigma$, by analogy with a
well-known construction in algebraic geometry~\cite{Simpson96}. We
recall basics of (co-)monads in section \ref{monads}.

Since infinite-order jets are essentially functions on formal
infinitesimal neighborhoods, we may synthetically define the
construction
$$
  J^\oo_\Sigma
    \;:\;
  \mathbf{H}_{/\Sigma}
    \longrightarrow
  \mathbf{H}_{/\Sigma}
$$
of jet bundles over some base manifold $\Sigma$ as nothing but the
\emph{base change comonad} (definition \ref{BaseChangeAdjunctions}
below) along the unit of $\Im$, thus abstractly exhibiting its comonad
structure. This is the content of section \ref{Jets}. (Here
$\mathbf{H}_{/\Sigma}$ denotes the slice topos, recalled as definition
\ref{SliceCategory} below.)

Aspects of the theory of jet bundles have been formalized in synthetic
differential geometry before, in~\cite{Kock80}, but only at finite jet
order and without the comonad structure. It is important to note that
for us the $J^\oo_\Sigma$ functor is then defined on any synthetic
category $\mathbf{H}$, such as the Cahiers topos, which when suitably
restricted to ordinary smooth bundles over $\Sigma$ (spaces of dependent
variables) yields the usual infinite jet bundles, with their canonical
comonad structure coinciding with the explicit comonad structure
previously observed in the work of Marvan~\cite{Marvan86} (presented in
expanded form in~\cite{Marvan89}). This we prove as theorem
\ref{ReproducingTraditionalJetBundle} below.

The central result in~\cite{Marvan86} is that key aspects of the theory
of differential operators and PDEs over $\Sigma$ can be reinterpreted as
certain well-known categorical constructions. For instance, given a
comonad like $J^\infty_\Sigma$ acting on the category of bundles over
$\Sigma$, we can consider the corresponding category of cofree
coalgebras over it (the \emph{Kleisli category} over the comonad,
definition~\ref{coKleisli} below), with the result being precisely the
category of (possibly non-linear) partial differential operators over
$\Sigma$:
$$
  \mathrm{DiffOp}_{\downarrow\Sigma}(\mathrm{LocProMfd})
    \simeq
  \mathrm{Kl}(J^\infty_\Sigma|_{\mathrm{LocProMfd}_{\downarrow \Sigma}})
  \,.
$$
This we discuss in section~\ref{DifferentialOperators}.

Moreover, dropping the cofreeness condition, the category of
coalgebras over the jet comonad (the \emph{Eilenberg-Moore category}
over the comonad, definition~\ref{EMOverComonad} below) is equivalent
to Vinogradov's category of formally integrable partial differential
equations (PDEs)~\cite{Vinogradov80,Vinogradov86,VinKras} with free variables ranging
in $\Sigma$:
$$
  \mathrm{PDE}_{\downarrow\Sigma}(\mathrm{LocProMfd})
    \simeq
  \mathrm{EM}(J^\infty_\Sigma|_{\mathrm{LocProMfd}_{\downarrow \Sigma}})
  \,.
$$

In section~\ref{PDEs} we prove a  synthetic generalization of these
results, which for classical manifolds were observed in \cite{Marvan86}. First we consider a general definition of the
category $\mathrm{PDE}_{/\Sigma}(\mathbf{H})$ of formally integrable
PDEs in any topos $\mathbf{H}$ equipped with an infinitesimal shape
operation $\Im$. This is definition~\ref{FormallyIntegrable} below. Then
we show that generally there is an equivalence of categories
$$
  \mathrm{PDE}_{/\Sigma}(\mathbf{H})
    \simeq
  \mathrm{EM}(J^\infty_\Sigma)
  \,.
$$
Moreover, for $\mathcal{C} \hookrightarrow \mathbf{H}_{/\Sigma}$
any full subcategory stable under forming jet bundles via
$J^\infty_\Sigma$,  the above restricts to an equivalence
$$
  \mathrm{PDE}(\mathcal{C})
    \simeq
  \mathrm{EM}(J^\infty_\Sigma|_{\mathcal{C}})
  \,.
$$
This is theorem~\ref{PDEIsEM} below. In the special case that
$\mathcal{C} = \mathrm{LocProMfd}_{\downarrow \Sigma}$,
$\mathrm{PDE}_{\downarrow\Sigma}(\mathrm{LocProMfd}) :=
\mathrm{PDE}(\mathcal{C})$ is the
category of fibered locally pro-manifolds over $\Sigma$ (definition
\ref{LocallyProManifold} below). This reduces via
\cite{Marvan86} to an equivalence between our synthetic PDEs and
Vinogradov's classical definition of PDEs (this is corollary \ref{GeneralizedPDEVsVinogradovCorollary}
below, see remark \ref{GeneralizedPDEVsVinogradov}).

An advantage of working with a synthetic formalization of PDEs that is
worth mentioning is the ability to form products and equalizers
(theorem~\ref{PDEHasKernels}), hence also arbitrary finite limits
(corollary~\ref{PDEHasFiniteLimits}), in the resulting category. From a
``big picture'' point of view, having worked out the notion of a category
of PDEs at this abstract level, this notion can now be extended in a
mechanical way to the supergeometric version of the Cahiers
topos~\cite{Yetter88}, or to an $\oo$-topos of $\oo$-stacks, or
even a combination of both \cite{dcct}.

Finally we observe in section~\ref{TheToposOfSyntheticPDEs} that whenever the
unit $\Sigma \to \Im \Sigma$ of the infinitesimal shape operation on
$\Sigma$ is epimorphic. Hence, if $\Sigma$ is formally smooth (Definition \ref{FormallyEtaleMorphism}),
which is true for all objects in the Cahiers topos (Proposition \ref{UnitOfInfinitesimalShapeIsEpiInCahiersTopos}), there is a further equivalence of categories
$$
  \mathrm{PDE}_{/\Sigma}(\mathbf{H})
    \simeq
  \mathbf{H}_{/\Im \Sigma}
$$
which identifies the category of formally integrable PDEs in
$\mathbf{H}$ over $\Sigma$ simply with the slice topos over the
infinitesimal shape $\Im \Sigma$. This is theorem
\ref{generalPDEsIsSliceOverInfinitesimalShape} below.

In algebraic geometry the linear version of this equivalence is
familiar: linear differential equations over a scheme $\Sigma$
incarnated as $\mathcal{D}$-modules over $\Sigma$ are equivalently
quasicoherent sheaves over the de Rham shape $\Im \Sigma$
(e.g.,~\cite[above theorem 0.4]{Lurie09}, \cite[sections 2.1.1 and
5.5]{GaitsgoryRozenblyum11}). Under this identification the
algebro-geometric analog of the category
$\mathrm{PDE}_{\downarrow\Sigma}(\mathrm{LocProMfd})$ of general
(non-linear) PDEs in manifolds are the \emph{$\mathcal{D}$-schemes}
of~\cite[chapter 2.3]{BeilinsonDrinfeld04}.

The above equivalence of categories has various interesting
implications. One is that it allows to exhibit PDEs in the generalized sense as
sheaves over (and hence as colimits of) PDEs in the ordinary sense: There is an equivalence of
categories
$$
  \mathrm{PDE}_\Sigma(\mathbf{H})
    \;\simeq\;
  \mathrm{Sh}( \mathrm{PDE}_\Sigma(\mathrm{FormalLocProMfd}) )
  \,.
$$
This is theorem \ref{PDESite} below.

Particularly interesting sheaves on the ordinary category of PDEs are
the variational bicomplex and the Euler-Lagrange complex of variational
calculus~\cite{Anderson-book,Olver93,VinKras}. With the above
equivalence this means that the collection of all Euler-Lagrange complexes
is secretly a single complex of generalized PDEs, in fact it turns out
to be the \emph{moduli stack} of variational calculus. The details will
be discussed in followup works.

\medskip
\medskip

\paragraph{Acknowledgments.}
We thank Dave Carchedi for discussion of pro-manifolds and Michal Marvan
for discussions of his work on the jet comonad.
We furthermore thank Grigorios Giotopoulos for discussion of the proof of Thm. 3.22, and the anonymous referee for their careful reading and comments on the manuscript.
Some typos in an earlier
version of this text were spotted by Dmitri Pavlov and Thomas Holder. 

IK was partially supported by the ERC Advanced Grant 669240 QUEST ``Quantum
Algebraic Structures and Models'' at the University of Rome 2 (Tor
Vergata), GA\v{C}R grant GA22-00091S and RVO: 67985840 at the Czech
Academy of Sciences (Prague). IK and US also thank the Max Planck
Institute for Mathematics (Bonn) for kind hospitality while this work
was being written up, a fair number of years ago.

\section{Spaces}
\label{CategoriesOfSPaces}

We work in a ``convenient category'' for differential geometry, which
faithfully contains smooth manifolds, but also contains more general
smooth spaces. The idea is that, even when we are interested only in
smooth manifolds, for some constructions it is still convenient to be
able to pass through this larger category that contains them.

The actual model we use is the category $\mathbf{H}$ of \emph{formal smooth sets}, namely of sheaves on the
category of all formal manifolds (the \emph{Cahiers topos} due to \cite{Dubuc79}, in the equivalent $\mathbb{R}$-algebraic form of \cite{dcct}, cf. \cite[\S 4.1]{GS26-FieldsII}), containing the
category $\mathbf{H}_\Re$ of plain \emph{smooth sets}, namely of sheaves on the category of ordinary
manifolds, recalled below in section~\ref{FormalSmoothSet}. This model
serves to exhibit our constructions in section~\ref{Jets} as subsuming
and generalizing traditional constructions in differential geometry.

But the only property of $\mathbf{H}$ that we actually use for the
formal development of variational calculus below is that the inclusion
$\mathbf{H}_\Re \hookrightarrow \mathbf{H}$ is exhibited by an
idempotent monad $\Re$ (``reduction'') which has two further right
adjoint functors:
$$
  (\Re \; \dashv \; \Im \; \dashv \; \mathrm{E}\!\mathrm{t})
    \;:\;
  \mathbf{H} \longrightarrow \mathbf{H}
  \,.
$$
See proposition~\ref{diffcohesionofFormalSmoothSet} and
definition~\ref{ReImEtAdjunctionOnFormalSmoothSets} below.

For technical aspects of the discussion of jet  bundles in
section~\ref{Jets}, it is important that $\mathbf{H}_\Re$ (hence also
$\mathbf{H}$) not only contains finite dimensional manifolds fully
faithfully, but also infinite-dimensional Fr{\'e}chet manifolds, in particular
manifolds locally modeled on $\mathbb{R}^n$, for $n \in \mathbb{N}\cup \{\infty\}$. This we
recall below in section~\ref{FrechetManifolds}.

\subsection{Formal smooth sets}
 \label{FormalSmoothSet}

Here we discuss the category of ``formal smooth sets,'' which is the
category in which all our developments of variational calculus in
section~\ref{Jets} take place. A ``formal smooth set'' is a sheaf on the
category of Cartesian spaces (which can be interpreted as a generalized
differentiable space in the sense of footnote~\ref{sheaves-ftn}) that
may be equipped with an infinitesimal thickening.

The most immediate way to improve smooth manifolds to a ``convenient
category'' (a topos) is to pass to sheaves over the category of all
smooth manifolds. These sheaves we may think of as very general ``smooth
sets.''

\begin{definition}
  \label{SmoothManifolds}
  (a) Write
  \begin{itemize}
    \item $\mathrm{SmthMfd}$ for the category of finite-dimensional paracompact smooth ($C^\infty$) manifolds with a countable set of connected components;
    \item $\mathrm{CartSp} \hookrightarrow \mathrm{SmthMfd}$ for its full subcategory on the Cartesian spaces $\mathbb{R}^n$, $n \in \mathbb{N}$;
  \end{itemize}
	We agree to say ``manifold'' for the objects of $\mathrm{SmthMfd}$ from now on.

	(b) Regard $\mathrm{SmthMfd}$ and $\mathrm{CartSp}$ as sites
	(definition \ref{site}), equipped with the Grothendieck pre-topology
	of (good) open covers, where ``good'' as usual means locally finite
	with contractible multiple intersections. Write
  $$
    \mathrm{SmoothSet} := \mathrm{Sh}(\mathrm{SmthMfd}) \simeq \mathrm{Sh}(\mathrm{CartSp})
  $$
	for the corresponding category of sheaves on smooth manifolds
	(``smooth sets'').
\end{definition}
We are to think of an object $X$  in $\mathbf{H}_\Re$ as a generalized
smooth space, defined not in terms of an underlying set of points, but
defined entirely by a consistent declaration of the sets $X(U)$ of
smooth maps ``$U \to X$'' out of smooth manifolds into the would-be
space $X$.
\begin{example}
	The Yoneda embedding $M \mapsto \mathrm{Hom}(-,M)$ constitutes a fully
	faithful inclusion
  $$
    \mathrm{SmthMfd} \hookrightarrow \mathrm{SmoothSet}
  $$
  of smooth manifolds into smooth sets.
\end{example}
\begin{example}[{\cite[Prop. 3.7.5, 4.4.15]{dcct}}] 
 \label{DiffeologicalSpace}
    A \emph{diffeological space} (or \emph{Chen
	smooth space}) is a ``concrete'' smooth set, in the sense of concrete
	sheaves, i.e.,\ a sheaf $X$ on $\mathrm{SmthMfd}$ with the property that
	for each $U \in \mathrm{SmthMfd}$
	there is a natural inclusion
  $$
    X(U) \hookrightarrow \mathrm{Hom}_{\mathrm{Set}}\big(U_s, X(\mathbb{R}^0)\big)
  $$
	of the set $X(U)$ (of functions declared to be smooth) into the set of
	all functions from the underlying set 
    $
      U_s
      :=
      \coprod_{ u \in U }
      \mathbb{R}^0
    $ 
    of $U$ into the ``underlying set'' $X(\mathbb{R}^0)$ of $X$.

   For example, the sheaf of differential 1-forms represents a smooth set but as such is not concrete, since it has a single underlying point $\Omega^1_{\mathrm{dR}}(\mathbb{R}^0) \simeq \ast$ (the zero 1-form) while the sets $\Omega^1_{\mathrm{dR}}(\mathbb{R}^{k})$ for $k > 0$ are infinite and hence far from being subsets of $\mathrm{Hom}_{\mathrm{Set}}(\mathbb{R}^k_s, \ast) \simeq \ast$.  

   More systematically, for each smooth set $X$ we have its \emph{codiscrete} smooth set $\sharp X$ given by
   $$
     \sharp X
     \;:\;
     U 
       \mapsto 
     \mathrm{Hom}_{\mathrm{SmoothSet}}\big(
      \textstyle{\underset{u \in U}{\coprod}}
      \mathbb{R}^0
       ,\,
       X
     \big)
     \simeq
     \mathrm{Hom}_{\mathrm{Set}}\big(
       U_s
       ,\,
       X(\mathbb{R}^0)
     \big)     
     \,,
   $$
   which receives a natural map from the original smooth set (the counit of an idempotent comonad structure on $\sharp$):
   $$
     \eta^{\sharp}_X
     :
     X \longrightarrow \sharp X
     \,,
   $$
   given over each $U$ by precomposition with the inclusion $\iota_U : U_s \hookrightarrow U$ within $\mathrm{SmoothSet}$: 
   $$
     \begin{tikzcd}
     \eta^\sharp_X(U)
     :
     X(U)
     =
     \mathrm{Hom}_{\mathrm{SmoothSet}}(U,X)
     \ar[
       r,
       "{
         (-)
         \circ
         \iota_U 
        }"
     ]
     &
     \mathrm{Hom}_{\mathrm{SmoothSet}}\big(
       U_s,
       X
     \big)
     \equiv
     (\sharp X)(U)
     \mathrlap{\,.}
     \end{tikzcd}
   $$
   The diffeological spaces are precisely those smooth sets $X$ for which $\eta^\sharp_X$ is a monomorphism.

	A homomorphism of diffeological spaces is a morphism of the
	corresponding sheaves. Smooth manifolds form a full subcategory of
	diffeological spaces, which in turn form a full subcategory of smooth
	sets:
  $$
    \mathrm{SmthMfd} \hookrightarrow \mathrm{DiffSp} \hookrightarrow \mathrm{SmoothSet}
    \,.
  $$
\end{example}
The intuition about the objects in $\mathrm{SmoothSet}$ as being smooth spaces defined by a rule for how to \emph{locally}
probe them by smoothly parameterized ways of mapping points into them is further supported by the
following fact:
\begin{proposition}[{\cite{dcct}}]
  \label{StalksOnCartSp}
	For each $n \in \mathbb{N}$ consider the operation of forming the
	stalk of some sheaf $X \in \mathrm{SmoothSet} :=
	\mathrm{Sh}(\mathrm{CartSp})$ at the origin of $\mathbb{R}^n$
  $$
    n^\ast X
    :=
    \varinjlim_{\delta > 0} X(B^n_{\delta})
    \,.
  $$
  Here $\delta \in \mathbb{R}_{> 0}$ and $B^n_\delta \hookrightarrow \mathbb{R}^n$
  denotes the ball of radius $\delta$ around the origin in $\mathbb{R}^n$, and
  the colimit is over the diagram of inclusions $B^n_{\delta_1} \hookrightarrow B^n_{\delta_2}$
  for all $\delta_1 < \delta_2$. (Here we are implicitly using diffeomorphisms $B^n_\delta \simeq \mathbb{R}^n$
  in order to evaluate $X$ on $B^n_{\delta}$.)

  Then:
  \begin{enumerate}
    \item the functor $n^\ast$ is a \emph{point of the topos}  $\mathrm{SmoothSet}$ in the sense of definition \ref{SitePoint};
    \item the set $\{n^*\}_{n \in \mathbb{N}}$ constitutes \emph{enough points} for $\mathrm{SmoothSet}$ in the sense of
    definition \ref{SitePoint}.
  \end{enumerate}
\end{proposition}

In order to appreciate the following constructions, it is useful to
recall the following fact:
\begin{proposition}[{\cite[\textsection35.8--10]{KolarMichorSlovak93}}]
  \label{EmbeddingOfManifoldsInRAlgebras}
	The functor that sends a smooth manifold $X$ to the
	$\mathbb{R}$-algebra of its smooth functions is fully faithful, hence
	exhibits a full subcategory inclusion
  $$
    C^\infty(-)
      \;:\;
    \mathrm{SmthMfd} \hookrightarrow \mathrm{CAlg}_\mathbb{R}^{\mathrm{op}}
  $$
  of smooth manifolds into the opposite of commutative $\mathbb{R}$-algebras.
\end{proposition}
(Notice that this statement is \emph{not} restricted to compact
manifolds.) It is worth pointing out the following trivial
\begin{examples} \label{SinglePoint}
	We can interpret a singleton set $*$ as a zero-dimensional manifold in
	$\mathrm{SmthMfd}$, which provides the identification $C^\infty(*) =
	\mathbb{R}$.
\end{examples}

We may hence generalize the category of smooth manifolds by
expanding it inside the opposite of the category of commutative
$\mathbb{R}$-algebras.

\begin{definition}
  \label{InfinitesimallyThickenedPoints}
  (a) Write
  $$
    \mathrm{InfThPoints}
      \hookrightarrow
    \mathrm{CAlg}_{\mathbb{R}}^{\mathrm{op}}
  $$
	for category of ``infinitesimally thickened points'', being the full subcategory
    of the opposite of that of commutative
	$\mathbb{R}$-algebras, on those whose underlying $\mathbb{R}$-vector
	space is of the form
  $$
    C^\infty(\mathbb{D}) := \mathbb{R} \oplus V
    \,,
  $$
	where $V$ is a finite dimensional nilpotent ideal, i.e., such that
	there exists $n \in \mathbb{N}$ with $V^n = 0$. Hence these  are
	\emph{local Artin $\mathbb{R}$-algebras}; in the literature on
	synthetic differential geometry they are called \emph{Weil algebras}.

	(b) The formal dual of such a Weil algebra we generically denote by
	$\mathbb{D}$ and think of it as an ``infinitesimally thickened
	point''. Write
  $$
    \mathrm{Sh}(\mathrm{InfPoint})
  $$
  for the category of presheaves on this category--- these are sheaves with respect to the induced Grothendieck topology on $\mathrm{InfPoint}$ from that on $\mathrm{FormalCartSp}$ from Definition  \ref{FormalSmoothSets} since there are no non-trivial open covers of a point.

	(c) Any infinitesimally thickened point $\mathbb{D}$ contains a
	canonical ordinary point, given as the image of the morphism ${*} \to
	\mathbb{D}$, which is dual to the projection of $C^\infty(\mathbb{D}) =
	\mathbb{R} \oplus V \to (\mathbb{R} \oplus V) / V \cong \mathbb{R}$.
\end{definition}
The key fact for discussion of these infinitesimally thickened points is this:
\begin{proposition}[Hadamard's lemma {\cite[Lemma 2.2.7]{duistermaat-kolk-1}}]
  \label{HadamardLemma}
  For $f : \mathbb{R} \to \mathbb{R}$ a smooth function, there exists a smooth function
  $g : \mathbb{R}\to \mathbb{R}$ such that
  $$
    f(x) = f(0)+ x \cdot g(x)
    \,.
  $$
  It follows that $g(0) = f'(0)$ is the derivative of $f$ at 0, and more generally that
  for every $k \in \mathbb{N}$ then the remainder of the partial Taylor expansion of $f$ at 0
  to order $k$ is a smooth function $h : \mathbb{R} \to \mathbb{R}$:
  $$
    f(x) = f(0) + f'(0)\cdot x + \tfrac{1}{2} f''(0) \cdot x^2 + \cdots + x^{k+1} \cdot h(x)
    \,.
  $$
  Still more generally, this implies that for $f : \mathbb{R}^n \to \mathbb{R}$ a smooth function,
  the remainder of any partial Taylor expansion in partial derivatives are smooth functions
  $h_{i_1\cdots i_{k+1}} : \mathbb{R}^n \to \mathbb{R}$:
  $$
    f(\vec x)
      =
    f(0)
      +
    \sum_{i = 1}^n \frac{\partial f}{\partial x^i}(0) \cdot x^i
      +
    \tfrac{1}{2} \sum_{i,j = 1}^n \frac{\partial^2 f}{\partial x^i \partial x^j}(0) \cdot x^i \cdot x^j
      +
    \cdots
      +
    \sum_{i_1, \cdots, i_{k+1} = 1}^n x^{i_1} \cdots x^{i_{k+1}} \cdot h_{i_1\cdots i_{k+1}}(\vec x)
    \,.
  $$
\end{proposition}
\begin{definition}
  \label{StandardInfinitesimalnDisk}
  For $n \in \mathbb{N}$ write
  $$
    \mathbb{D}^n(k) \in \mathrm{InfPoint}
  $$
  for the infinitesimally thickened point whose corresponding algebra is
  a \emph{jet algebra}, namely a quotient algebra of the form
  $$
    C^\infty(\mathbb{R}^n)/ (x^1, \cdots, x^n)^{k+1}
    \,,
  $$
  where $\{x^i\}_{i = 1}^n$ are the canonical coordinates on $\mathbb{R}^n$.
  We call this the \emph{standard infinitesimal $n$-disk of order $k$}.
\end{definition}
The archetypical example is the following:
\begin{example}
  \label{RingOfDualNumbers}
	Write $\mathbb{D}^1(1)$ for the infinitesimally thickened point
	(definition~\ref{InfinitesimallyThickenedPoints}) corresponding to the
	``ring of dual numbers'' over $\mathbb{R}$, i.e.,
  $$
    C^\infty(\mathbb{D}^1(1)) := \mathbb{R} \oplus \varepsilon \mathbb{R}
  $$
  with $\varepsilon^2 = 0$.
	By Hadamard's lemma (proposition~\ref{HadamardLemma}), this is
	equivalently the quotient of the $\mathbb{R}$-algebra of smooth
	functions on $\mathbb{R}$ by the ideal generated by $x^2$ (for $x
	\colon \mathbb{R} \to \mathbb{R}$ the canonical coordinate function):
  $$
    C^\infty(\mathbb{D}^1(1))
    \simeq
    C^\infty(\mathbb{R}^1)/(x^2)
    \,.
  $$
\end{example}
\begin{proposition}[e.g.,\ {\cite[proposition 4.43]{CarchediRoytenberg12}}]
  \label{InfinitesimallyThickenedPointsInsideInfinitesimalDisks}
  Every infinitesimally thickened point $\mathbb{D}$ (definition~\ref{InfinitesimallyThickenedPoints})
  embeds into an infinitesimal $n$-disk of some order $k$ (example \ref{StandardInfinitesimalnDisk})
  $$
    \mathbb{D} \hookrightarrow \mathbb{D}^n(k)
    \,.
  $$
  Dually, every Weil algebra $\mathbb{R} \oplus V$ is a quotient of a jet algebra
  $$
    \mathbb{R} \oplus V \simeq C^\infty(\mathbb{R}^n)/(x_1, \cdots x_n)^{k+1}
  $$
\end{proposition}
\begin{definition}
  \label{FormalSmoothSets}
  Write
  $$
    \mathrm{FormalCartSp} \hookrightarrow \mathrm{CAlg}_{\mathbb{R}}^{\mathrm{op}}
  $$
	for the full subcategory of the opposite of commutative
	$\mathbb{R}$-algebras which are Cartesian products $\mathbb{R}^n
	\times \mathbb{D}$ of a Cartesian space (via the embedding of
	proposition~\ref{EmbeddingOfManifoldsInRAlgebras}) with an
	infinitesimally thickened point (via the defining embedding of
	definition~\ref{InfinitesimallyThickenedPoints}).

	Explicitly, since the coproduct in $\mathrm{CAlg}_{\mathbb{R}}$ is the
	tensor product over $\mathbb{R}$, this is the full subcategory on
	those commutative $\mathbb{R}$-algebras of the form
  $$
    C^\infty(\mathbb{R}^n \times \mathbb{D})
      :=
    C^\infty(\mathbb{R}^n) \otimes_{\mathbb{R}} (\mathbb{R}\oplus V)
    \,,
  $$
  for $n \in \mathbb{N}$ and where $V$ is a finite-dimensional nilpotent ideal.

	We regard this as a site by taking the covering families those sets of
	morphisms
  $$
    \left\{
      \xymatrix{
      U_i \times \mathbb{D}
        \ar[rr]^{\phi_i \times \mathrm{id}_{\mathbb{D}}}
        &&
      \mathbb{R}^n \times \mathbb{D}
      }
    \right\}_{i \in I}
  \quad \text{such that} \quad
    \left\{
      U_i
        \overset{\phi_i }{\longrightarrow}
      \mathbb{R}^n
    \right\}_{i \in I}
  $$
  is a covering family in $\mathrm{CartSp}$ (i.e.,\ a good open cover).

  We write
  $$
    \mathrm{FormalSmoothSet} := \mathrm{Sh}(\mathrm{FormalSmoothCartSp})
  $$
  for the category of sheaves over this site (``formal smooth sets'').
\end{definition}
The category $\mathrm{FormalSmoothSet}$ in definition~\ref{FormalSmoothSets} was
introduced in \cite{Dubuc79} as a well-adapted model of the
Kock-Lawvere axioms \cite[I.12]{KockBookA} \cite[1.3]{KockBook} for
synthetic differential geometry. It is sometimes referred to as the
``Cahiers topos.'' While we do discuss a kind of ``synthetic''
axiomatization of differential geometry in $\mathrm{FormalSmoothSet}$ in section \ref{JetsPDEs}, we do not use the Kock-Lawvere axioms, but instead the
fact that $\mathrm{FormalSmoothSet}$ satisfies the axioms of ``differential cohesion''
\cite{dcct}, a property recalled as proposition~\ref{diffcohesionofFormalSmoothSet} below.

\begin{example}
  \label{VectorFieldsAsMapsFromSpecOfRingOfDualNumbers}
  \label{JetsOfFunctionsBetweenCartesianSpaces}
  By Hadamard's lemma (proposition~\ref{HadamardLemma}), morphisms in $\mathrm{FormalCartSp}$ (definition~\ref{FormalSmoothSets})
  out of $\mathbb{D}^1(1)$ into any Cartesian space
  $$
    \mathbb{D}^1(1) \longrightarrow \mathbb{R}^n
  $$
  are in bijection with tangent vectors
  $$
    \sum_{i = 1}^n v^i \frac{\partial}{\partial x^i} \in T_x \mathbb{R}^n
    \,.
  $$
  The corresponding algebra homomorphism
  $$
    C^\infty(\mathbb{R}^n) \longrightarrow (\mathbb{R} \oplus \varepsilon \mathbb{R})
  $$
  is given by
  $$
    f \mapsto f(x) + \varepsilon \sum_{i = 1}^n v^i \frac{\partial f}{\partial x^i}(x)
    \,.
  $$

  More generally, 
  by iterated application of Hadamamard's lemma,
  a morphism
  $$
    \mathbb{D}^n(k) \longrightarrow \mathbb{R}^n
  $$
  out of the order-$k$ infinitesimal $n$-disk (definition~\ref{StandardInfinitesimalnDisk})
  is equivalently the equivalence class of a smooth function
  $$
    \mathbb{R}^n \longrightarrow \mathbb{R}^n
    \,,
  $$
  where two such functions are regarded as equivalent if they have the same
  partial derivatives at 0 up to order $k$. Such an equivalence class is called a ``$k$-jet'' of a smooth function.

  The corresponding algebra homomorphism
  $$
    C^\infty(\mathbb{R}^n) \longrightarrow C^\infty(\mathbb{R}^n)/(x^1, \cdots, x^n)^{k+1}
  $$
	is given by restricting smooth functions to their partial derivatives
	up to order $k$ around a given point. 
\end{example}
\begin{proposition}
  \label{InfinitesimalPointInSmoothManifoldFactorsThroughInfinitesimalDisk}  
  Every map from an infinitesimally thickened point $\mathbb{D}$ (definition \ref{InfinitesimallyThickenedPoints}) to a $d$-manifold $X$ factors through the $d$-dimensional infinitesimal disk $\inlinetikzcd{\mathbb{D}^d(k) \ar[r, hook] \& X}$ (definition~\ref{StandardInfinitesimalnDisk}) whose order $k$ of infinitesimality is that of $\mathbb{D}$:
  $$
    \begin{tikzcd}
      \mathbb{D}
      \ar[
        r,
        "{ f }"
      ]
      \ar[
        d,
        dashed
      ]
      &
      X
      \\
      \mathbb{D}^d(k)
      \mathrlap{\,.}
      \ar[
        ur,
        hook',
        dashed
      ]
    \end{tikzcd}
  $$
\end{proposition}
\begin{proof}
  Since every such map $f$ clearly factors through a chart $\inlinetikzcd{\mathbb{R}^d \ar[r, hook] \& X}$ (for instance: a tubular neighbourhood of the image of the global point of $\mathbb{D}$), it is sufficient to prove the case $X = \mathbb{R}^d$, which is what we proceed to do.
  
  Now first to recall that, by definition, we have
  \begin{equation}
    \label{WeilAlgebraOfAndInfntsmlDsk}
    C^\infty(\mathbb{D})
    \simeq_{{}_{\mathbb{R}}}
    \mathbb{R} \oplus V
    \;\;\;\mbox{with}\;\;\;
    V^{k+1} = 0
    \mathrlap{\,,}
  \end{equation}
  for some $k$ which is the order of nilpotency of the ideal $V$.
  Moreover, without restriction of generality we may assume that $f(0) = 0$. This means that we have vanishing pullback $i^\ast f^\ast$ of the ideal $(x^1, \cdots, x^d)$ of smooth functions on $\mathbb{R}^d$ that vanish at the origin:
  $$
    \begin{tikzcd}[
      row sep=0pt
    ]
      \ast
      \ar[r, hook, "{ i }"]
      &
      \mathbb{D}
      \ar[r, "{ f }"]
      &
      \mathbb{R}^d
      \\
      \mathbb{R}
      \ar[
        r,
        <-,
        "{ i^\ast }"
      ]
      &
      \mathbb{R} \oplus V
      \ar[
        r,
        <-,
        "{ f^\ast }"
      ]
      &
      C^\infty(\mathbb{R}^d)
      \\
      0
      \ar[
        rr,
        <-|,
        shorten=5pt,
        "{ i^\ast f^\ast }"
      ]
      &&
      (x^1, \cdots, x^d)
      \mathrlap{\,,}
    \end{tikzcd}
  $$
  hence that
  \begin{equation}
    \label{PullbackOfx1...xdToD}
    f^\ast(x^1, \cdots, x^d)
    \subset 
    V
    \mathrlap{\,,}
  \end{equation}
  where $(x^1,\cdots, x^d)$ denotes the ideal generated by the coordinates.
  
  Now, given $g \in C^\infty(\mathbb{R}^d)$, Hadamard's lemma implies the existence of
  $$
    \left(
      h_\alpha 
        \in 
      C^\infty(\mathbb{R}^d)
    \right)_{
      \substack{
        \alpha \in \mathbb{N}^d,
        \\
        \vert \alpha \vert = k+1
      }
    }
  $$
  such that
  $$
    g(x)
    =
    \sum_{\vert \alpha \vert \leq k}
    \tfrac{1}{\alpha!}
    x^\alpha
    \cdot
    \partial_\alpha g(0)
    +
    \sum_{\vert \alpha \vert = k+1}
    x^\alpha \cdot h_\alpha(x)
  $$
  (where $\vert \alpha \vert := \alpha_1+ \cdots + \alpha_d$ and $\alpha! := \alpha_1 ! \cdot \cdots \cdot \alpha_d !$), whence \eqref{PullbackOfx1...xdToD} and the order of nilpotency \eqref{WeilAlgebraOfAndInfntsmlDsk} imply that
$$
  (f^\ast g)(x)
  =
  \sum_{\vert \alpha \vert \leq k}
  \tfrac{1}{\alpha!}
  f^\ast(x^\alpha) \cdot \partial_\alpha g (0)
  \mathrlap{\,,}
$$
hence that the algebra homomorphism $f^\ast$ factors as:
$$
  \begin{tikzcd}
    \mathbb{R} \oplus V
    \ar[
      r,
      <-
    ]
    \ar[
      rr,
      uphordown,
      <-,
      "{
        f^\ast
      }"{description}
    ]
    &
    C^\infty(\mathbb{R}^d)%
    /(x^1, \cdots, x^d)^{k+1}
    \ar[
      r,
      <<-
    ]
    &
    C^\infty(\mathbb{R}^d)
    \mathrlap{\,.}
  \end{tikzcd}
$$
This is the formal dual of the factorization to be shown.
\end{proof}

\begin{proposition}[differential cohesion of formal smooth sets]
  \label{diffcohesionofFormalSmoothSet}
	The canonical embedding of Cartesian spaces into formal Cartesian
	spaces (definition~\ref{FormalSmoothSets}) is coreflective, i.e.,\ the
	embedding functor $i$ has a right adjoint functor $p$
  $$
    \xymatrix{
      \mathrm{CartSp}
      \ar@{^{(}->}@<+5pt>[rr]^-i
      \ar@{<-}@<-5pt>[rr]_-p^-{\bot}
      &&
      \mathrm{FormalCartSp}
    }
    \,.
  $$
	Moreover, Kan extension (proposition \ref{FactsAboutLeftKanExtension}) along these functors induces a quadruple of
	adjoint functors (definition \ref{AdjointFunctor}) of the form
  $$
    \xymatrix{
      \mathrm{SmoothSet}\;
      \ar@<+18pt>@{^{(}->}[rr]|-{i_!}
      \ar@<+9pt>@{<-}[rr]|-{i^\ast \simeq p_!}
      \ar@<+0pt>@{^{(}->}[rr]|-{i_\ast \simeq p^\ast}
      \ar@<-9pt>@{<-}[rr]|-{p_\ast}
      &&
      \;\mathrm{FormalSmoothSet}
    }
    \,.
  $$
  (where each functor on top is left adjoint to the functor below)
  between smooth sets (definition \ref{SmoothManifolds}) and formal smooth sets (definition \ref{FormalSmoothSets}).
\end{proposition}
\begin{proof}
  The first statement is the restricted formal dual of 
  the fact that reduced $\mathbb{R}$-algebras are 
  reflective inside all $\mathbb{R}$-algebras.
  We make this explicit:
  Observe that in $\mathrm{FormalCartSp}$ every morphism out of a finite
  manifold into an infinitesimally thickened point
  $$
    \mathbb{R}^n \longrightarrow \mathbb{D}
  $$
  necessarily factors through the actual point
  $$
    \mathbb{R}^n \overset{\exists !}{\longrightarrow} \ast \overset{\exists!}{\longrightarrow} \mathbb{D}
    \,.
  $$
  This is because, dually, an algebra homomorphism of the form
  $$
    C^\infty(\mathbb{R}^n ) \longleftarrow (\mathbb{R} \oplus V)
  $$
  has to vanish on the nilpotent ideal $V$. (This is because respect for the product implies that nilpotent
  algebra elements are sent to nilpotent algebra elements, but there are no non-zero nilpotent elements in
  $C^\infty(\mathbb{R}^n)$.)
  
  This means that there is a natural bijection between morphisms in $\mathrm{FormalCartSp}$ of the form
  $$
    \mathbb{R}^{n_1} \longrightarrow \mathbb{R}^{n_2} \times \mathbb{D}
  $$
  and morphisms in $\mathrm{CartSp} \hookrightarrow \mathrm{FormalCartSp}$ of the form
  $$
    \mathbb{R}^{n_1} \longrightarrow \mathbb{R}^{n_2}
    \,.
  $$
  This in turn means that if a right adjoint $p$ to the inclusion exists then it is given by projection on
  the non-thickened part (the reduced part)
  $$
    p(\mathbb{R}^n \times \mathbb{D}) \simeq \mathbb{R}^n
    \,.
  $$
  It only remains to observe that this is indeed functorial 
  %
  and the first claim has been established.

  With this, the adjoint quadruple between categories of sheaves follows directly from the
  corresponding system of adjunctions on presheaves (proposition \ref{FactsAboutLeftKanExtension})
  because the site $\mathrm{FormalCartSp}$ is by definition such that the Grothendieck
  topology along the infinitesimal directions is trivial.
\end{proof}
The synthetic theory that we develop below will be based entirely on the existence of an adjoint quadruple as in
proposition \ref{diffcohesionofFormalSmoothSet}. Therefore, while Dubuc's \emph{Cahiers topos} of formal smooth sets
is the archetypical example of an ambient category that we consider, we consider this structure in more generality:
\begin{definition}[differential cohesion {\cite[\S 3.5]{dcct}\cite[\S 3.1.2]{SS20-Orb}}]
  \label{ReImEtAdjunctionOnFormalSmoothSets}
  \label{DifferentialCohesion}
  We say that a full inclusion of toposes
  $$
    i_!
    :
    \mathbf{H}_{\Re}
      \hookrightarrow
    \mathbf{H}
  $$
  exhibits $\mathbf{H}$ as being \emph{differentially cohesive} over $\mathbf{H}_{\Re}$
  (and exhibits $\mathbf{H}_\Re$ as being the full subcategory of \emph{reduced objects} of $\mathbf{H}$)
  if the inclusion has three adjoint functors to the right
  $$
    \xymatrix{
      \mathbf{H}_{\Re}
        \ar@<+18pt>@{^{(}->}[rr]|-{i_!}
        \ar@<+9pt>@{<-}[rr]|-{i^\ast}
        \ar@<+0pt>@{^{(}->}[rr]|-{i_\ast}
        \ar@<-9pt>@{<-}[rr]|-{i_!}
        &&
      \mathbf{H}
    }
    \,.
  $$
  We then write
  $$
    (\Re \,\dashv\, \Im \,\dashv\, \mathrm{E}\!\mathrm{t})
    \;:=\;
    (i_!\circ i^\ast
    \,\dashv\, i_\ast \circ i^\ast
    \,\dashv\, i_\ast \circ i_!)
    \;:\;
    \mathbf{H}
      \longrightarrow
    \mathbf{H}
  $$
	for the triple of adjoint idempotent (co-)monads on formal smooth set induced
    (via example \ref{AdjointPairFromAdjointTriple}) by this adjoint quadruple.
We pronounce $\Re$ ``reduction'' (due to proposition \ref{ReductionOnFormalManifolds} below)
and $\Im$ ``de Rham shape'' \cite{Simpson96} or ``infinitesimal shape'' \cite{dcct}.
\end{definition}
So in this terminology proposition \ref{diffcohesionofFormalSmoothSet}
says that the topos $\mathbf{H} = \mathrm{FormalSmoothSet}$ is
differentially cohesive over the full subtopos $\mathbf{H}_{\Re} =
\mathrm{SmoothSet}$.
\begin{proposition}
  \label{ReductionOnFormalManifolds}
	For $\mathbb{R}^n \times \mathbb{D}\in
	\mathrm{FormalCartSp}\hookrightarrow \mathrm{FormalSmoothSet}$ we have
  $$
    \Re(\mathbb{R}^n \times \mathbb{D}) \simeq \mathbb{R}^n
  $$
	and the counit (definition \ref{monad}) of the comonad $\Re$
	(definition~\ref{ReImEtAdjunctionOnFormalSmoothSets}) is the canonical
	inclusion
  $$
    \epsilon_{\mathbb{R}^n\times \mathbb{D}} : \Re(\mathbb{R}^n\times \mathbb{D})
			\simeq \mathbb{R}^n
			\stackrel{(\mathrm{id},0)}{\longrightarrow} \mathbb{R}^n \times \mathbb{D}
    \,.
  $$
	It follows that for $X \in \mathbf{H}=
	\mathrm{Sh}(\mathrm{FormalCartSp})$ a sheaf, we have that $\Im X \in
	\mathrm{Sh}(\mathrm{FormalCartSp})$ is the sheaf given by
  $$
    \Im X
      \;:\;
    \mathbb{R}^n \times \mathbb{D}
      \mapsto
    X(\mathbb{R}^n)
    \,,
  $$
  and that the unit (definition \ref{monad})
  $$
    \eta_X \;:\; X \longrightarrow \Im X
  $$
	of the monad $\Im$
	(definition~\ref{ReImEtAdjunctionOnFormalSmoothSets}) is the morphism
	of sheaves which over any $\mathbb{R}^n \times \mathbb{D} \in
	\mathrm{FormalCartSp}$ is given by the function
  $$
    \mathrm{Hom}_{\mathbf{H}}(\mathbb{R}^n \times \mathbb{D}, X)
      \longrightarrow
    \mathrm{Hom}_{\mathbf{H}}(\mathbb{R}^n \times \mathbb{D}, \Im X)
			\simeq
		\mathrm{Hom}_{\mathbf{H}}(\Re(\mathbb{R}^n\times \mathbb{D}), X)
		\,,
  $$
  given by precomposition with $\epsilon$,
  $$
    (\mathbb{R}^n \times \mathbb{D} \stackrel{\phi}{\longrightarrow} X)
     \;\mapsto\;
    \left(
    	\Re(\mathbb{R}^n\times \mathbb{D})
    	\overset{\epsilon}{\longrightarrow}
			\mathbb{R}^n \times \mathbb{D}
			\stackrel{\phi}{\longrightarrow}
			X
		\right)
    \,.
  $$
\end{proposition}
\begin{proof}
	Using that left Kan extension $i_!$ along a functor $i$ is on
	representables given by that functor
	(proposition~\ref{FactsAboutLeftKanExtension}), we have, for all
	$\mathbb{R}^{n_1}\times \mathbb{D} \in \mathrm{FormalCartSp}$ and all
	$\mathbb{R}^{n_2}\in \mathrm{CartSp} \hookrightarrow
	\mathrm{FormalCartSp}$, the natural isomorphisms
  $$
    \begin{aligned}
      \mathrm{Hom}_{\mathbf{H}_\Re}(\mathbb{R}^{n_2}, i^\ast ( \mathbb{R}^{n_1}\times \mathbb{D} ))
      & \simeq
      \mathrm{Hom}_{\mathbf{H}}(i_!(\mathbb{R}^{n_2}), \mathbb{R}^{n_1} \times \mathbb{D})
      \\
      & \simeq \mathrm{Hom}_{\mathrm{FormalCartSp}}(\mathbb{R}^{n_2},\mathbb{R}^{n_1}\times \mathbb{D})
      \\
      & \simeq \mathrm{Hom}_{\mathrm{FormalCartSp}}(\mathbb{R}^{n_2}, \mathbb{R}^{n_1})
      \, ,
    \end{aligned}
  $$
	where in the last line we used proposition
	\ref{diffcohesionofFormalSmoothSet}. This shows that
  $$
    i^\ast ( \mathbb{R}^{n_1}\times \mathbb{D})
    \simeq
    \mathbb{R}^{n_1}
    \;\;\;\;
    \in \mathbf{H}_\Re
    \,.
  $$
	Use again that left Kan extension $i_!$ preserves representables
	to find that
  $$
		\Re(\mathbb{R}^{n_1}\times \mathbb{D})
			:=
		i_! i^*(\mathbb{R}^{n_1}\times \mathbb{D})
			\simeq
		i_! (\mathbb{R}^{n_1}) \simeq \mathbb{R}^{n_1}
			\,.
  $$

	Now, the counit $\epsilon=\epsilon_{\mathbb{R}^{n_2}\times \mathbb{D}}
	: \Re(\mathbb{R}^{n_2}\times \mathbb{D}) \to \mathbb{R}^{n_2}\times
	\mathbb{D}$ is characterized by the identity $g = \epsilon_Y
	\circ i_!(f)$ with respect to an arbitrary morphism $f\colon X \to
	i^*(Y)$ and its adjoint $g\colon i_!(X) \to Y$. Setting $Y =
	\mathbb{R}^{n_2} \times \mathbb{D}$ and $X = i^*(\mathbb{R}^{n_1}
	\times \mathbb{D}')$, and recalling also that $i_!(f) = f$, this means
	that the formula $g = \epsilon \circ f$ must give a bijection between
	morphisms of the form
	$$
		\mathbb{R}^{n_1}
			\simeq
		i^*(\mathbb{R}^{n_1} \times \mathbb{D}')
			\overset{f}{\longrightarrow}
		i^*(\mathbb{R}^{n_2} \times \mathbb{D})
			\simeq
		\mathbb{R}^{n_2}
		\quad \text{and} \quad
		\mathbb{R}^{n_1}
			\simeq
		i_! i^*(\mathbb{R}^{n_1} \times \mathbb{D}')
			\simeq
		\Re(\mathbb{R}^{n_1} \times \mathbb{D}')
			\overset{g}{\longrightarrow}
		\mathbb{R}^{n_2} \times \mathbb{D}
		\,.
	$$
	From the proof of proposition~\ref{diffcohesionofFormalSmoothSet}, we
	know that $g$ uniquely factors as
	$$
		g\colon \mathbb{R}^{n_1}
			\overset{g'}{\longrightarrow}
		\mathbb{R}^{n_2}
			\overset{(\mathrm{id},0)}{\longrightarrow}
		\mathbb{R}^{n_2} \times \mathbb{D}
		\,.
	$$
	Hence, under the hypothesis that $g' = f$, we conclude that we can
	identify $\epsilon = (\mathrm{id},0)$. It remains only to verify this
	hypothesis. First, note that when $\mathbb{D} = \{*\}$, we simply have
	$g \simeq f$ (the adjunction is unique, if it exists, and this
	satisfies all the desired properties). Second, recalling the
	naturality of the adjunction, we find that, with $p_2 \colon
	\mathbb{R}^{n_2} \times \mathbb{D} \to \mathbb{R}^{n_2}$ the
	projection onto the second factor, the adjunct of $i^*(p_2) \circ f =
	f$ is $p_2 \circ g = g'$, from which we can conclude that $g'=f$.

	Finally, since $\Re \dashv \Im$ and the unit $\eta_X$ is correspondingly
	adjunct to the counit $\epsilon_X$, it is straightforward to see that
	the sheaf morphism representing $\eta_X$ is given by the formulas in
	the statement of the proposition.
\end{proof}
\begin{example}
  \label{ReductionOnRepresentablesIsRetraction}
  The reduction counit on $\mathbb{R}^n \times \mathbb{D}$
  (proposition \ref{ReductionOnFormalManifolds}) has in fact a retraction:
  $$
    \xymatrix{
      \mathbb{R}^n
        \ar[rr]^{(\mathrm{id}, 0)}
        \ar@/_1pc/[rrrr]_{\mathrm{id}}
      &&
      \mathbb{R}^n \times \mathbb{D}
        \ar[rr]
        &&
      \mathbb{R}^n
    }
  $$
  given by the product with $\mathbb{R}^n$ applied to the basic retraction
  $$
    \xymatrix{
      \ast
        \ar[rr]^{\exists !}
        \ar@/_1pc/[rrrr]_{\mathrm{id}}
      &&
      \mathbb{D}
        \ar[rr]^{\exists !}
        &&
      \ast
    }
    \,.
  $$
  Dually this is the retraction of commutative $\mathbb{R}$-algebras
  $$
    \xymatrix{
      C^\infty(\mathbb{R}^n)
        \ar@{<-}[rr]^-{\mathrm{reduction}}
        \ar@{<-}@/_1.4pc/[rrrr]_{\mathrm{id}}
      &&
      C^\infty(\mathbb{R}^n) \otimes_{\mathbb{R}} C^\infty(\mathbb{D})
        \ar@{<-}[rr]^-{f \mapsto (f,0)}
        &&
      C^\infty(\mathbb{R}^n)
    }
    \,,
  $$
  where the left morphism ``$\mathrm{reduction}$'' sends every nilpotent algebra element to 0 and is the identity on the remaining elements.

  From this the statements about $\Im$ follow by the adjunction $(\Re \dashv \Im)$.
\end{example}
\begin{proposition}
  \label{UnitOfInfinitesimalShapeIsEpiInCahiersTopos}
  For every $X \in \mathrm{FormalSmoothSet}$ (definition \ref{FormalSmoothSets}),
  the unit $\eta_X : X \longrightarrow \Im X$ of the infinitesimal shape monad
  from definition \ref{ReImEtAdjunctionOnFormalSmoothSets}
  is an epimorphism in $\mathrm{FormalSmoothSet}$, i.e.,\ all objects of $\mathrm{FormalSmoothSet}$
  are formally smooth in the sense of Definition \ref{FormallyEtaleMorphism}.
\end{proposition}
\begin{proof}
  By proposition \ref{EpimorphismsInCategoriesOfSheaves} it is sufficient to see that for
  $\mathbb{R}^n \times \mathbb{D} \in \mathrm{FormalCartSp} \hookrightarrow \mathbf{H}$
  any representable,
  the induced function on hom-sets
  $$
    \mathrm{Hom}_{\mathbf{H}}( \mathbb{R}^n \times \mathbb{D} , X  )
      \longrightarrow
    \mathrm{Hom}_{\mathbf{H}}( \mathbb{R}^n \times \mathbb{D}, \Im X )
  $$
  is a surjection. By proposition \ref{ReductionOnFormalManifolds}
  this is the case precisely if every morphism in $\mathbf{H}$ of the form
  $$
    \mathbb{R}^n \longrightarrow X
  $$
  factors as
  $$
    \mathbb{R}^n
       \overset{(\mathrm{id},0)}{\longrightarrow}
    \mathbb{R}^n \times \mathbb{D}
      \longrightarrow
    X
    \,.
  $$
  This follows from example \ref{ReductionOnRepresentablesIsRetraction}.
\end{proof}

  \subsection{Locally pro-manifolds}
  \label{FrechetManifolds}

We consider here a category of ``locally pro-manifolds''
(definition \ref{LocallyProManifold} below), by
which we mean projective limits of finite-dimensional smooth
manifolds, formed inside the category of Fr{\'e}chet manifolds.

Fr{\'e}chet manifolds are smooth manifolds of possibly
infinite-dimension, which are however only ``mildly
infinite-dimensional'' in that they arise as projective limits of Banach
manifolds. They still embed fully faithfully into the category
$\mathrm{SmoothSet}$ of smooth sets
(proposition~\ref{SmoothSetsReceiveFrechetManifolds} below), hence also
into the category $\mathrm{FormalSmoothSet}$ of formal smooth sets from the previous
section. The key point is that projective limits of finite dimensional
smooth manifolds do exist in the category of Fr{\'e}chet manifolds
(proposition~\ref{ProjectiveLimitNatureOfRInfinity} below). Accordingly,
infinite jet bundles of smooth manifolds
(definition~\ref{InfiniteJetBundleOfManifolds} below) may naturally be
regarded as Fr{\'e}chet manifolds (a point of view advocated in
\cite{Saunders89,Michor80,GuPf13,DoGaVa}), in fact they are in the more restrictive
class of $\mathbb{R}^\infty$-manifolds (definition~\ref{RInfinity}
below).

Beware that sometimes jet bundles are instead regarded as pro-objects in
the category of finite dimensional smooth manifolds, which is
essentially what is done in~\cite[Sec.4.1]{VinKras} and~\cite[appendix
to Chapter 2]{DeligneFreed},
hence as purely formal projective limits. This perspective is not
equivalent: While a smooth function on a pro-manifold is, by definition,
globally a function on one of the finite stages in the projective
system, a smooth function on the corresponding projective limit formed
in Fr{\'e}chet manifolds is in general only \emph{locally} a function on
one of the finite stages
(proposition~\ref{SmoothFunctionsOnFrechetRInfinity} below).
This is why we speak of ``locally pro-manifolds''.
As mentioned earlier, our approach largely agrees with~\cite{Saunders89,Michor80,GuPf13,DoGaVa}, even if
we differ slightly in terminology.%
	\footnote{See~\cite{MO401651} for an informal discussion of these two
	ways of looking at jet bundles and where they appear in the
	literature.} %
Notice that while the smooth functions are not the same in both cases, the cohomology of the
complexes of the corresponding differential forms may still agree, see
\cite{GMS-cohom} or \cite[appendix to Chapter 2]{DeligneFreed}.

\medskip

The calculus of smooth functions on Fr{\'e}chet manifolds  is a
special case of a calculus that may be defined more generally on manifolds locally
modeled on locally convex topological vector spaces (see~\cite{Michor80}
for details, in particular Chapter~8 for the precise definition of
smoothness).

\begin{definition}[{\cite[\textsection8.1--3]{Michor80}}, {\cite[Definition~7.1.5]{Saunders89}}]
	\label{FrechetSmoothFunction}
	Let $f\colon X \to Y$ be a continuous map (not necessarily linear)
	between two Fr\'echet spaces $X$ and $Y$. The \emph{G\^ateaux
	derivative} $Df \colon U \times X \to Y$ of $f$ on an open $U\subset
	X$ is defined as the directional derivative
	\[
		Df(x,w) := D_w f(x) = \lim_{t\to 0} \frac{f(x+tw) - f(x)}{t} \, ,
	\]
	with the limit taken in the topology of $Y$. The map $f$ is smooth at
	$x\in X$ when there exists an open neighborhood $U \ni x$ such that
	the iterated G\^ateaux derivatives $(u,v_1,\ldots, v_k) \mapsto
	D^k_{v_1,\ldots, v_k} f(u)$, defined in the obvious way, exist and are
	jointly continuous
	\[
		D^kf \colon U \times \underbrace{X \times \cdots \times X}_{k\text{ times}}
			\to Y
	\]
	for each $k\ge 0$.
\end{definition}

This notion of smoothness is at the very least compatible with the
fundamental theorem of calculus.

\begin{proposition}[{\cite[\textsection 8.4]{Michor80}}]
	\label{FrechetFundamentalTheorem}
	Let $f\colon X\to Y$ be a map between Fr\'echet spaces, with $f$
	smooth on an open convex subset $U\subset X$. Then, for any $u,v \in
	U$,
	\[
		f(v) - f(u) = \int_0^1 \d{t} \, D_{v-u}f(u+t(v-u)) ,
	\]
	using the Bochner integral (Riemann sums convergent in the topology of
	$Y$).
\end{proposition}
Usually the Bochner integral is defined on a measure space and is valued
in a Banach space. Since we are only considering integration over the
interval $[0,1]$ and only of continuous functions, there is no
difference between the Riemann and Lebesgue integrals. On a Banach space
the Bochner integral is defined by convergence with respect to a norm,
which is equivalent to convergence in the topology of the Banach space
(as opposed to a weak topology, for example). Though Fr\'echet spaces
are not normed, we can still define the Bochner integral by convergence
in the topology of the Fr\'echet space.

\begin{definition}
 \label{FrechetManifoldsDef}
Let $\mathrm{FrMfd}$ be the category of topological spaces endowed with
atlases of charts onto open subsets of Fr\'echet vector spaces, with
continuous maps that are smooth in local charts (in the sense of def. \ref{FrechetSmoothFunction}) as
morphisms.
\end{definition}
There is the canonical inclusion
$$
  \mathrm{SmthMfd} \hookrightarrow \mathrm{FrMfd}
$$
of finite dimensional smooth manifolds
(definition~\ref{SmoothManifolds}) into Fr{\'e}chet manifolds
(definition~\ref{FrechetManifoldsDef}).

\begin{proposition}
  \label{SmoothSetsReceiveFrechetManifolds}
  Let $i : \mathrm{SmthMfd} \hookrightarrow \mathrm{FrMfd}$
  be the canonical inclusion of finite-dimensional smooth manifolds
  (definition~\ref{SmoothManifolds}) into Fr{\'e}chet manifolds
  (definition~\ref{FrechetManifoldsDef}).
  Then the functor
  \begin{gather*}
    i_{\mathrm{Sh}}
    \;:\;
    \xymatrix{
    \mathrm{FrMfd}
      \ar[r]^-{y_{\mathrm{FrMfd}}}
      &
    \mathrm{PSh}(\mathrm{FrMfd})
      \ar[r]^-{i^\ast}
      &
    \mathrm{PSh}(\mathrm{SmthMfd})
      \ar[rr]^-{L_{\mathrm{SmthMfd}}}
      &&
    \mathrm{Sh}(\mathrm{SmthMfd})
    =
    \mathrm{SmoothSet}
    }
    \, , \\
    i_{\mathrm{Sh}}(X) : U \mapsto \mathrm{Hom}_{\mathrm{FrMfd}}(i(U),X)
    \, ,
  \end{gather*}
	where $y_{(-)} \colon (-) \to \mathrm{PSh}(-)$ is the Yoneda embedding
	and $L_{(-)} \colon \mathrm{PSh}(-) \to \mathrm{Sh}(-)$ is the
	sheafification functor, is fully faithful, hence exhibits Fr{\' e}chet
	manifolds as a full subcategory of smooth sets
	(definition~\ref{SmoothManifolds}).
\end{proposition}
\begin{proof}
	The functor factors evidently through the full subcategory of concrete
	sheaves on smooth manifolds, which is the category of diffeological
	spaces from example \ref{DiffeologicalSpace}. It is hence sufficient
	to see that the factorization
  $$
    \mathrm{FrMfd} \longrightarrow \mathrm{DiffSp}
  $$
  is fully faithful. This is the content of \cite[theorem 3.1.1]{Losik94}.
\end{proof}
\begin{proposition}
	There exists a subcanonical (definition \ref{subcanonical})
	Grothendieck topology on $\mathrm{FrMfd}$ such that its sheaf topos
	agrees with that over $\mathrm{SmthMfd}$:
  $$
    \mathrm{Sh}(\mathrm{FrMfd}) \simeq \mathrm{Sh}(\mathrm{SmthMfd})
    \,.
  $$
\end{proposition}
\begin{proof}
	This is the result of applying
	proposition~\ref{SmoothSetsReceiveFrechetManifolds} in
	proposition~\ref{AmbientCategoryGrothedieckTopology}.
\end{proof}

\begin{definition}
  \label{RInfinity}
  Write
  $$
    \mathbb{R}^\infty \in \mathrm{FrMfd}
  $$
	for the Fr{\' e}chet manifold whose underlying topological vector
	space is the \emph{projective} limit over the sequence of the standard
	projections $\cdots \to \mathbb{R}^2 \to \mathbb{R}^1 \to
	\mathbb{R}^0$, equipped with the family of seminorms given for each $n
	\in \mathbb{N}$ by the composites
  $$
    \Vert - \Vert_n
      \;:\;
    \mathbb{R}^\infty
      \overset{p_n}{\longrightarrow}
    \mathbb{R}^n
      \overset{\Vert -\Vert}{\longrightarrow}
    \mathbb{R}
  $$
	of the standard projection to $\mathbb{R}^n$ followed by the standard
	norm on $\mathbb{R}^n$.
\end{definition}
\begin{remark}
  \label{RInfinityAsTopologicalProjectiveLimit}
	The definition \ref{RInfinity} of $\mathbb{R}^\infty$ says that for
	each $x \in \mathbb{R}^n$ the open balls
  $$
    B_\delta^n(x)
    :=
     \left\{
      y \in \mathbb{R^\infty}
         \;\vert\;
      \Vert y-x\Vert_n < \delta
     \right\}
     \;\;\;\;\;\;\;
     \mbox{for $\delta > 0, n \in \mathbb{N}$}
  $$
	form a base of neighborhoods of $x$ for the topology on
	$\mathbb{R}^\infty$. For fixed $n$ these are of course the preimages
	under $p_n : \mathbb{R}^\infty \to \mathbb{R}^n$ of the opens of the
	standard base of neighborhoods of $\mathbb{R}^n$, hence the open balls
	$B_\delta^n(x)$ induce the coarsest topology on $\mathbb{R}^\infty$
	such that all the $p_n$ are continuous functions. This exhibits the
	underlying topological space of $\mathbb{R}^\infty$ as the projective
	limit in topological spaces over the system of topological spaces
	$(\cdots \to \mathbb{R}^2 \to \mathbb{R}^1 \to \mathbb{R}^0)$.
\end{remark}
\begin{proposition}
  \label{ProjectiveLimitNatureOfRInfinity}
	The Fr{\'e}chet manifold $\mathbb{R}^\infty$ from
	definition~\ref{RInfinity} is the projective limit, formed in the
	category $\mathrm{FrMfd}$ of Fr{\' e}chet manifolds, of the projective
	system $(\cdots \to \mathbb{R}^2 \to \mathbb{R}^1 \to \mathbb{R}^0)$
	of finite dimensional manifolds, regarded as a system in
	$\mathrm{FrMfd}$.
\end{proposition}
\begin{proof}
	We need to check that for any $X \in \mathrm{FrMfd}$ and a sequence
	$\{ X \overset{f_n}{\to} \mathbb{R}^n \}$ of compatible morphisms, there
	is a unique morphism $f\colon X \to \mathbb{R}^\oo$ that commutes with
	the $f_n$ and the projections $X\to \mathbb{R}^n$. By remark
	\ref{RInfinityAsTopologicalProjectiveLimit} we already know that the
	underlying topological space of $\mathbb{R}^\infty$ is the projective
	limit of the projective sequence of topological spaces $(\cdots
	\to \mathbb{R}^2 \to \mathbb{R}^1 \to \mathbb{R}^0)$. So all of our
	conditions are already satisfied, though at this point we can only
	conclude that $f$ is  continuous. Hence what remains to be shown is
	that our $f$ is smooth precisely when all the corresponding $f_n$ are.
	This is easily done by repeating the above argument for the iterated
	derivatives $D^k f_n$. The details can be found in~\cite[Lemma
	7.1.8]{Saunders89}.
\end{proof}
\begin{example}
	It follows from proposition~\ref{ProjectiveLimitNatureOfRInfinity}
	that for $y(\mathbb{R})^n  \simeq  i_{\mathrm{Sh}}(\mathbb{R}^n) \in
	\mathrm{Sh}(\mathrm{SmthMfd})$ the Cartesian spaces regarded as smooth
	sets (via Yoneda embedding $y(-)$, hence equivalently via the embedding of
	proposition~\ref{SmoothSetsReceiveFrechetManifolds}) then their
	projective limit as smooth sets is represented (via
	proposition~\ref{SmoothSetsReceiveFrechetManifolds}) by their
	projective limit as Fr{\' e}chet manifolds
  $$
    \underset{\longleftarrow}{\lim}_n i_{\mathrm{Sh}}(\mathbb{R}^n)
    \simeq
    i_{\mathrm{Sh}}(\mathbb{R}^\infty)
    \,.
  $$
	Moreover, the full faithfulness of the embedding due to
	proposition~\ref{SmoothSetsReceiveFrechetManifolds} says that
	morphisms of smooth sets of the form
  $$
    \underset{\longleftarrow}{\lim}_n i_{\mathrm{Sh}}(\mathbb{R}^n)
    \simeq
    i_{\mathrm{Sh}}(\mathbb{R}^\infty)
    \longrightarrow
    i_{\mathrm{Sh}}(\mathbb{R})
  $$
  are in natural bijection with morphisms of Fr{\' e}chet manifolds of the form
  $$
    \mathbb{R}^\infty
      \longrightarrow
    \mathbb{R}
    \,.
  $$
\end{example}

\begin{definition}[{\cite{Takens79}}]
  \label{LocallyOfFiniteOrder}
	A function (of underlying sets) on $\mathbb{R}^\infty$
	(definition~\ref{RInfinity})
  $$
    f : \mathbb{R}^\infty \longrightarrow \mathbb{R}
  $$
	we call \emph{smooth and locally of finite order} if for every point
	$x \in \mathbb{R}^\infty$ there is an open neighbourhood $U_k \subset
	\mathbb{R}^k$ of $p_k(x)$ and a smooth function $f_k : U_k \to
	\mathbb{R}$ which determines the restriction of $f$ to the pre-image
	$U = p_k^{-1}(U)$:
  $$
    f|_{U} = f_k \circ p_k
    \,.
  $$
\end{definition}

\begin{proposition}[{\cite[\textsection 9.5.9]{Michor80}}]
  \label{SmoothFunctionsOnFrechetRInfinity}
  Morphisms of Fr{\' e}chet manifolds of the form
  $$
    f : \mathbb{R}^\infty \longrightarrow \mathbb{R}
  $$
  are precisely the functions of underlying sets that are \emph{smooth and locally of finite order}
  in the sense of definition~\ref{LocallyOfFiniteOrder}.
\end{proposition}
\begin{proof}
	If $f$ is smooth and locally of finite order, then it is
	straightforward to check that it is smooth also in the sense of
	definition~\ref{FrechetSmoothFunction}. On the other hand, if we
	already know that $f$ is locally of finite order, Fr\'echet smoothness
	directly implies the smoothness of each of its finite order
	restrictions. Thus, it remains only to show that a Fr\'echet smooth
	function is locally of finite order.

	This result eventually follows from the joint continuity of the
	derivative $(u,v) \mapsto D_vf(u)$ as a map $Df\colon U \times
	\mathbb{R}^\oo \to \mathbb{R}$, which turns out to be a rather strong
	requirement.

	For any point $x \in \mathbb{R}^\oo$, by linearity in the second
	argument, $D_0f(x) = 0$. By joint continuity, the preimage
	$(Df)^{-1}(-1,1)$ is open and contains every point $(x,0) \in
	\mathbb{R}^\oo \times \mathbb{R}^\oo$ with $x\in X$. The product
	topology implies the existence of a product neighborhood $U_x \times
	V_x \ni (x,0)$ for any $x\in \mathbb{R}^\infty$, such that $U_x \times
	V_x \subset D f^{-1}(-1,1)$. The projective limit topology implies
	that there exists a positive integer $k$ such that we can choose $U_x
	= U_x^k \times W$ and $V = V_x^k \times W$, where $U_x^k, V_x^k
	\subset \mathbb{R}^k$ are open and convex and $W \cong \ker
	(\mathbb{R}^\oo \to \mathbb{R}^k)$ for the standard projection.
	The order $k$ could be different
	for $U_x$ and $V_x$, but we can just take the maximum of the two. And,
	of course, $k$ depends both on $x$ and on the choice of the product
	neighborhood $U_x \times V_x$.

	With $w\in W$, $tw \in W$ for any $t\in \mathbb{R}$. But then for any
	$u\in U_x$, by construction and linearity,
	\[
		|D_{tw}f(u)| = |t| |D_w f(u)| \le 1 .
	\]
	Since $|t|$ could be arbitrarily large, this means that $D_w f(u) = 0$
	for any $w \in W$. Thus, $k$ gives us a local uniform upper
	bound on the number of independent non-vanishing partial derivatives
	of $f$ in $U_x$. To conclude the proof, we use the fundamental theorem
	of calculus (Proposition~\ref{FrechetFundamentalTheorem}) to show that
	$k$ is also an upper bound on the number of coordinates on which $f$
	depends in a non-trivial way in $U_x$. Namely, any $u \in U_x =
	U_x^k \times W$ has the form $u = (u_0,w)$, where $u_0 \in U_x^k$ and
	$w \in W$, so that
	\[
		f(u_0,w) - f(u_0,0) = \int_0^1\d{t} \, D_{(0,w)}f(u_0,tw) = 0 .
	\]
	Hence $f|_{U_x}$ is constant along the fibers of $U_x \to U_x^k$ and
	thus factors through that projection. Since $x$ was arbitrary, $f$ is
	everywhere locally of finite order, which completes the proof.
\end{proof}

We will be considering the following full subcategory of ``locally pro-manifolds'' (see remark \ref{LocProTerminology} below)
inside Fr{\'e}chet manifolds (this is essentially the unnamed definition in
\cite[section 1]{Marvan86}):
\begin{definition}
  \label{LocallyProManifold}
  Write
  $$
    \mathrm{LocProMfd}
      \hookrightarrow
    \mathrm{FrMfd}
  $$
  for the full subcategory of that of Fr{\'e}chet manifolds (definition \ref{FrechetManifoldsDef})
  on those Fr{\'e}chet manifolds which are $n$-dimensional for $n \in \mathbb{N}$
  (hence ``$\mathbb{R}^n$-manifolds'') or
  which are $\mathbb{R}^\infty$-manifolds (definition \ref{RInfinity}).
\end{definition}
\begin{remark}
  \label{LocProMfdFullyIncludedInFrechetManifolds}
Hence by proposition \ref{SmoothSetsReceiveFrechetManifolds} there is a sequence of full inclusion
$$
  \mathrm{SmthMfd}
     \hookrightarrow
     \mathrm{LocProMfd} \hookrightarrow \mathrm{FrMfd} \hookrightarrow \mathrm{SmoothSet} \hookrightarrow \mathrm{FormalSmoothSet}
  \,.
$$
Since full inclusions reflect limits, it is still true that
$\mathbb{R}^\infty$ is the projective limit over the $\mathbb{R}^n$-s
when formed in $\mathrm{LocProMfd}$.
\end{remark}
\begin{remark}
\label{LocProTerminology}
The terminology in definition \ref{LocallyProManifold} is motivated as
follows: A pro-object in finite dimensional manifolds (``pro-manifold''
for short) is a \emph{formal} projective limit of finite dimensional
manifolds. This implies that a smooth function out of a pro-manifold is
represented by a smooth function on a finite stage of the corresponding
projective system, hence is \emph{globally of finite order}. Now by
proposition \ref{ProjectiveLimitNatureOfRInfinity} also the objects in
$\mathrm{LocProMfd}$ are projective limits of finite dimensional
manifolds, but, by proposition \ref{SmoothFunctionsOnFrechetRInfinity}, smooth
functions on them are in general only locally of finite order. In this
sense objects of $\mathrm{LocProMfd}$ locally look like pro-manifolds,
while globally they are a little more flexible.
\end{remark}

We will need to combine the concept of locally pro-manifolds with that of formal thickening,
yielding the following concept of formal locally pro-manifolds:

\begin{definition}
  \label{FormalLocalProManifolds}
  Write
  $$
    \mathrm{FormalLocProMfd} \hookrightarrow \mathrm{FormalSmoothSet}
  $$
  for the full subcategory of the Cahiers topos (definition \ref{FormalSmoothSets})
  of objects $X$ such that
  \begin{enumerate}
    \item the reduction (definition \ref{ReImEtAdjunctionOnFormalSmoothSets}) of $X$ is a locally pro-manifold
    (definition \ref{LocallyProManifold})
    $$
      \Re X \in \mathrm{LocProMfd} \hookrightarrow \mathrm{FormalSmoothSet}
      \,,
    $$
    \item there exists an order $k$-infinitesimal disk $\mathbb{D} \in \mathrm{InfPoint} \in \mathrm{FormalSmoothSet}$
    (definition \ref{StandardInfinitesimalnDisk}, definition \ref{InfinitesimalDiskInManifold})
    or a formal disk $\mathbb{D} \in \mathbf{H}$ (proposition \ref{FormalDiskInSmoothManifolds})
    such that $X$ is locally the Cartesian product of a locally pro-manifold with $\mathbb{D}$,
    hence such that
     there exists an open cover $\{U_i \overset{\phi_i}{\longrightarrow} \Re X\}_{i \in I}$
     and for each $i \in I$ a commuting diagram of the form
     $$
       \begin{gathered}
       \xymatrix{
         U_i \ar[d]_{\mathrm{et}}^{\phi_i}
         \ar[rr]^{\epsilon_{U_i \times \mathbb{D}}}
           &&
         U_i \times \mathbb{D}
          \ar[d]^{\mathrm{et}}
         \\
         \Re X
           \ar[rr]_-{\epsilon_X}
         &&
         X
       }
       \end{gathered}
       \, ,
     $$
     where the horizontal morphisms are the units of the $\Re$-monad, and where in addition to the left
     vertical morphisms (which is formally {\'e}tale by proposition \ref{LocalDiffeoIsFormallyEtals})
     also the right morphisms is formally {\'e}tale (definition \ref{FormallyEtaleMorphism}).
  \end{enumerate}
  Following \cite{Kock80} we call these objects
  \emph{formal locally pro-manifolds}, or just \emph{formal manifolds}, for short.
\end{definition}

\begin{definition}[{\cite[Sec.1.2]{Marvan86}}]
  \label{SubmersionBetweenLocallyProManifolds}
	We say that a morphism $f\colon X \longrightarrow Y$ between locally
	pro-manifolds (definition \ref{LocallyProManifold}) is a
	\emph{surjective submersion} if for every $x\in X$, there exist open
	neighborhoods $U$ of $x$ and $V$ of $y=f(x)$ of the form $U \simeq
	\mathbb{R}^m \times \mathbb{R}^n$ and $V \simeq \mathbb{R}^n$, with $n
	= \dim Y$ and $m$ either a finite non-negative integer or $\oo$, such
	that $U \stackrel{f}{\longrightarrow} V$ is equivalent to the
	projection $\mathbb{R}^m \times \mathbb{R}^n \to \mathbb{R}^n$ onto
	the second factor.
\end{definition}

\section{Jets and PDEs}
\label{JetsPDEs}

We discuss here how purely formal constructions induced by the presence
of the infinitesimal shape operation $\Im$ in a differentially cohesive topos $\mathbf{H}$ (definition \ref{DifferentialCohesion})
yields a synthetic theory of jet bundles, and of partial differential equations.

We frequently illustrate the general discussion with the archetypical
example the topos
$$
	\mathbf{H} = \mathrm{FormalSmoothSet}
$$
from definition~\ref{FormalSmoothSets} (Dubuc's ``Cahiers topos''),
which is differentially cohesive%
\footnote{
  In fact, the Cahiers topos is also (plain) \emph{cohesive} in the sense of Lawvere \cite{Lawvere07}, cf.~\cite[Prop. 4.5.8]{dcct}. This combined cohesion and differential cohesion of the Cahiers topos plays a key role for the general development of differential geometry (\cite[\S 4.2]{SS20-Orb}) but in the present article we will exclusively use the axiomatics of differential cohesion.
} %
by proposition~\ref{diffcohesionofFormalSmoothSet}, showing that and how it
reproduces and generalizes traditional constructions.

\subsection{$V$-Manifolds}

We consider now an axiomatization of manifolds in the generality where the local model space may be
any group object $V$ in a differentially cohesive topos such as that of formal smooth sets, or rather the formal neighbourhood of its neutral element
(see above definition~\ref{VManifold} below).
Ordinary $n$-dimensional manifolds constitute the special case where $V = \mathbb{R}^n \in \mathrm{FormalSmoothset}$, equipped with
its canonical translation group structure.

For the present purpose of studying PDEs, the main point of the $V$-manifolds
is that their formal disk bundles have good ``microlinear'' structure, as shown by proposition \ref{FormalDiskBundleLocallyTrivializableOverVManifold} below, highlighted by corollary \ref{ComponentVersionOfDiskBundleMonadProduct}.
This plays a key role in the analysis of generalized PDEs with free variables ranging in $V$-manifolds,
below in section \ref{PDEs}.

Beyond that, the full force of the abstract definition of $V$-manifolds
is obtained when passing to the formal smooth $\infty$-groupoids, in which case we obtain
a concept of {\'e}tale $\infty$-stacks locally modeled on some $\infty$-group $V$.
This idea will be developed in followup works.

Throughout, let $\mathbf{H}$ be a differentially cohesive topos (definition \ref{DifferentialCohesion})
such as Dubuc's Cahiers topos $\mathrm{FormalSmoothSet}$ (definition \ref{FormalSmoothSets}).

\begin{definition}\label{FormallyEtaleMorphism}
  For $f : X \longrightarrow Y$ a morphism in $\mathbf{H}$,
  we say that it is \emph{formally {\'e}tale} precisely if the
  naturality square of the unit of the $\Im$-monad (definition~\ref{ReImEtAdjunctionOnFormalSmoothSets}) is Cartesian (a pullback square):
  $$
    \begin{gathered}
    \xymatrix{
      X \ar[d]_f \ar[r]^{\eta_X}_-{\ }="s" & \Im X \ar[d]^{\Im f}
      \\
      Y \ar[r]_{\eta_Y}^{\ }="t" & \Im Y
      \ar@{}|{\mbox{\tiny (pb)}} "s"; "t"
    }
    \end{gathered}
    \,.
  $$
  Moreover, we say that the \emph{formal {\'e}talification} of $f$
  is the pullback $(\Im_Y X \stackrel{\Im_Y f}{\to} Y) := \eta_Y^\ast \Im f$ in
  $$
    \begin{gathered}
    \xymatrix{
      \Im_Y X  \ar[r] \ar[d]_{\Im_Y f} \ar@{}[dr]|{\mbox{\tiny (pb)}}
      & \Im X \ar[d]^{\Im f}
      \\
      Y \ar[r]_{\eta_Y} & \Im Y
    }
    \end{gathered}
    \,.
  $$
  If the canonical comparison morphism
  $$
    X \longrightarrow \Im_Y X
  $$
  is
  \begin{itemize}
    \item  an epimorphism, then $f$ is called \emph{formally smooth},
    \item a monomorphism, then $f$ is called \emph{formally unramified}.
  \end{itemize}
  Finally, we say that an object $X  \in \mathbf{H}$ is \emph{formally smooth} or \emph{formally
  unramified} if the terminal morphism $X \to \ast$ is so, respectively.
  This is the case precisely if
  $$
    \eta_X : X \to \Im X
  $$
  is an epimorphism or monomorphism, respectively.
\end{definition}
\begin{proposition}
  \label{LocalDiffeoIsFormallyEtals}
  For $X,Y \in \mathrm{SmthMfd} \hookrightarrow \mathrm{SmoothSet} \overset{i_!}{\hookrightarrow} \mathrm{FormalSmoothSet}$
  two smooth manifolds,
  regarded as formal smooth sets, a morphism $f : X \to Y$ between them
  is formally {\'e}tale in the sense of definition~\ref{FormallyEtaleMorphism}
  precisely if it is a local diffeomorphism in the traditional sense.
\end{proposition}
\begin{proof}
  A function $f$ between smooth manifolds is a local diffeomorphism in the traditional sense
  precisely if the induced square
  $$
    \xymatrix{
      T X \ar[d]_{d f} \ar[r]_{\ }="s" & X \ar[d]^f
      \\
      T Y \ar[r]^{\ }="t" & Y
      \ar@{}|{\mbox{\tiny (pb)}} "s"; "t"
    }
  $$
  of tangent bundles is Cartesian (is a pullback) in $\mathrm{SmthMfd}$.
  This equivalently means that for all $U \in \mathrm{SmthMfd}$ then the image
  $$
   \mathrm{Hom}_{\mathrm{SmthMfd}}
   \left(
     U,
    \begin{gathered}
    \xymatrix{
      T X \ar[d]_{d f} \ar[r]_{\ }="s" & X \ar[d]^f
      \\
      T Y \ar[r]^{\ }="t" & Y
      \ar@{}|{\mbox{\tiny (pb)}} "s"; "t"
    }
    \end{gathered}
    \right)
  $$
  of this square under
  $\mathrm{Hom}_{\mathrm{SmthMfd}}(U,-) : \mathrm{SmthMfd} \to \mathrm{Set}$ is a Cartesian square in sets.

  This we may equivalently rewrite as
  $$
    \mathrm{Hom}_{\mathrm{SmthMfd}}\left(
    U,
    \begin{gathered}
    \xymatrix{
      X^{\mathbb{D}^1(1)} \ar[d]_{f^{\mathbb{D}^1(1)}} \ar[rr]^-{X^{(\ast \to \mathbb{D}^1(1))}}_{\ }="s"
      && X \ar[d]^f
      \\
      Y^{\mathbb{D}^1(1)} \ar[rr]_-{Y^{(\ast \to \mathbb{D}^1(1))}}^{\ }="t"
      && Y
      \ar@{}|{\mbox{\tiny (pb)}} "s"; "t"
    }
    \end{gathered}
    \right)
    \,,
  $$
  where $(-)^{\mathbb{D}^1(1)}$ denotes the internal hom%
		\footnote{The \emph{internal hom} $A^B$ is an object
		satisfying the identity $\mathrm{Hom}(-,A^B) \simeq
		\mathrm{Hom}(-\times B, A)$ as a natural bijection of sets.} %
	out of the standard first order infinitesimal 1-disk
	(definition~\ref{StandardInfinitesimalnDisk}).
    (Hence $X^{(\ast \to \mathbb{D}^1(1))}$ denotes the morphism on internal homs into $X$ out of $\ast$ (which is $X$) and
    out of $\mathbb{D}^1(1)$, respectively, which is induced by the unique point inclusion $\ast \to \mathbb{D}^1(1)$.)
	By the $(\Re \dashv
	\Im)$-adjunction (definition~\ref{ReImEtAdjunctionOnFormalSmoothSets})
	and since $\Re (\mathbb{D}^1(1)) \simeq  \ast$
	(proposition~\ref{ReductionOnFormalManifolds}), this in turn is a pullback of sets
precisely if the following is:
  \begin{equation*}
    \label{star}
    \tag{$\star$}
    \mathrm{Hom}
    \left(
     U \times \mathbb{D}^1(1)
     ,\;
    \begin{gathered}
    \xymatrix{
      X \ar[d]_{f} \ar[r]^{\eta_X}_{\ }="s" & \Im X \ar[d]^{\Im f}
      \\
      X \ar[r]_{\eta_Y}^{\ }="t" & \Im Y
      %
    }
    \end{gathered}
    \right)
    \,.
  \end{equation*}
  Now the square on the right is Cartesian precisely if the stronger condition holds that
  \begin{equation*}
    \label{star-star}
    \tag{$\star\star$}
    \mathrm{Hom}
    \left(
     U \times \mathbb{D}
     ,\;
    \begin{gathered}
    \xymatrix{
      X \ar[d]_{f} \ar[r]^{\eta_X}_{\ }="s" & \Im X \ar[d]^{\Im f}
      \\
      X \ar[r]_{\eta_Y}^{\ }="t" & \Im Y
      %
    }
    \end{gathered}
    \right)
  \end{equation*}
  is a pullback, for all $\mathbb{D} \in \mathrm{InfPoint}$.

	So this shows that $f$ being formally {\'e}tale implies that it is a
	locally diffeomorphism. To see that there is no further condition from
	using more general $\mathbb{D}$ than $\mathbb{D}^1(1)$ observe that a
	local diffeomorphism between smooth manifolds is in fact an
	isomorphism on every sufficiently small open neighbourhood around
	every point $x \in X$. Since all infinitesimal points $\mathbb{D}$
	necessarily factor through any such open neighbourhood, this shows
	that~\eqref{star-star} is a pullback already if~\eqref{star} is.

\end{proof}

The following gives a generalized concept of manifolds locally modeled
on some model space $V$. In this article the only explicit examples we
consider are $V = \mathbb{R}^n$ for $n \in \mathbb{N} \cup \{\infty\}$
(example \ref{SmoothManifoldsAreVManifolds} below), but for later
development more generality could be needed. The idea is that all we need of a
local model space $V$ is that it has some kind of differentiable
structure and that it has some kind of additive structure that allows to
translate along it, even if it is not commutative. The first point is
satisfied by taking $V$ to be an object of our category $\mathbf{H}$.
The second point means that it is a group object in $\mathbf{H}$  (definition
\ref{group}).

\begin{definition}
  \label{VManifold}
	Let $V \in \mathrm{Grp}(\mathbf{H})$ be a group object in
	$\mathbf{H}$ (definition \ref{group}). We say that an object $X \in \mathbf{H}$ is a
	\emph{$V$-manifold} if there exists $U \in \mathbf{H}$ and a diagram
	of the form
  $$
    \begin{gathered}
    \xymatrix{
      & U \ar@{->>}[dr]^{\mathrm{et}}
      \ar[dl]_{\mathrm{et}}
      \\
      V && X
    }
    \end{gathered}
    \,,
  $$
	i.e.,\ a span from $V$ to $X$ such that both legs are formally
	{\'e}tale according to definition~\ref{FormallyEtaleMorphism} and such
	that in addition the right leg is an epimorphism in $\mathbf{H}$. Any
	such $U$ we call a \emph{$V$-atlas} of $X$.
\end{definition}
\begin{example}
  \label{SmoothManifoldsAreVManifolds}
	An ordinary smooth manifold of dimension $n$ becomes a $V$-manifold in
	the sense of definition~\ref{VManifold} by taking $V = \mathbb{R}^n$
	equipped with its canonical additive group structure. A $V$-atlas is
	obtained from any ordinary atlas with charts $\{\mathbb{R}^n
	\overset{\phi_i}{\to} X\}_{i \in I}$ by setting $U := \bigsqcup_{i \in
	I} \mathbb{R}^n$ and taking the two morphisms to be the canonical ones
  $$
    \begin{gathered}
    \xymatrix{
      & 
        \displaystyle U 
        = 
        \bigsqcup_{\smash{i \in I}} \mathbb{R}^n
      \ar@{->>}[dr]^{(\phi_i)_{i \in I}}
      \ar[dl]_{(\mathrm{id})_{i \in I}}
      \\
      \mathbb{R}^n && X
    }
    \end{gathered}
    \,.
  $$
	Since $U\to X$ is an open cover, every germ of a function on an open
	$n$-ball $B^n_\delta \to X$ factors through $\mathbb{R}^n \simeq
	B^n_\delta \to U \to X$. Thinking of $U$ and $X$ as sheaves in the
	sheaf topos $\mathrm{SmoothTopos} = \mathrm{Sh}(\mathrm{CartSp})$,
	this means that $U\to X$ induces an epimorphism from stalks of $U$ to
	the stalks of $X$ over enough points of the topos $\mathrm{SmoothSet}$
	(proposition~\ref{StalksOnCartSp}), which in turn implies that $U \to
	X$ is itself an epimorphism
	(proposition~\ref{EpimorphismsInCategoriesOfSheaves}).
	Since both morphisms are local diffeomorphisms, by construction, they
	are formally {\'e}tale morphisms in $\mathrm{SmoothSet}$ by
	proposition~\ref{LocalDiffeoIsFormallyEtals}.

	Similarly, a locally pro-manifold in the sense of
	definition~\ref{LocallyProManifold} becomes a $V$-manifold
	in the sense of definition~\ref{VManifold}, with $V = \mathbb{R}^n$
	and $n \in \mathbb{N} \cup \{\oo\}$, under the embedding of
	Fr{\'e}chet manifolds
	(proposition~\ref{SmoothSetsReceiveFrechetManifolds}).
\end{example}

\subsection{Infinitesimal disks}
\label{FormalDisks}

Jets (discussed below in section~\ref{Jets}) are functions out of infinitesimal
and formal disks (balls). Therefore here we first discuss infinitesimal
and formal disks.

\begin{definition}
  \label{InfinitesimalDiskInManifold}
	For $X \in \mathrm{SmthMfd} \hookrightarrow \mathrm{SmoothSet}
	\overset{i_!}{\hookrightarrow} \mathrm{FormalSmoothSet}$ a smooth manifold of
	dimension $n$, for $x : \ast \to X$ a point and for $k \in
	\mathbb{N}$, we have that the \emph{infinitesimal disk of order $k$ around $x$
	in $X$} is the subobject
  $$
    \mathbb{D}_x(k)
      \hookrightarrow
    X
  $$
  in $\mathrm{FormalSmoothSet}$ given, up to isomorphism, as the composite
  $$
    \begin{gathered}
    \xymatrix{
      \mathbb{D}^n(k)
       \ar@{^{(}->}[r]
      &
      \mathbb{R}^n
        \ar[r]^-{\phi}
      &
      X
    }
    \end{gathered}
    \,,
  $$
	where the first inclusion is the defining one of the standard
	infinitesimal $d$-disk of order $k$ into $\mathbb{R}^d$
	(definition~\ref{StandardInfinitesimalnDisk}) and where $\phi$ is any
	smooth embedding sending the origin of $\mathbb{R}^n$ to $x$.
\end{definition}
More abstractly, we may speak of the formal disk around any point in any
object of a differentially cohesive topos $\mathbf{H}$ as follows.
\begin{definition}
  \label{FormalDiskAbstractly}
	Let $X \in \mathbf{H}$ and $x : \ast \to X$ a point. Then the
	\emph{formal disk} in $X$ around $x$ is the fiber product
  $$
    \mathbb{D}_x := X \times_{\Im X} \{x\}
    \,,
  $$
  i.e.,\ the object sitting in the pullback diagram of the form
  $$
    \xymatrix{
      \mathbb{D}_x \ar[r]_{\ }="s" \ar[d] & \ast \ar[d]^{x}
      \\
      X \ar[r]_-{\eta_X}^{\ }="t" & \Im X
      \ar@{}|{\mbox{\tiny (pb)}} "s"; "t"
    }
    \,,
  $$
	where $\Im$ is the monad from
	definition~\ref{ReImEtAdjunctionOnFormalSmoothSets}, and $\eta_X : X
	\to \Im X$ its unit at $X$.
\end{definition}
The following proposition unwinds this abstract definition for the case
that $X$ is a manifold.
\begin{proposition}
  \label{FormalDiskInSmoothManifolds}
	Let $X \in \mathrm{SmthMfd} \hookrightarrow \mathrm{FormalSmoothSet}$ be a smooth
	manifold and $x \colon \ast \to X$ a point. Then the formal disk
	$\mathbb{D}_x \in \mathbf{H} \simeq
	\mathrm{Sh}(\mathrm{FormalSmoothSet})$ according to
	definition~\ref{FormalDiskAbstractly} is given by the sheaf which is
	the filtered colimit over the infinitesimal disks $\mathbb{D}_x(n)
	\hookrightarrow X$ in $X$ at $x$ (from
	definition~\ref{InfinitesimalDiskInManifold})
  $$
    \mathbb{D}_x \simeq \varinjlim_k \mathbb{D}_x(k)
    \,.
  $$
\end{proposition}
\begin{proof}
    First of all we observe that the diagram
    $$
      \xymatrix{
        \varinjlim_k \mathbb{D}_x(k)
        \ar[d]
        \ar[r]
        &
        X
        \ar[d]
        \\
        \ast \ar[r]_x & \Im X
      }
    $$
    indeed commutes, where the top morphism is given component-wise by the defining inclusions $\mathbb{D}_x(k) \hookrightarrow X$.
    By the $(\Re \dashv \Im)$ adjunction (definition \ref{ReImEtAdjunctionOnFormalSmoothSets}) this corresponds to the diagram
    $$
      \begin{gathered}
      \xymatrix{
        \Re(\varinjlim_k \mathbb{D}_x(k))
        \ar[d]
        \ar[r]
        &
        \Re X
        \ar[d]
        \\
        \Re \ast
        \ar[r]
        &
        X
      }
      \end{gathered}
      \,.
    $$
    Since $X$ is an ordinary manifold and hence already reduced by assumption,
    the morphism on the right is an isomorphism, and the bottom morphism is just the
    inclusion of the point $x$. Moreover, since $\Re$ is left adjoint it preserves the
    colimit in the top left, and since $\Re (\mathbb{D})\simeq \ast$ for all infinitesimally
    thickened points $\mathbb{D}$ (proposition \ref{ReductionOnFormalManifolds})
    also the top morphism is just the inclusion of the point $x$. Hence the diagram commutes.

    Observe that the top morphism in the first diagram is a monomorphism
    $$
      \varinjlim_k \mathbb{D}_x(k) \hookrightarrow X
      \,.
    $$
    (Since each $\mathbb{D}_x(k) \to X$ is a monomorphism and monomorphisms as well as
    colimits of sheaves are detected on stalks.)

	Therefore to check that the first square above is indeed Cartesian, it is sufficient to check
	that for every $\mathbb{R}^n \times \mathbb{D} \in
	\mathrm{FormalCartSp}$ (definition~\ref{FormalSmoothSets}) then
	morphisms $f : \mathbb{R}^n \times \mathbb{D} \longrightarrow X$ such
	that $\eta_X \circ f$ is constant on $x$ factor through
	$\varinjlim_k \mathbb{D}_x(k) \to X$.

    Let a given $\mathbb{D}$ be of infinitesimal order $k_{\mathbb{D}}$
    (meaning that in $C^\infty(\mathbb{D}) = \mathbb{R}\oplus V$ then $V^{k_{\mathbb{D}}+1}= 0$).
    We will show that then in fact we have factorization through
    $\mathbb{D}^n(k_{\mathbb{D}})\to \varinjlim_k \mathbb{D}_x(k) \to X$

 To that end, observe that for every point $y : \ast \to \mathbb{R}^n$ we have the restriction $f|_y$ of $f$ to that point
 $$
   \xymatrix{
     \mathbb{D}
     \ar[dr]^{f|_y}
     \ar[d]_{(y,\mathrm{id})}
     \\
     \mathbb{R}^n \times \mathbb{D}
     \ar[r]_-{f}
     &
     X
   }
   $$
   and the bottom morphism $f$ is fully determined once
   all its restrictions $f|_x$ are known (since, dually, smooth functions in the image of
   $C^\infty(X) \overset{f^\ast}{\longrightarrow} C^\infty(\mathbb{R}^n)\otimes_{\mathbb{R}}C^\infty(\mathbb{D})$
   are uniquely determined%
   \footnote{
     Elements of the tensor product $C^\infty(\mathbb{R}^n) \otimes_{\mathbb{R}} C^\infty(\mathbb{D})$ are polynomials in formal nilpotent variables with coefficients in $C^\infty(\mathbb{R}^n)$. These elements are clearly determined by their values at all points $y$ in $\mathbb{R}^n$.
   }
   by their restrictions along
   $(-)|_y :  C^\infty(\mathbb{R}^n)\otimes_{\mathbb{R}}C^\infty(\mathbb{D}) \to C^\infty(\mathbb{D})$, for all $y \in \mathbb{R}^n$).

  Therefore it is sufficient to check the factorization property
  for morphisms of the form $\mathbb{D} \to X$, with $\mathbb{D} \in
  \mathrm{InfPoint}$ (definition~\ref{InfinitesimallyThickenedPoints}), and
through a fixed stage $k$ of the colimit: because if we have factorizations
  of the form shown by the dashed morphism here
  $$
    \xymatrix{
       \mathbb{D}
       \ar@{-->}@/_2.4pc/[ddrr]
       \ar|{(y,\mathrm{id})}[dr]
       \\
       & \mathbb{R}^n \times \mathbb{D}
       \ar@{..>}[dr]
       \ar@/_1pc/[ddr]
       \ar@/^1pc/[drr]
       \\
       &&
       \mathbb{D}_x(k_{\mathbb{D}})
       \ar[d]
       \ar@{^{(}->}[r]
       &
       X
       \ar[d]
       \\
       & & \ast \ar@{^{(}->}[r] & \Im X
    }
  $$
  for all $y \in \mathbb{R}^n$,
  then there is an induced factorization as shown by the dotted arrow.

  Observe that the composite morphism $\mathbb{D} \longrightarrow X$
  in the above diagram must take the global point $\ast \to \mathbb{D}$ to
  $\ast \overset{x}{\to} X$, by the same kind of argument as for the commutativity of the
  square at the beginning of the proof.

	Now if $\mathbb{D}$ is of order $k$ (i.e.,\ if the maximal ideal $V$ of
	$C^\infty(\mathbb{D})$ satisfies $V^{k+1} = 0$), this means that
	morphisms $f : \mathbb{D} \to X$ are equivalently $\mathbb{R}$-algebra
	homomorphisms of the form $(\mathrm{ev}_x,\rho) :  C^\infty(X) \to
	\mathbb{R} \oplus V$ where $\rho$ is a function of the partial
	derivatives of $f$ at $x$ of order at most $k$. By
	proposition~\ref{InfinitesimalPointInSmoothManifoldFactorsThroughInfinitesimalDisk}
	such an algebra homomorphism factors through
	$C^\infty(\mathbb{D}_x(k))$.
 Hence we have a factorization
 $$
    f :
    \xymatrix{
      \mathbb{D}
         \ar[r]
         \ar@/_1pc/[rr]
      &
      \mathbb{D}_x(k)
         \ar[r]
         &
      \varinjlim_k \mathbb{D}_x(k)
        \ar[r]
        &
      X
    } \, .
  $$
\end{proof}

We collect all the infinitesimal disks
(definition~\ref{FormalDiskAbstractly}) of a formal smooth set as
follows:
\begin{definition}
  \label{InfinitesimalDiskBundle}
	Let $X \in \mathbf{H}$ be any object of a differentially cohesive topos. Then its \emph{formal
	disk bundle}
  $$
    \xymatrix{
      T^\infty X
      \ar[d]
      \\
      X
    }
  $$
  is the pullback in
  $$
    \begin{gathered}
    \xymatrix{
      T^\infty X \ar[r]^{\mathrm{ev}}_-{\ }="s" \ar[d]_\pi & X \ar[d]^{\eta_X}
      \\
      X \ar[r]_{\eta_X}^-{\ }="t" & \Im X
      \ar@{}|{\mbox{\tiny (pb)}} "s"; "t"
    }
    \end{gathered}
    \,.
  $$
	More generally, for $E \overset{p}{\to} X$ a morphism in $\mathbf{H}$
	(e.g.,\ a bundle over $X$), $T^\infty_X E$ is the
	left vertical composite $\pi \circ \mathrm{ev}^\ast p$ in the pasting
	diagram
  $$
    \begin{gathered}
    \xymatrix{
      T^\infty_X E \ar[r]^-{\mathrm{ev}_E}_{\ }="s2" \ar[d]_{\mathrm{ev}^\ast p} & E \ar[d]^{p}
      \\
      T^\infty X \ar[r]^{\mathrm{ev}}^-{\ }="t2"_-{\ }="s" \ar[d]_\pi & X \ar[d]^{\eta_X}
      \\
      X \ar[r]_{\eta_X}^-{\ }="t" & \Im X
      \ar@{}|{\mbox{\tiny (pb)}} "s2"; "t2"
      \ar@{}|{\mbox{\tiny (pb)}} "s"; "t"
    }
    \end{gathered}
    \,.
  $$
	(This definition is essentially that in~\cite{Kock80}, just preceding
	its Proposition~2.2).

	More succinctly, according to proposition~\ref{basechange},
	$T^\infty_X$ is the monad on the slice of $\mathbf{H}$ over $X$ given
	by left base change along the unit of the $\Im$-monad
	(definition~\ref{ReImEtAdjunctionOnFormalSmoothSets}):
  $$
    T^\infty_X
    \; := \;
    (\eta_X)^\ast \circ (\eta_X)_!
    \; : \;
    \mathbf{H}_{/X}
      \longrightarrow
    \mathbf{H}_{/X}
    \,.
  $$
\end{definition}
\begin{remark}
	The object $T^\infty X$ in definition~\ref{InfinitesimalDiskBundle} is
	also called the \emph{formal neighbourhood of the diagonal} (cf.\ for
	instance below Proposition~2.1 in~\cite{Kock80}). As such one wants to
	think of the two canonical maps out of it on a more symmetric footing,
	and write $\mathrm{pr}_1 := \mathrm{\pi}$, $\mathrm{pr}_2 =
	\mathrm{ev}$. By the universal property of the pullback, a generalized
	point in $T^\infty X$ is an ordinary point of $X$ together with two
	infinitesimal neighbours
  $$
    
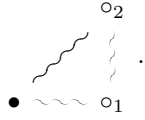
\begin{gathered}
    \xymatrix{
      & \circ_2 \ar@{~}[d]
      \\
      \bullet \ar@{~}[r] \ar@{~}[ur] & \circ_1
    }
    \end{gathered}
    \,.
  $$
	The projection $\mathrm{pr}_i$ takes such a triple to $\circ_i$.
	Definition \ref{InfinitesimalDiskBundle} above takes the non-symmetric
	perspective that $\mathrm{pr}_1 = \mathrm{\pi}$ is to be regarded as a
	bundle projection. Viewed this way then the above picture illustrates
	a point $\circ_2$ in the fiber of $T^\infty X$ over $\circ_1 \in X$.

	Assume that $X \in \mathbf{H}$ is formally smooth (Definition \ref{FormallyEtaleMorphism}),
    hence that $\eta_X : X \longrightarrow \Im X$ is an epimorphism in
	the topos $\mathbf{H}$ (which is the case generally for $\mathbf{H} = \mathrm{FormalSmoothSet}$,
    according to Proposition \ref{UnitOfInfinitesimalShapeIsEpiInCahiersTopos}). All epimorphisms in a sheaf topos are regular
	epimorphisms, meaning that they are the coequalizers of their kernel
	pair (prop. \ref{RegularMonomorphismInSheafTopos}). Here this means that the diagram
  $$
    \xymatrix{
      T^\infty X
       \ar@<-3pt>[rr]_{\mathrm{pr}_2 = \mathrm{ev}}
       \ar@<+3pt>[rr]^{\mathrm{pr}_1 = \pi}
      &&
      X
      \ar[rr]^{\eta_X}_{\mathrm{coeq}}
      &&
      \Im X
    }
  $$
	is a coequalizer diagram, manifestly exhibiting $\Im X$ as the
	quotient of the equivalence relation
  $$
    (x \sim y) \Leftrightarrow (\mbox{$x$ is an infinitesimal neighbour of $y$})
    \,.
  $$
\end{remark}
\begin{example}
  \label{FormalDiskBundleOfPoint}
	For $X \in \mathbf{H}$ any object and for $x : \ast \to X$
	a point, regarded as an object of $\mathbf{H}_{/X}$, its formal
	disk bundle according to definition~\ref{InfinitesimalDiskBundle} is
	just the formal disk $\mathbb{D}_x$ in $X$ (according to
	definition~\ref{FormalDiskAbstractly}) at that point:
  $$
    T^\infty_X \{x\} \simeq \mathbb{D}_x
    \,.
  $$
\end{example}
\begin{proof}
  Def. \ref{InfinitesimalDiskBundle} gives that
  $$
    T^\infty_X \{x\} \simeq T^\infty X \times_X \{x\}
    \,.
  $$
	The statement thus follows with the pasting law
	(proposition~\ref{PastingLaw}) and
	definition~\ref{FormalDiskAbstractly}:
  $$
    \begin{gathered}
    \xymatrix{
      \mathbb{D}_x \ar[r]
      \ar[d]
      \ar@{}[dr]|{\mbox{\tiny (pb)}}
        &
        \ast
        \ar[d]^{s}
      \\
      T^\infty X
        \ar[r]|{\mathrm{ev}}
        \ar[d]_{\pi}
        \ar@{}[dr]|{\mbox{\tiny (pb)}}
        &
      X
      \ar[d]^{\eta_X}
      \\
      X
        \ar[r]_{\eta_X}
        &
      \Im X
    }
    \end{gathered}
    \,.
  $$
\end{proof}
\begin{proposition}
   \label{MonadOperationsOnInfinitesimalDiskBundles}
	For $X \in \mathbf{H}$, the formal disk bundle monad $T^\infty_X :
	\mathbf{H}_{/X} \to \mathbf{H}_{/X}$ from
	definition~\ref{InfinitesimalDiskBundle} has the following structure
	morphisms:
  \begin{enumerate}
    \item The monad unit at an object $[E \stackrel{p}{\to} X]$ is $\eta_p = (p,\mathrm{id}) : E \to X \times_{\Im X} E = T^\infty_X E$;
    \item the monad product is given by $\Delta_X = T^\infty_X \mathrm{ev}_E : T^\infty_X
    T^\infty_X E \to T^\infty_X E$,
  \end{enumerate}
	where in the second item we regard $\mathrm{ev}_E$
	as a morphism in $\mathbf{H}_{/X}$ via
    $$
      \begin{gathered}
      \xymatrix{
        T^\infty_X E
          \ar[rr]^{\mathrm{ev}_E}
          \ar[dr]
          &&
        E
        \ar[dl]
        \\
        & X
      }
      \end{gathered}
      \, .
    $$
\end{proposition}
\begin{proof}
    This follows by direct unwinding of the definition of the monad structure
    induced by an adjunction (example \ref{AdjunctionGivesMonad}).
    For the first item
	observe that this is the unit of the adjunction $(\eta_X)_! \dashv (\eta_X)^\ast$:
    $$
      \begin{gathered}
      \xymatrix{
        E
        \ar@{-->}[dr]^{\eta_E}
        \ar@/_1.4pc/[dddr]_{p}
        \ar@/^1.4pc/[drr]^{\mathrm{id}}
        \\
        & T^\infty_X E
            \ar[r]^{\mathrm{ev}_E}
            \ar@{}[dr]|{\mbox{\tiny (pb)}}
            \ar[d]_{\pi_E}
        &
        E
          \ar[d]^{p}
        \\
        & T^\infty X
           \ar@{}[dr]|{\mbox{\tiny (pb)}}
           \ar[r]
           \ar[d]_\pi
        &
        X \ar[d]^{\eta_X}
        \\
        & X \ar[r]_{\eta_X} &  \Im X
      }
      \end{gathered}
      \, .
    $$
    For the second item, observe that the counit of the
	$(\eta_X)_! \dashv (\eta_X)^\ast$-adjunction on any $Q \overset{p}{\to} \Im X$ is the top horizontal
	morphism in
  $$
    \begin{gathered}
    \xymatrix{
      (\eta_X)^\ast Q \ar@{}[dr]|{\mbox{\tiny (pb)}} \ar[d]\ar[r] & Q \ar[d]^{p}
      \\
      X \ar[r]_{\eta_X} & \Im X
    }
    \end{gathered}
    \,,
  $$
	regarded as a morphism in $\mathbf{H}_{/\Im X}$.
For $(Q \to \Im X) = (\eta_X)_! E =  (E \to X \overset{\eta_X}{\to} \Im X)$ this square becomes the defining rectangle from definition \ref{InfinitesimalDiskBundle}:
  $$
    \begin{gathered}
    \xymatrix{
      T^\infty_X E
        \ar[r]^-{\mathrm{ev}_E}
        \ar[d]_{\pi_E}
        &
      E
      \ar[d]^p
      \\
      T^\infty X
        \ar[r]^-{\mathrm{ev}}
        \ar[d]_{\pi}
        &
      X
      \ar[d]^{\eta_X}
      \\
      X \ar[r]_-{\eta_X}
      &
      \Im X
    }
    \end{gathered}
    \, .
  $$
  Hence the monad product is the image of the top morphism $\mathrm{ev}_E$ here, regarded as
  a morphism over $\Im X$, under base change via pullback along $\eta_X : X \to \Im X$. This yields the claim.
\end{proof}
We now establish some basic properties inherited by $T^\infty X$ in the case that $X$ is a $V$-manifold.
\begin{proposition}
  \label{PullbackAlongFormallyEtalePreservesFormalDiskBundles}
	If a morphism $f : X \longrightarrow Y$ in $\mathbf{H}$ is formally
	{\'e}tale (definition~\ref{FormallyEtaleMorphism}), then pullback along
	$f$ preserves infinitesimal disk bundles
	(definition~\ref{InfinitesimalDiskBundle}), in that
  $$
    f^\ast (T^\infty Y) \simeq T^\infty X \;\;\; \in \mathbf{H}_{/X}
    \,.
  $$
\end{proposition}
\begin{proof}
  Observe that we have a commuting cube as follows
  $$
    \begin{gathered}
    \xymatrix{
      f^\ast T^\infty Y
       \ar[d]
       \ar[r]_{\ }="s1"
       &
      T^\infty Y
        \ar[d]
        \ar[r]_{\ }="s2"
      &
      Y
      \ar[d]^{\eta_Y}
      \\
      X
        \ar[r]_-f^{\ }="t1"
        &
      Y
       \ar[r]_{\eta_Y}^{\ }="t2"
       &
      \Im Y
      \ar@{}|{\mbox{\tiny (pb)}} "s1"; "t1"
      \ar@{}|{\mbox{\tiny (pb)}} "s2"; "t2"
    }
    \end{gathered}
    \;\;\;
     \simeq
    \;\;\;
    \begin{gathered}
    \xymatrix{
      T^\infty X
        \ar[d]
        \ar[r]_{\ }="s1"
        &
      X
       \ar[d]
       \ar[r]^f_{\ }="s2"
       &
      Y
       \ar[d]^{\eta_Y}
      \\
      X
        \ar[r]_{\eta_X}^{\ }="t1"
        &
      \Im X
        \ar[r]_{\Im f}^{\ }="t2"
        &
      \Im Y
      \ar@{}|{\mbox{\tiny (pb)}} "s1"; "t1"
      \ar@{}|{\mbox{\tiny (pb)}} "s2"; "t2"
    }
    \end{gathered}
    \,.
  $$
	Here the bottom composites agree by the naturality of $\eta$, and from
	this the squares are identified via the pasting law for pullbacks
	(proposition~\ref{PastingLaw}) and using the definition of the formal
	disk bundle (definition~\ref{InfinitesimalDiskBundle}) and of the
	property that $f$ is formally {\'e}tale
	(definition~\ref{FormallyEtaleMorphism}). The resulting isomorphism in
	the top left is the statement to be proven.
\end{proof}
\begin{proposition}
  \label{FormalDiskBundleLocallyTrivializableOverVManifold}
	Let $V \in \mathrm{Grp}(\mathbf{H})$ be a group object if a differentially cohesive topos.
	\begin{enumerate} \item The infinitesimal disk bundle $T^\infty V$
	over $V$ (definition~\ref{InfinitesimalDiskBundle}) trivializes, with
	typical fiber the formal disk $\mathbb{D}_e$
	(definition\ref{FormalDiskAbstractly}) of $V$ at the neutral element $e\in
	V$, i.e.,\ there is an equivalence of the form
  $$
    \begin{gathered}
    \xymatrix{
      T^\infty V
      \ar[dr]_{\pi}
      \ar[rr]^{\simeq }
      &&
      V \times \mathbb{D}_e
      \ar[dl]^{\mathrm{pr}_1}
      \\
      & V
    }
    \end{gathered}
    \,.
  $$
	Moreover, under this equivalence the morphism $\mathrm{ev} : T^\infty
	V \to V$ from definition~\ref{InfinitesimalDiskBundle} is identified
	with the operation of right translation of elements in $V$ by elements
	in $\mathbb{D}_e \overset{i}{\to} V$ in that the following diagram
	commutes
  $$
    \begin{gathered}
    \xymatrix{
      V \times \mathbb{D}_e
        \ar[d]_{\simeq} \ar[r]^{\mathrm{id} \times i}
      &
      V \times V
       \ar[r]^->>{(-) \cdot (-)^{-1}}
      &
      V
      \ar[d]^=
      \\
      T^\infty V
        \ar[rr]_{\mathrm{ev}}
      &&
      V
    }
    \end{gathered}
    \,.
  $$
  \item
	For $X$ a $V$-manifold (definition~\ref{VManifold}), its infinitesimal
	disk bundle $T^\infty X$ (definition~\ref{InfinitesimalDiskBundle}) is
	a locally trivial $\mathbb{D}_e$-fiber bundle. Specifically, for
	$\xymatrix{U \ar@{->>}[r]^{\mathrm{et}} & X}$ any $V$-atlas of $X$,
	there is a Cartesian diagram
  $$
    \begin{gathered}
    \xymatrix{
      U \times \mathbb{D}_e
      \ar[d]_{\mathrm{pr}_1}
      \ar[r]_{\ }="s"
      &
      T^\infty X
      \ar[d]
      \\
      U \ar@{->>}[r]_{\mathrm{et}}^{\ }="t" & X
      \ar@{}|{\mbox{\tiny (pb)}} "s"; "t"
    }
    \end{gathered}
    \, .
  $$
  \end{enumerate}
\end{proposition}
\begin{proof}
	Since $\Im$ is a right adjoint (by
	proposition~\ref{diffcohesionofFormalSmoothSet},
	definition~\ref{ReImEtAdjunctionOnFormalSmoothSets}) it preserves
	limits and hence in particular it preserves group objects. Therefore
	the defining Cartesian square
  $$
    \begin{gathered}
    \xymatrix{
      T^\infty V \ar[r]^{\mathrm{ev}}_-{\ }="s" \ar[d]_\pi & V \ar[d]^{\eta_V}
      \\
      V \ar[r]_{\eta_V}^-{\ }="t" & \Im V
      \ar@{}|{\mbox{\tiny (pb)}} "s"; "t"
    }
    \end{gathered}
  $$
	(from definition~\ref{InfinitesimalDiskInManifold}) is a fiber product
	over a group object $\Im V$. Hence the ``nonabelian Mayer-Vietoris
	argument'' applies (proposition~\ref{MayerVietoris}) and shows that
	there is equivalently a Cartesian square of the form
  $$
    \begin{gathered}
    \xymatrix{
      T^\infty V
      \ar[d]_{(\pi, \mathrm{ev})}
      \ar[rr]_-{\ }="s"
      &&
      \ast
      \ar[d]^{e}
      \\
      V \times V
      \ar[rr]_-{(\eta_V)\cdot (\eta_V)^{-1}}^-{\ }="t"
      &&
      \Im V
      \ar@{}|{\mbox{\tiny (pb)}} "s"; "t"
    }
    \end{gathered}
    \,.
  $$
	Since $\eta_V$ is a group homomorphism, the bottom morphism here is
	equivalently $\eta_V \circ ((-)\cdot (-)^{-1})$. Hence by the pasting
	law (proposition~\ref{PastingLaw}) and using the definition of the
	formal disk $\mathbb{D}_e$ (definition~\ref{FormalDiskAbstractly}), we
	obtain a factorization of the above diagram as
  $$
    \begin{gathered}
    \xymatrix{
      T^\infty V
      \ar@{}[dr]|{\mbox{\tiny (pb)}}
      \ar[d]_{(\pi, \mathrm{ev})}
      \ar[r]_-{\ }="s"
      &
      \mathbb{D}_e
      \ar[d]^i
      \ar[r]_-{\ }="s2"
      &
      \ast
      \ar[d]^{e}
      \\
      V \times V
      \ar[r]_-{(-)\cdot (-)^{-1}}^-{\ }="t"
      &
      V
      \ar[r]_-{\eta_V}^-{\ }="t2"
      &
      \Im V
      \ar@{}|{\mbox{\tiny (pb)}} "s2"; "t2"
    }
    \end{gathered}
    \,.
  $$
	From the universal property of the pullback, it follows that the square
	on the left here is equivalently
  $$
    \begin{gathered}
    \xymatrix{
      V \times \mathbb{D}_e
      \ar[d]_{(\mathrm{pr}_1,\, (-)\cdot (i(-))^{-1})}
      \ar[r]^-{\mathrm{pr}_2}_-{\ }="s"
      &
      \mathbb{D}
      \ar[d]^i
      \\
      V \times V
      \ar[r]_-{(-) \cdot (-)^{-1}}^-{\ }="t"
      &
      V
      \ar@{}|{\mbox{\tiny (pb)}} "s"; "t"
    }
    \end{gathered}
    \,.
  $$
	The equivalence of this square with the left square above is the first
	statement to be shown.

	For the second statement, observe, by
	proposition~\ref{PullbackAlongFormallyEtalePreservesFormalDiskBundles},
	that pullback along the two legs of any $V$-atlas
	(definition~\ref{VManifold})
  $$
    \xymatrix{
      & U \ar@{->>}[dr]^{\mathrm{et}}
      \ar[dl]_{\mathrm{et}}
      \\
      V && X
    }
  $$
	preserves formal disk bundles. Hence the previous statement, given
	that $T^\infty V$ is trivial, implies that $T^\infty U \simeq U \times
	\mathbb{D}_e$ is also trivial, and that so is the pullback of
	$T^\infty X$ to the $V$-atlas $U$. Hence $T^\infty X$ is locally
	trivial.
\end{proof}
\begin{remark}
	The proof of
	proposition~\ref{FormalDiskBundleLocallyTrivializableOverVManifold}
	is ``purely formal,'' and works in every
    differentially cohesive topos (definition~\ref{ReImEtAdjunctionOnFormalSmoothSets}). As
	such it has been formalized in the formal language of modal homotopy
	type theory \cite{Wellen17,Cherubini2024}. This means that the proof also works in
	the $\infty$-topos version of $\mathbf{H}$, cf. \cite[\S 10.2]{SS20-Orb}. There we may furthermore
	formally conclude that for $X$ a $V$-manifold, we have that the formal disk
	bundle $T^\infty X$ is associated to an
	$\mathbf{Aut}(\mathbb{D})_e$-principal bundle. This is the \emph{jet
	frame bundle} of the $V$-manifold $X$. The frame bundle (at any jet
	order) is the starting point for Cartan geometry (i.e.,\ the theory of
	torsion-free $G$-structures on manifolds) \cite{crampin-saunders}.
	This is a rich theory, for more on which see \cite[\S 10.2]{SS20-Orb}.
\end{remark}
In order to highlight the content of proposition \ref{FormalDiskBundleLocallyTrivializableOverVManifold} it will be convenient to
introduce the following notation:
\begin{definition}
  \label{ComponentNotationForInfinitesimalDiskBundles}
  Let $X$ be a $V$-manifold (definition \ref{VManifold}). Denote the local evaluation action
  from item 1 of proposition \ref{FormalDiskBundleLocallyTrivializableOverVManifold} by
  $$
    + : V \times \mathbb{D}_e \longrightarrow V
    \,.
  $$
  For $E \in \mathbf{H}$ an object, we may regard any morphism $p : \cdots \to E$ into it as a generalized element of $E$.
  Say that a generalized element of the total space of a bundle $E \to X$ is \emph{local} if its
  projection $x : \cdots \to E \to X$ factors through some $V$-cover $U \to X$.
  By item 2 of proposition \ref{FormalDiskBundleLocallyTrivializableOverVManifold} a local generalized element of
  the infinitesimal disk bundle $T^\infty_X E$ is equivalently a pair $p = (x,a)$ consisting of a generalized base point $x : \cdots \to X$ and of a generalized element
  $a : \cdots \to \mathbb{D}_e \times E$, such that the evaluation morphism
  $$
    \mathrm{ev}_E : T^\infty_X E \longrightarrow E
  $$
  is given by
  $$
    (x,a) \mapsto x+a
    \,.
  $$
\end{definition}
\begin{corollary}
  \label{ComponentVersionOfDiskBundleMonadProduct}
  In the notation of definition \ref{ComponentNotationForInfinitesimalDiskBundles},
  the structure morphism (definition \ref{monad}) of the $T^\infty_X$-monad from definition \ref{InfinitesimalDiskBundle}
  are given on local generalized elements as follows:
  \begin{enumerate}
  \item the unit morphism
  $$
    \eta_E \;:\; E \to T^\infty_X E
  $$
  is given by
  $$
    x \mapsto (x,0)
  $$
  (where $0 : \cdots \to \ast \to \mathbb{D}_e$ is the generalized element constant on the unique global base point of $\mathbb{D}_e$)
  \item
  the product morphism
  $$
    \nabla_E \;:\; T^\infty_X T^\infty_X E \longrightarrow T^\infty_X E
  $$
  is given by
  $$
    (b,a,b) \mapsto (b, a + b)
    \,.
  $$
  \end{enumerate}
\end{corollary}
\begin{proof}
  In view of proposition \ref{FormalDiskBundleLocallyTrivializableOverVManifold}
  this is the statement of proposition \ref{MonadOperationsOnInfinitesimalDiskBundles}.
\end{proof}
The following example shows what the abstract phenomena of proposition \ref{FormalDiskBundleLocallyTrivializableOverVManifold}
look like when realized in a basic concrete special case.
\begin{example}
  \label{AdditiveStructureInInfinitesimalDiskBundle}
	Let $X = \mathbb{R} \in \mathrm{CartSp} \hookrightarrow \mathrm{FormalSmoothSet}$ be
	the real line and write
  $$
    \mathbb{D}_0 = \varinjlim_k \mathbb{D}_0(k) \hookrightarrow \mathbb{R}
  $$
	for the formal infinitesimal disk in $\mathbb{R}$ around the origin
	(definition \ref{InfinitesimalDiskInManifold}). Then the product
	morphism of the infinitesimal disk bundle monad according to proposition
	\ref{MonadOperationsOnInfinitesimalDiskBundles} is of the form
  $$
    \xymatrix{
      T^\infty_X T^\infty_X X \ar[d]_\simeq \ar[rr] && T^\infty_X X \ar[d]^\simeq
      \\
      \mathbb{R} \times \mathbb{D}_0 \times \mathbb{D}_0
        \ar[rr]
        &&
      \mathbb{R} \times \mathbb{D}_0
    }
  $$
  and is given at infinitesimal order $k \in \mathbb{N}$ by the morphism
  $$
    \mathbb{R} \times \mathbb{D}_0(k) \times \mathbb{D}_0(k)
      \longrightarrow
    \mathbb{R} \times \mathbb{D}_0(k)
  $$
  in $\mathrm{FormalCartSp} \hookrightarrow \mathrm{FormalSmoothSet}$
  which is dual to the algebra homomorphism
  $$
    C^\infty(\mathbb{R}) \otimes_{\mathbb{R}} \langle (\mathbb{R} \oplus \langle \varepsilon_1\rangle)/(\varepsilon_1^{k+1})
    \otimes_{\mathbb{R}} (\mathbb{R} \oplus \langle \varepsilon_2\rangle)/(\varepsilon_2^{k+1})
      \longleftarrow
    C^\infty(\mathbb{R})\otimes_{\mathbb{R} \oplus \langle \varepsilon\rangle }/(\varepsilon^{k+1})
  $$
  that takes
  $$
    f(x,\varepsilon) \mapsto f(x, \varepsilon_1 + \varepsilon_2)
    \,.
  $$
\end{example}
\begin{proof}
  By proposition \ref{FormalDiskBundleLocallyTrivializableOverVManifold} and proposition \ref{FormalDiskInSmoothManifolds} the morphism
  $$
    \mathrm{ev} \;:\; T^\infty_{\mathbb{R}} \mathbb{R} \longrightarrow \mathbb{R}
  $$
  is the composite
  $$
    \xymatrix{
      \mathbb{R} \times \mathbb{D}_0
        \ar@{^{(}->}[r]
        &
        \mathbb{R} \times \mathbb{R}
        \ar[rr]^{(-)+(-)}
        &&
       \mathbb{R}
    }
  $$
  in $\mathrm{FormalCartSp} \hookrightarrow \mathrm{FormalSmoothSet}$ (definition \ref{FormalSmoothSets}).
  Dually, at infinitesimal order $k \in \mathbb{N}$, this is the algebra homomorphism
  $$
    \xymatrix{
      C^\infty(\mathbb{R}) \otimes_{\mathbb{R}}(\mathbb{R} \oplus \langle \varepsilon_1\rangle)/(\varepsilon_1^{k+1})
      \ar@{<-}[rr]^-{}
     &&
      C^\infty(\mathbb{R})
    }
    :
    \exp(\varepsilon \partial_x)
  $$
  which takes
  $$
    f(x) \mapsto \exp(\varepsilon \partial_x)f(x) := \sum_{n_1=0}^k \tfrac{1}{n!}f^{(n_1)}(x)\varepsilon_1^{n_1}
    \,.
  $$
  By proposition \ref{MonadOperationsOnInfinitesimalDiskBundles} the morphism that we are after is the image
  $T^\infty_X \mathrm{ev}$ of this morphism under $T^\infty_X$, hence its base change, as a morphism
  over $\Im \mathbb{R}$, along $\eta_{\mathbb{R}} : \mathbb{R} \to \Im \mathbb{R}$.
  To analyze what this does on fibers, we check what it does over any point, say $0 \in \mathbb{R}$, by base changing further along
  $\{0\} \to \mathbb{R} \to \Im \mathbb{R}$. This is given by the front top morphism in the following diagram:
  $$
    \begin{gathered}
    \xymatrix{
      & \mathbb{R} \times \mathbb{D}_0
        \ar[rr]^-{\mathrm{ev}}
        \ar[dd]|!{[dl];[dr]}{\hole}
        &&
      \mathbb{R}
      \ar[dd]
      \\
      \mathbb{D}_0 \times \mathbb{D}_0
      \ar[rr]
      \ar[dd]
      \ar[ur]
      &&
      \mathbb{D}_0
      \ar[dd]
      \ar[ur]
      \\
      & \mathbb{R}
        \ar[rr]|!{[ur];[dr]}{\hole}
      &&
      \Im \mathbb{R}
      \mathrlap{\,,}
      \\
      \mathbb{D}_0 \ar[rr] \ar[ur] & & \ast \ar[ur]_0
    }
    \end{gathered}
  $$
  where all squares are Cartesian (are pullback squares).
  Consider the top square.
  Since colimits are universal in $\mathbf{H}$ (proposition \ref{UniversalColimits})
  we may analyze this at any finite infinitesimal order $\mathbb{D}_0(k) \hookrightarrow \mathbb{D}_0$,
  where it becomes a pullback in the site $\mathrm{FormalCartSp} \overset{y}{\hookrightarrow} \mathrm{FormalSmoothSet}$.
  Since the Yoneda embedding preserves limits, we may compute this pullback square
  in $\mathrm{FormalCartSp}$. Since by definition there is a full inclusion
  $\mathrm{FormalCartSp}^{\mathrm{op}} \hookrightarrow \mathrm{CAlg}_{\mathbb{R}}$,
  for this it is sufficient to find a pushout square
  in $\mathrm{CAlg}_{\mathbb{R}}$. As such this is the following:
  $$
    \xymatrix{
      C^\infty(\mathbb{R}) \otimes_{\mathbb{R}}(\mathbb{R} \oplus \langle \varepsilon_1\rangle)/(\varepsilon_1^{k+1})
      \ar@{<-}[rrr]^-{f(x) \mapsto \sum_{n_1=0}^k \tfrac{1}{n!}f^{(n_1)}(x)\varepsilon_1^{n_1}}
      \ar[dd]_-{
          \sum_{n_1 = 0}^k \tfrac{1}{n_1!} g_{n_1}(x) \varepsilon_1^{n_1}
          \hfill
          \atop
             \mapsto
           \sum_{n_1,n_2 = 0 }^k \tfrac{1}{n_1! n_2!} g_{n_1}^{(n_2)}(0) \varepsilon_1^{n_1} \varepsilon_2^{n_2} }
     &&&
      C^\infty(\mathbb{R})
      \ar[ddlll]|{f(x) \mapsto \sum_{n=0}^k \tfrac{1}{n!} f^{(n)}(\varepsilon_1 + \varepsilon_2)^n}
      \\
      \\
      (\mathbb{R} \oplus \langle \varepsilon_1, \varepsilon_2\rangle)/(\varepsilon_1^{k+1},\varepsilon_2^{k+1})
        \ar@{<-}[rrr]_{\varepsilon \mapsto (\varepsilon_1 + \varepsilon_2)}
       &&&
      (\mathbb{R} \oplus \langle \varepsilon \rangle)/(\varepsilon^{k+1})
      \mathrlap{\,.}
       \ar@{<-}[uu]_{f(x) \mapsto \sum_{n=0}^k\tfrac{1}{n!}f^{(n)}(0) \varepsilon^{n} }
    }
  $$
  Here the top homomorphism is from the previous discussion, while the right morphism
  is the one that defines the inclusion $\mathbb{D}_0(k) \hookrightarrow \mathbb{R}$ by
  definition \ref{StandardInfinitesimalnDisk}.
  Hence the pushout of commutative $\mathbb{R}$-algebras simply identifies the tensor
  copy of $C^\infty(\mathbb{R})$ in the algebra at the top left with the algebra
  in the bottom right. This shows that the left morphism has to be as shown.
  To see that the composite of the top and the left morphism
  is the diagonal morphism as shown, use the exponential expression from before, in
  terms of which this identity is
  $$
    \exp(\varepsilon_2 \partial_x) \exp(\varepsilon_1 \partial_x) f(x)
    =
    \exp( ( \varepsilon_1 + \varepsilon_2) \partial_x ) f(x)
    \,.
  $$
  This implies that the bottom morphism is as shown, which is the claim to be proven.
\end{proof}

\subsection{Jet bundles}
\label{Jets}

We recall the traditional definition of jet bundles, and then show that
this is reproduced by abstract constructions in the Cahiers topos $\mathrm{FormalSmoothSet}$
and finally use this to give a general synthetic discussion of infinite-order jet bundles
in any differentially cohesive topos, via the abstract construction
right adjoint to that of formal disk bundles in the previous section
\ref{FormalDisks}.

A traditional arena for defining jets and jet
bundles is the following slight generalization of fiber bundles (this is
essentially the unnamed definition in \cite[section 1]{Marvan86}, see
also~\cite[sec.I.A]{Anderson-book}):
\begin{definition}[fibered manifold]
\label{FiberedManifold}
	We say a morphism $E \longrightarrow \Sigma$ of locally pro-manifolds
	(definition \ref{LocallyProManifold}) is a \emph{fibered manifold} if
	the morphism is a surjective submersion according to definition
	\ref{SubmersionBetweenLocallyProManifolds}. We write
  $$
    \mathrm{LocProMfd}_{\downarrow \Sigma}
      \hookrightarrow
    \mathrm{LocProMfd}_{/\Sigma}
  $$
  for the full subcategory of the slice category (definition \ref{SliceCategory})
  of locally pro-manifolds over $\Sigma$ on those objects which are fibered manifolds in this sense.
\end{definition}

\medskip

In generalization of example \ref{VectorFieldsAsMapsFromSpecOfRingOfDualNumbers} one says:
\begin{definition}[e.g.,\ {\cite[\textsection\textsection1.2,1.12]{Michor80}}]
 \label{JetsOfBundlesOfSmoothManifolds}
	Let $E \overset{p}{\longrightarrow} \Sigma$ be a morphism in
	$\mathrm{FormalSmoothSet}$ (thought of as a bundle over $\Sigma$, example
	\ref{SectionsOfBundles}). For $s \colon \ast \to \Sigma$ a point, and
	$k \in \mathbb{N}$, we have that a \emph{$k$-jet} of $p$ at $s$ is a morphism
	$\phi : \mathbb{D}_s(k) \longrightarrow E$ out of the infinitesimal
	disk of order $k$ around that point
	(definition~\ref{InfinitesimalDiskInManifold}) making the following
	diagram commute
  $$
    \begin{gathered}
    \xymatrix{
      \mathbb{D}_s(k)
      \ar@{_{(}->}[dr]
      \ar[rr]^\phi
      &&
      E
       \ar[dl]^p
      \\
      & \Sigma
    }
    \end{gathered}
    \,,
  $$
	where on the left we have the defining inclusion (from
	definition~\ref{InfinitesimalDiskInManifold}).

	The collection of all $k$-jets of $E$ naturally forms a formal smooth set
	$J^k_\Sigma E$, called the $k$-th \emph{jet bundle} of $E$.
	When $E \overset{p}{\longrightarrow} \Sigma$ is a fibered manifold (definition~\ref{FiberedManifold}),
	$J^k_\Sigma E$ is naturally also a smooth manifold fibered over
	$\Sigma$~\cite[Prop.6.2.5]{Saunders89} (the reference actually asks that
	$E \overset{p}{\longrightarrow} \Sigma$ is locally trivializable, but the
	provided atlas works equally well when $p$ is only a surjective submersion).
	When $E$ is a fibered manifold (definition~\ref{FiberedManifold}),
	then so is $J^k_\Sigma E$. As $k$-varies, these naturally form a
	projective system of smooth manifolds
  $$
    \cdots \longrightarrow J^2_\Sigma E \longrightarrow J^1_\Sigma E \longrightarrow J^0_\Sigma E = E
    \,.
  $$
\end{definition}

The following definition is equivalent to the standard definition in the
literature, see remark \ref{InfiniteJetBundleInTheLiterature} below.
\begin{definition}
  \label{InfiniteJetBundleOfManifolds}
	For $E \to \Sigma$ a fibered manifold (definition
	\ref{FiberedManifold}), its \emph{infinite jet bundle}
	$J^\infty_\Sigma E$ is the limit over the projective system of finite
	jet bundles, according to
	definition~\ref{JetsOfBundlesOfSmoothManifolds}, taken in the category
	of locally pro-manifolds (definition~\ref{LocallyProManifold}, remark
	\ref{LocProMfdFullyIncludedInFrechetManifolds}):
  $$
    J^\infty_\Sigma E
    :=
    \varprojlim_k J^k_\Sigma E
    \;\;\;\;
    \in \mathrm{LocProMfd}
    \,.
  $$
\end{definition}
\begin{remark}
  \label{InfiniteJetBundleInTheLiterature}
	In \cite{Saunders89} the infinite jet bundle $J^\infty_\Sigma E$ is
	considered more explicitly as the set of infinite jets equipped with
	the Fr{\'e}chet manifold structure locally modeled on
	$\mathbb{R}^\infty$, with $\mathbb{R}^\infty$ regarded as a
	Fr{\'e}chet manifold via the evident seminorms $\|-\|_n :
	\mathbb{R}^\infty \overset{}{\to} \mathbb{R}^n \overset{\|-\|}{\to}
	\mathbb{R}$ (definition~\ref{RInfinity}). But by
	proposition~\ref{ProjectiveLimitNatureOfRInfinity} this is
	equivalently the projective limit of the finite dimensional
	$\mathbb{R}^n$s formed in Fr{\'e}chet manifolds, hence (by remark
	\ref{LocProMfdFullyIncludedInFrechetManifolds}) in locally
	pro-manifolds:
  $$
    \mathbb{R}^\infty 
      \,\simeq\, 
    \varprojlim_k \mathbb{R}^k
    \;\;
    \in 
    \mathrm{LocProMfd} \hookrightarrow \mathrm{FrMfd}
    \,.
	$$
    Hence for the case that $E \to \Sigma$ is a fibered manifold (definition \ref{FiberedManifold}),
    definition~\ref{InfiniteJetBundleOfManifolds} reduces to what is considered in \cite{Saunders89}.

  %
Namely, the non-trivial statement to check is that the fibers of $J^\infty_\Sigma E$, formed as a projective limit, have the structure of $\mathbb{R}^\infty$ as above.
This is seen from the following iterated pullback diagram, formed over any chart $\mathbb{R}^n$ over which the fibered manifold $E \to \Sigma$ trivializes, using that (projective) limits are clearly preserved under pullback:
$$
  \begin{tikzcd}[
    row sep=12pt,
    column sep=15pt
  ]
    \mathbb{R}^n 
      \times 
    \mathbb{R}^m
      \times
    \mathbb{R}^\infty
    \ar[r]
    \ar[
      d,
      dotted
    ]
    \ar[
      dr,
      phantom,
      "{ \lrcorner }"{pos=.1}
    ]
    & 
    J^\infty_\Sigma E
    \ar[
      d,
      dotted
    ]
    \\
    \mathbb{R}^{n} 
     \times
    \mathbb{R}^{m}
    \times 
    \mathbb{R}^{m(n + n^2/2)}
    \ar[d]
    \ar[r]
    \ar[
      dr,
      phantom,
      "{ \lrcorner }"{pos=.1}
    ]
    & 
    J^2_\Sigma E
    \ar[d]
    \\
    \mathbb{R}^{n} 
     \times
    \mathbb{R}^{m}
    \times 
    \mathbb{R}^{m n}
    \ar[d]
    \ar[r]
    \ar[
      dr,
      phantom,
      "{ \lrcorner }"{pos=.1}
    ]
    & 
    J^1_\Sigma E
    \ar[d]
    \\
    \mathbb{R}^{n} 
     \times
    \mathbb{R}^{m}
    \ar[
      d
    ]
    \ar[
      r
    ]
    \ar[
      dr,
      phantom,
      "{ \lrcorner }"{pos=.1}
    ]
    &
    E
    \ar[d]
    \\
    \mathbb{R}^n
    \ar[
      r,
      hook
    ]
    &
     \Sigma
    \mathrlap{\,.}
  \end{tikzcd}
$$
    
\end{remark}
\begin{theorem}
  \label{ReproducingTraditionalJetBundle}
	Let $E \overset{p}{\to} \Sigma$ be a fibered manifold
	(definition~\ref{FiberedManifold}) with finite dimensional base $\Sigma$. Then under the embedding
	(proposition~\ref{SmoothSetsReceiveFrechetManifolds})
	$$
		\mathrm{LocProMfd} \overset{i}{\hookrightarrow} \mathbf{H}_\Re \overset{\iota}{\hookrightarrow} \mathbf{H}
		\, ,
	$$
	the infinite-jet bundle $J^\infty_\Sigma E$ in the sense of
	definition~\ref{InfiniteJetBundleOfManifolds} is isomorphic in
	$\mathbf{H}_{/\Sigma}$ to the right base change of $E$ (according to
	proposition~\ref{basechange}) along the unit $\eta_\Sigma : \Sigma \to
	\Im \Sigma$ of the $\Im$-monad
	(definition~\ref{ReImEtAdjunctionOnFormalSmoothSets}):
  $$
    J^\infty_\Sigma E
      \;\simeq\;
    (\eta_\Sigma)^\ast (\eta_\Sigma)_\ast E
    \;\;\;\;
    \in
    \mathrm{FormalSmoothSet}
    \,.
  $$
\end{theorem}
\begin{proof}
  By applying proposition~\ref{SliceSite} to definition~\ref{FormalSmoothSets}, we have an equivalence of categories
  $$
    \mathrm{FormalSmoothSet}_{/\Sigma}
      \simeq
    \mathrm{Sh}( \mathrm{FormalCartSp}/\Sigma )
    \,.
  $$
  Therefore it is sufficient to show that for every morphism $U \to \Sigma$
  out of a formal Cartesian space $U = \mathbb{R}^d \times \mathbb{D}$, there is a natural bijection
  $$
    \mathrm{Hom}_{\mathbf{H}_{/\Sigma}}\big(
      U
      ,\,
      (\eta_\Sigma)^\ast (\eta_\Sigma)_\ast E
    \big)
    \;\;
      \simeq
    \;\;
    \mathrm{Hom}_{\mathbf{H}_{/\Sigma}}\big(
      U
      \,,
      J^\infty_\Sigma E
    \big)
  $$
  (where we leave the maps to $\Sigma$ notationally implicit).
  Now by the base change adjoint triple $(\eta_\Sigma)_! \dashv (\eta_\Sigma)^\ast \dashv (\eta_\Sigma)_\ast$
  (proposition~\ref{basechange})
  the left hand side is equivalently
  $$
    \begin{aligned}
      \mathrm{Hom}_{\mathbf{H}_{/\Sigma}}\big(
        U
        ,\,
        (\eta_\Sigma)^\ast (\eta_\Sigma)_\ast E
      \big)
      &
      \simeq
      \mathrm{Hom}_{\mathbf{H}_{/\Sigma}}\big(
        (\eta_\Sigma)^\ast (\eta_\Sigma)_! U,\,
        E
      \big)
      \\
      & =
      \mathrm{Hom}_{\mathbf{H}_{/\Sigma}}\big(
        T^\infty_\Sigma U
        ,
        E
      \big)
      \mathrlap{\,,}
    \end{aligned}
  $$
  where we used definition~\ref{InfinitesimalDiskBundle} to identify the formal disk bundle
  $T^\infty_\Sigma U  = (\eta_\Sigma)^\ast (\eta_\Sigma)_! U = T^\infty \Sigma \times_\Sigma U$.

  Since the statement to be proven is local over $\Sigma$, we may assume without restriction
  of generality that $\Sigma \simeq \mathbb{R}^n$ (otherwise restrict to the charts of any atlas
  for the manifold $\Sigma$). By proposition~\ref{FormalDiskBundleLocallyTrivializableOverVManifold}
  we have that $T^\infty \Sigma \simeq \Sigma \times \mathbb{D}^n$ and hence (using definition~\ref{InfinitesimalDiskBundle})
  $$
    \begin{aligned}
      T^\infty_\Sigma U
        & \simeq
      U \times \mathbb{D}^n
        \\
      & \simeq
      U \times \varinjlim_k \mathbb{D}^n(k)
      \\
      & \simeq
      \varinjlim_k U \times \mathbb{D}^n(k)
      \mathrlap{\,,}
    \end{aligned}
  $$
  where in the second step we used proposition~\ref{FormalDiskInSmoothManifolds} and in the last step
  proposition~\ref{UniversalColimits}.
  Hence there is a natural bijection
  $$
    \begin{aligned}
      \mathrm{Hom}_{\mathbf{H}_{/\Sigma}}\big(
        U
        ,\,
        (\eta_\Sigma)^\ast (\eta_\Sigma)_\ast E
      \big)
      &
      \simeq
      \mathrm{Hom}_{\mathbf{H}_{/\Sigma}}\big(
        \varinjlim_k U \times \mathbb{D}^n(k)
        ,
        E
      \big)
      \\
      &\simeq
      \varprojlim_k
      \mathrm{Hom}_{\mathbf{H}_{/\Sigma}}\big(
         U \times \mathbb{D}^n(k)
        ,
        E
      \big)
      \\
      &\simeq
      \varprojlim_k
      \mathrm{Hom}_{\mathbf{H}_{/\Sigma}}\big(
         U
        ,
        \iota i J^k_\Sigma E
      \big)
      \\
      & \simeq
      \mathrm{Hom}_{\mathbf{H}_{/\Sigma}}\big(
         U
        ,\,
        \varprojlim_k
        \iota i J^k_\Sigma E
      \big)
      \mathrlap{\,,}
    \end{aligned}
  $$
	where in the second but last step we observed that a morphism $U \times
	\mathbb{D}^n(k)\to E$ over $\Sigma$ is precisely a smooth
	$U$-parameterized family of $k$-jets in $E$, according to
	definition~\ref{JetsOfBundlesOfSmoothManifolds}. 
	Now observing that $\varprojlim_k \iota i J^k_\Sigma E$ (formed in $\mathbf{H}$) is still reduced, in that
    $
      \varprojlim_k
      \iota i J^k_\Sigma E
      \simeq
      \iota
      \varprojlim_k
      i J^k_\Sigma E
    $
    (this is shown in \cite{GKSS2026}),
    and using again the fact
	that locally pro-manifolds embed fully faithfully, $\mathrm{LocProMfd}
	\overset{i}{\hookrightarrow} \mathbf{H}_\Re \overset{\iota}{\hookrightarrow} \mathbf{H}$
	(proposition~\ref{SmoothSetsReceiveFrechetManifolds}, remark \ref{LocProMfdFullyIncludedInFrechetManifolds}), we conclude as follows, now for $U = \mathbb{R}^d$ a plain Cartesian space:
   \begin{align*}
      \mathrm{Hom}_{\mathbf{H}_{/\Sigma}}\big(
         \iota i U
        ,\,
        \varprojlim_k
        \iota i J^k_\Sigma E
      \big)
      &
      \simeq
      \mathrm{Hom}_{\mathbf{H}_{/\Sigma}}\big(
         \iota i U
        ,\,
        \iota
        \varprojlim_k
        i J^k_\Sigma E
      \big)
      \\
      & \simeq
      \mathrm{Hom}_{\mathbf{H}_{\Re/\Sigma}}\big(
        i U
        ,\,
        \varprojlim_k
        i J^k_\Sigma E
      \big)
    \\
    &
    \simeq
    \varprojlim_k
    \mathrm{Hom}_{\mathbf{H}_{\Re /\Sigma}}\big(
      i U,
      i J^k_\Sigma E
    \big)
    \\
    &
    \simeq
    \varprojlim_k
    \mathrm{Hom}_{\mathrm{LocProMfd} /\Sigma}\big(
      U,
      J^k_\Sigma E
    \big)
    \\
    &
    \simeq
    \mathrm{Hom}_{\mathrm{LocProMfd} /\Sigma}\big(
      U,
      \varprojlim_k
      J^k_\Sigma E
    \big)
    \\
    &
    \simeq
    \mathrm{Hom}_{\mathrm{LocProMfd} /\Sigma}\big(
      U,
      J^\infty_\Sigma E
    \big)
    \\
    &
    \simeq
    \mathrm{Hom}_{\mathbf{H}_{\Re /\Sigma}}\big(
      i U,
      i J^\infty_\Sigma E
    \big)
    \\
    &
    \simeq
    \mathrm{Hom}_{\mathbf{H}_{/\Sigma}}\big(
      \iota i U,
      \iota i J^\infty_\Sigma E
    \big)
    \mathrlap{\,.}
    \qedhere
  \end{align*}
\end{proof}
We now see some of the power of the synthetic formalism. Classically,
jets are defined for fibered manifolds over $\Sigma$
(definition~\ref{InfiniteJetBundleInTheLiterature}), where locally
adapted coordinate charts enter the definition in an essential way. Above,
we have found a synthetic formalization of that construction. But now,
this construction of $J^\infty_\Sigma$ can be applied to the entire
slice topos $\mathbf{H}_{/\Sigma}$, which admits objects $E \to \Sigma$
where $E$ is a much more general space than a locally pro-manifold and
the map to $\Sigma$ no longer needs to be a surjective submersion.
That is, theorem~\ref{ReproducingTraditionalJetBundle} says that the
following abstract terminology makes good sense:
\begin{definition}
  \label{jetcomonad}
    For $\mathbf{H}$ any differentially cohesive topos,
    and for $\Sigma \in \mathbf{H}$, we say that
    the base change comonad (from proposition~\ref{basechange}) along the
    unit $\eta_\Sigma : \mathbf{H}_{/\Sigma} \to \mathbf{H}_{/\Sigma}$ of
    the $\Im$-monad (definition~\ref{ReImEtAdjunctionOnFormalSmoothSets})
    is the \emph{infinite jet bundle} operation $J^\infty_\Sigma$:
  $$
    J^\infty_\Sigma
      :=
    (\eta_\Sigma)^\ast \circ (\eta_\Sigma)_\ast
    \,.
  $$
	We denote its coproduct and counit natural transformations,
	respectively, by $\Delta\colon J^\oo_\Sigma \to J^\oo_\Sigma
	J^\oo_\Sigma$ and $\epsilon\colon J^\oo_\Sigma \to \mathrm{id}$.
	In other words (by proposition~\ref{basechange}) this is the right
	adjoint to the construction of formal disk bundles (according to
	definition~\ref{InfinitesimalDiskBundle}):
  $$
    T^\infty_\Sigma
      \;\dashv\;
    J^\infty_\Sigma
    \;:\;
    \mathbf{H}_{/\Sigma}
      \longrightarrow
    \mathbf{H}_{/\Sigma}
    \,.
  $$
	If $\sigma \in \mathrm{Hom}(T^\oo_\Sigma X, Y)$ and $\tau \in \mathrm{Hom}(X,
	J^\oo_\Sigma Y)$ are adjunct morphisms, in the sequel we will use the
	notation $\tau = \widetilde{\sigma}$ and $\sigma = \overline{\tau}$.
\end{definition}
That forming jets is right adjoint to forming infinitesimal disk bundles
was observed before in \cite[Proposition~2.2]{Kock80} (at least for jets of finite order).

\begin{proposition}
  \label{JetBundleEndoFunctorOnSuitableManifolds}
	For $\Sigma \in \mathrm{SmthMfd}$, the jet bundle functor from
	proposition~\ref{jetcomonad} restricts to an  endofunctor on the
	category of fibered manifolds $\mathrm{LocProMfd}_{\downarrow \Sigma}$
    (definition \ref{FiberedManifold}):
  $$
    J^\infty_\Sigma
     \;:\;
    \mathrm{LocProMfd}_{\downarrow \Sigma}
      \longrightarrow
    \mathrm{LocProMfd}_{\downarrow \Sigma}
    \,.
  $$
\end{proposition}

\begin{remark}
\label{ObserveThatThereIsACoMonadStructureOnJ}
As a direct corollary, theorem~\ref{ReproducingTraditionalJetBundle}
implies that the infinite jet bundle functor $J^\infty_\Sigma$ of
definition~\ref{jetcomonad} naturally carries the structure of a
comonad on $\mathrm{FormalSmoothSet}_{/\Sigma}$, and by proposition
\ref{JetBundleEndoFunctorOnSuitableManifolds} this restricts to a
comonad structure on the traditional jet bundle construction
on (infinite-dimensional) smooth manifolds.
\end{remark}

That there is a comonad structure on the traditional infinite jet
bundle construction has been observed before in~\cite{Marvan86} (and,
apparently, only there). We now unwind the structure morphisms in the
abstract comonad structure on $J^\infty_\Sigma$ and show that,
restricted to locally pro-manifolds, they indeed coincide with the
traditional definitions.

\begin{example}
  \label{JetBundleOfSigma}
  For any $\Sigma \in \mathbf{H}$, regarded trivially as a bundle over itself
  $$
    [\Sigma \stackrel{\mathrm{id}}{\longrightarrow} \Sigma]
    \;\;
    \in
    \mathbf{H}_{/\Sigma}
    \,,
  $$
  we have that the corresponding infinite jet bundle (according to definition~\ref{jetcomonad}) coincides with
  $\Sigma$, in that the unique morphism
  $$
    J^\infty_\Sigma \Sigma \overset{\simeq}{\longrightarrow} \Sigma
  $$
  is an isomorphism.
\end{example}
\begin{proof}
	In full generality, this follows from the fact that $\Sigma \in
	\mathbf{H}_{/\Sigma}$ (i.e.,\ $\mathrm{id}_\Sigma$) is a terminal
	object, and the fact that $J^\infty_\Sigma$ is a right adjoint, by
	definition~\ref{jetcomonad}, hence preserves terminal objects.

	Explicitly, if $\mathbf{H} = \mathrm{FormalSmoothSet}$ and $\Sigma \in \mathrm{SmthMfd} \hookrightarrow
	\mathrm{FormalSmoothSet}$ happens to be a smooth manifold, then
	definition~\ref{JetsOfBundlesOfSmoothManifolds} says that the jets
	of $\mathrm{id}_\Sigma$ all have to be constant.
\end{proof}

\begin{definition}
  \label{jetprolongation}
	For $E \overset{p}{\longrightarrow} \Sigma$ any morphism in
	$\mathbf{H}$, regarded as an object of $\mathbf{H}_{/\Sigma}$, and
	with $J^\infty_\Sigma E \in \mathbf{H}_{/\Sigma}$ denoting the
	corresponding infinite jet bundle according to
	definition~\ref{jetcomonad}, we say that \emph{jet prolongation} (or
	\emph{jet extension}) is the function
  $$
    j^\infty \;:\; \Gamma_\Sigma(E) \longrightarrow \Gamma_\Sigma(J^\infty_\Sigma E)
  $$
	from sections of $E$ (definition \ref{SectionsOfBundles}) to sections of
	$J^\infty_\Sigma E$ which is given by
  $$
    (\Sigma \overset{\sigma}{\longrightarrow} E)
    \;\mapsto\;
    ( j^\oo\sigma\colon \Sigma
        \overset{\simeq}{\longrightarrow}
      J^\infty_\Sigma \Sigma
        \overset{J^\infty_\Sigma \sigma}{\longrightarrow}
      J^\infty_\Sigma E
    )
    \,,
  $$
  where the equivalence on the right is that of example \ref{JetBundleOfSigma}.
\end{definition}
Notice that by the $(T^\infty_\Sigma \dashv J^\infty_\Sigma)$-adjunction,
	{$E'$-parameterized sections} of a jet bundle $J^\infty_X E$
	(definition \ref{jetcomonad}), namely morphisms of the form
  $$
    \sigma : E' \longrightarrow J^\infty_X E
  $$
  in $\mathbf{H}_{/X}$ (see example \ref{SectionsOfBundles}) are in
	natural bijection with morphisms of the form
  $$
    \overline{\sigma} : T^\infty_X E' \longrightarrow E
  $$
  in $\mathbf{H}_{/X}$.
  Now:
  \begin{proposition}
  \label{ParameterizedSectionsOfJetBundles}
	Let $X \in \mathbf{H}$ be any object, and let $E', E \in
	\mathbf{H}_{/X}$ be two bundles over $X$. Then the double $(T^\infty_\Sigma \dashv J^\infty_\Sigma)$-adjunct
  of a morphism of the form
  $$
    E' \longrightarrow J^\infty_X E
			\overset{\Delta_E}{\longrightarrow} J^\infty_X J^\infty_X E
    \,,
  $$
  hence of the image of an $E'$-parameterized section of the jet bundle under the jet
	coproduct operation, is the image of $\overline{\sigma}$
	under the product $\nabla_{E'}$ of the infinitesimal disk bundle monad, namely the morphism
  $$
    T^\infty_X T^\infty_X E'
			\overset{\nabla_{E'}}{\longrightarrow} T^\infty_X E
			\overset{\overline{\sigma}}{\longrightarrow} E
    \,.
  $$
\end{proposition}
\begin{proof}
 By proposition \ref{PropertiesOfAdjointPairsFromAdjointTriples}.
\end{proof}
The following two propositions show that the abstract concept of jet
prolongation in definition \ref{jetprolongation} reduces to the
traditional concept of jet prolongation on jets of sections of smooth
manifolds.
\begin{proposition}
  \label{JetProlongationReproduced}
	For $[E \longrightarrow \Sigma] \in \mathbf{H}_{/\Sigma}$ any slice in $\mathbf{H}$ (definition \ref{DifferentialCohesion}),
	the operation of jet
	prolongation from definition~\ref{jetprolongation}
  $$
    j^\infty :\Gamma_\Sigma(E) \longrightarrow \Gamma_\Sigma(J^\infty_\Sigma E)
  $$
	takes a smooth section $\sigma : \Sigma \to E$ (definition \ref{SectionsOfBundles}) to the section of the
	infinite jet bundle whose value $j^\infty\sigma(s)$ over any point $s \in \Sigma$ is the
	$(T^\oo_\Sigma \dashv J^\oo_\Sigma)$-adjunct (definition~\ref{jetcomonad}) infinite jet of $\sigma$ at $s$ (definition~\ref{JetsOfBundlesOfSmoothManifolds}):
  $$
    j^\infty\sigma(s)
      =
    \widetilde{\sigma|_{\mathbb{D}_s}}
    :
    \mathbb{D}_s \longrightarrow \Sigma \overset{\sigma}{\longrightarrow} E
    \,.
  $$
\end{proposition}
\begin{proof}
  By definition~\ref{jetprolongation} the value of $j^\infty\sigma$ over $s$ is
  $$
    j^\infty\sigma(s)
    :
    \{s\}
      \overset{}{\longrightarrow}
    \Sigma
      \simeq
    J^\infty_\Sigma \Sigma
      \overset{J^\infty_\Sigma \sigma}{\longrightarrow}
    J^\infty_\Sigma E
    \,.
  $$
	By the adjunction $(T^\infty_\Sigma \dashv J^\infty_\Sigma)$ of
	definition~\ref{jetcomonad}, this corresponds to
  $$
    T^\infty_\Sigma \{s\}
      \longrightarrow
    \Sigma
      \overset{\sigma}{\longrightarrow}
    E
    \,.
  $$
  By example \ref{FormalDiskBundleOfPoint} this is
  $$
    \mathbb{D}_s
      \longrightarrow
    \Sigma
      \overset{\sigma}{\longrightarrow}
    E
  $$
  as claimed.
\end{proof}
But this statement holds generally:
\begin{proposition}
  \label{JetProlongationInTermsOfEv}
  Let $\Sigma \in \mathbf{H}$ and $[E \to \Sigma] \in \mathbf{H}_{/\Sigma}$ be any objects,
  and let $\sigma : \Sigma \to E$ be any section. Then
  the $(T^\infty_\Sigma \dashv J^\infty_\Sigma)$-adjunct $\overline{j^\infty \sigma}$ (definition \ref{AdjointFunctor})
  of the jet prolongation $j^\infty \sigma : \Sigma \to J^\infty_\Sigma E$
  of $\sigma$ (definition \ref{JetProlongationReproduced}) is the composite
  $$
    \overline {j^\infty \sigma}
    \;:\;
    T^\infty \Sigma
      \overset{\mathrm{ev}}{\longrightarrow}
    \Sigma
      \overset{\sigma}{\longrightarrow}
    E
    \,,
  $$
  where $\mathrm{ev}$ is as in definition \ref{FormalDiskAbstractly}.
\end{proposition}
\begin{proof}
  We are to find the adjunct for the composite
  $$
    \xymatrix{
    \Sigma
      \ar[r]_\simeq^\phi
      &
    J^\infty_\Sigma \Sigma
      \ar[r]^{J^\infty_\Sigma \sigma}
      &
    J^\infty_\Sigma E
    }
    \,.
  $$
  By the naturality of forming adjuncts, this is
  $$
    T^\infty_\Sigma \Sigma
      \overset{\overline{\phi}}{\longrightarrow}
    \Sigma
      \overset{\sigma}{\longrightarrow}
    E
    \,,
  $$
  where $\overline{\phi}$ is the adjunct of $\phi$. By the formula for adjuncts (proposition \ref{AdjunctionInTermsOfUnitAndCounit}),
  the latter is
  $$
    \overline{\phi}
    :
    \xymatrix{
     T^\infty_\Sigma \Sigma
       \ar[rr]_{\simeq}^{T^\infty_\Sigma \phi}
       &&
     T^\infty_\Sigma J^\infty_\Sigma \Sigma
      \ar[rr]
     &&
     \Sigma
    }
    \,,
  $$
  where the first morphism is the canonical isomorphism (since $\phi$ is) and
  where the second morphism is the counit of the $(T^\infty_\Sigma \dashv J^\infty_\Sigma)$-adjunction, which is given by
  the units of the adjoint triple $((\eta_\Sigma)_! \dashv (\eta_\Sigma)^\ast \dashv (\eta_\Sigma)_\ast)$
  as the composite
  $$
    (\eta_\Sigma)^\ast (\eta_\Sigma)_! (\eta_\Sigma)^\ast (\eta_\Sigma)_\ast  \Sigma
      \overset{\tau}{\longrightarrow}
    (\eta_\Sigma)^\ast (\eta_\Sigma)_\ast  \Sigma
      \overset{\simeq}{\longrightarrow}
    \Sigma
    \,.
  $$
  The second morphism here is the isomorphism $J^\infty_\Sigma \Sigma \simeq \Sigma$ already used.
  So we are reduced to seeing that the first morphism here is $\tau = \mathrm{ev}$.

  To see that this is indeed the case, first notice that $(\eta_\Sigma)_\ast \Sigma \simeq \Im \Sigma$,
  since the right adjoint $(\eta_\Sigma)_\ast$ preserves terminal objects (which is also the origin of the
  very isomorphism just mentioned). Hence
  $(\eta_\Sigma)_! (\eta_\Sigma)^\ast (\eta_\Sigma)_\ast  \Sigma \to (\eta_\Sigma)_\ast \Sigma$ is the top morphism in
  the pullback square
  $$
    \xymatrix{
      \Sigma \ar[d]_\simeq  \ar@{}[dr]|{\mbox{\tiny (pb)}}  \ar[r]^{} & \Im \Sigma \ar[d]^{=}
      \\
      \Sigma \ar[r]_{\eta_\Sigma} & \Im \Sigma
    }
  $$
  which is again $\eta_\Sigma$. Hence the full morphism in question is the pullback of $\eta_\Sigma$
  along itself. This is either $\pi$ or $\mathrm{ev}$ (definition~\ref{InfinitesimalDiskBundle}), depending on the order of the arguments.
  Checking the commutativity of the pullback diagram, one finds that it is $\mathrm{ev}$.
\end{proof}

The following proposition shows that the abstractly induced comonad
structure on $J^\infty_\Sigma$
(remark~\ref{ObserveThatThereIsACoMonadStructureOnJ}) reduces on jets of
bundles of smooth manifolds to the comonad structure considered
in~\cite{Marvan86}.
\begin{proposition}
  \label{AbstractJetCoproductIsMichalsJetCoproduct}
	Consider $[E \longrightarrow \Sigma]$ in
	$\mathbf{H}_{/\Sigma}$. Then the comultiplication of the abstractly
	defined jet comonad $\Delta : J^\infty_\Sigma \to J^\infty_\Sigma
	J^\infty_\Sigma$ from definition~\ref{jetcomonad} takes a jet
	$j^\infty_\Sigma\sigma(s)$, written as a jet prolongation of some
	section $\sigma$, via definition~\ref{jetprolongation}, to the jet
	prolongation of that jet prolongation
  $$
    \Delta_E (j^\infty\sigma(s))
     =
    (j^\infty j^\infty\sigma)(s)
    \,.
  $$
	Hence if $\mathbf{H} = \mathrm{FormalSmoothSet}$ and $[E \overset{p}{\longrightarrow} \Sigma] \in
	\mathrm{LocProMfd}_{\downarrow \Sigma} \hookrightarrow
	\mathbf{H}_{/\Sigma}$ is a fibered manifold (definition
	\ref{FiberedManifold}) then this reduces to the jet coproduct
	considered in \cite[p.3]{Marvan86}.
\end{proposition}
\begin{proof}
  Under the adjunction $T^\infty_\Sigma
	\dashv J^\infty_\Sigma$ (definition~\ref{jetcomonad})
  $$
    j^\infty\sigma(s) \colon \{s\} \longrightarrow J^\infty_\Sigma E
  $$
  is represented by
  $$
    \overline{j^\infty\sigma}(s)
    \;:\;
    T^\infty_\Sigma \{s\}
      \simeq
    \mathbb{D}_s
      \hookrightarrow
    \Sigma
      \overset{\sigma}{\longrightarrow}
    E
    \,,
  $$
	where the isomorphism on the left is from example
	\ref{FormalDiskBundleOfPoint}. By proposition
	\ref{ParameterizedSectionsOfJetBundles}, its image
  $$
    \{s\} \longrightarrow J^\infty_\Sigma E \overset{\Delta_E}{\longrightarrow} J^\infty_\Sigma J^\infty_\Sigma E
  $$
	under the jet coproduct $\Delta_E$ corresponds dually to the
	pre-composition of this morphism with the product operation in the
	formal disk bundle monad (definition~\ref{InfinitesimalDiskBundle}):
  $$
    T^\infty_\Sigma T^\infty_\Sigma \{s\}
      \stackrel{\nabla_{\{s\}}}{\longrightarrow}
    T^\infty_\Sigma \{s\}
      \simeq
    \mathbb{D}_s
      \hookrightarrow
    \Sigma
      \overset{\sigma}{\longrightarrow}
    E
    \,.
  $$
	The $(T^\infty_\Sigma \dashv J^\infty_\Sigma)$-adjunct of this
	morphism is (by proposition \ref{AdjunctsInTermsOfUnitAndCounit}) the
	composite
  $$
    \xymatrix{
      \mathbb{D}_s
      \simeq
      T^\infty_\Sigma\{s\}
      \ar@/_1pc/[rr]_{\simeq}
       \ar[r]
       &
      J^\infty_\Sigma T^\infty_\Sigma T^\infty_\Sigma \{s\}
       \ar[r]
       &
      J^\infty_\Sigma T^\infty_\Sigma \{s\}
        \ar[rr]^-{J^\infty_\Sigma( \overline{j^\infty\sigma(s)} )}
        &&
      J^\infty_\Sigma E
    }
    \,,
  $$
	where the left composition has to be an isomorphism, as shown, because
	$J^\infty_\Sigma T^\infty_\Sigma \{s\} \simeq J^\infty_\Sigma
	\mathbb{D}_s \simeq \mathbb{D}_s \simeq T^\infty_\Sigma \{s\}$ (by
	example \ref{FormalDiskBundleOfPoint} and example
	\ref{JetBundleOfSigma}) and since in $\mathbf{H}_{/\Sigma}$ the object
	$\mathbb{D}_s$ is sub-terminal so that its only endomorphism is the
	identity. Hence this morphism is identified with $(j^\infty
	j^\infty\sigma)(s)$ by example \ref{jetprolongation}.
\end{proof}

\subsection{Differential operators}
\label{DifferentialOperators}

We recall the traditional concept of differential operators, formulated
in terms of jet bundles. Then we show that the category whose objects
are bundles over some $\Sigma$ and whose morphisms are differential
operators between their sections is equivalently the ``Kleisli
category'' of the jet comonad $J^\infty_\Sigma$ from definition
\ref{jetcomonad}. (This is a special case of the more general
identification of PDEs with coalgebras over the jet comonad, that we
turn to below in sections \ref{PDEs} and \ref{FormallyIntegrablePDEs}.)

Notice that throughout by ``differential operator'' we mean in general
non-linear differential operators.

\medskip

\begin{definition}[e.g.,\ {\cite[definition 6.2.22]{Saunders89}}]
	\label{DifferentialOperatorsAsMapsOutOfJetBundle}
	Consider two objects $E\to \Sigma$ and $F\to \Sigma$ in
	$\mathbf{H}_{/\Sigma}$.

	(a) A morphism $D \colon J^\infty_\Sigma E \to F$ induces, by
	composition with jet prolongation $j^\infty$ (definition
	\ref{jetprolongation}), a map $\hat D(\sigma) :=  D\circ
	j^\infty\sigma$ between spaces of sections (definition
	\ref{SectionsOfBundles}),
	$$
		\hat D \colon \Gamma_\Sigma(E) \to \Gamma_\Sigma(F) ,
	$$
	whose effect is illustrated by the following commutative diagram:
  $$
    \begin{gathered}
    \xymatrix{
      & E \ar[d]
      \\
      \Sigma \ar@{=}[r] \ar[ur]^{\sigma} & \Sigma
    }
    \end{gathered}
    \;\;\;\;\;
    \mapsto
    \;\;\;\;\;
    \begin{gathered}
    \xymatrix{
      & &
      J^\infty_\Sigma E
				\ar[d] \ar[r]^D &
			F
				\ar[dl]
				\ar[d]
      \\
      \Sigma \ar[r]^-{\simeq} &
      J^\infty_\Sigma \Sigma
				\ar@{=}[r]
				\ar[ur]^{j^\infty\sigma \simeq J^\infty_\Sigma \sigma} &
			J^\infty_\Sigma \Sigma &
			\Sigma
				\ar[l]_-{\simeq}
				\ar@/_1pc/[u]_{\hat{D}(\sigma)}
		}
    \end{gathered}
    \, .
  $$
	(The equivalences on the right are those from example
	\ref{JetBundleOfSigma}.)

	(b) Conversely, a map between spaces of sections $B \colon
	\Gamma_\Sigma(E) \to \Gamma_\Sigma(F)$ is called a \emph{differential
	operator} if it factors through the jet bundle this way. That is $B =
	\hat{D}$ for some morphism $D\colon J^\oo_\Sigma E \to F$. Naturally,
	the morphism $D$ is also referred to as a differential operator, or as
	a \emph{formal differential operator} when more precision is needed.

	The formal differential operator $D$ corresponding to a differential
	operator $B = \hat D$ is clearly unique, which we will denote by $D =
	\underline{B}$.
\end{definition}
\begin{remark}
Definition~\ref{DifferentialOperatorsAsMapsOutOfJetBundle} is the
standard formalization of the informal notion of a differential operator
that usually goes along with the notation $\hat{D}[\sigma](x) =
D(x,\sigma(x), \partial\sigma(x), \partial\partial \sigma(x), \ldots )$.
\end{remark}
\begin{definition}
	\label{InfiniteProlongation}
	Given a formal differential operator (definition
	\ref{DifferentialOperatorsAsMapsOutOfJetBundle})
   $$
     D : J^\infty_\Sigma E \longrightarrow F
     \, ,
   $$
   the composite
   $$
     p^\infty D
       :=
     J^\infty_\Sigma E
       \stackrel{\Delta_E}{\longrightarrow}
     J^\infty_\Sigma J^\infty_\Sigma E
       \overset{J^\infty_\Sigma D}{\longrightarrow}
     J^\infty_\Sigma F
   $$
	is called its \emph{infinite prolongation}. It is characterized by the
	relation
  $$
    \widehat{p^\oo D}[\sigma] = j^\oo \hat{D}[\sigma]
  $$
  for any section $\sigma\in \Gamma_\Sigma(E)$.
\end{definition}
\begin{proposition}
  \label{FormalDifferentialOperatorsComposition}
	For $D_1 \colon J^\infty_\Sigma E_1 \longrightarrow E_2$ and $D_2
	\colon J^\infty_\Sigma E_2 \longrightarrow E_3$ two formal
	differential operators according to definition
	\ref{DifferentialOperatorsAsMapsOutOfJetBundle}, the formal
	differential operator corresponding to the direct composite of the
	corresponding maps on spaces of sections (definition
	\ref{SectionsOfBundles})
  $$
    \hat D_2 \circ \hat D_1
       \colon ~
     \Gamma_\Sigma(E_1)
       \overset{\hat D_1}{\longrightarrow}
     \Gamma_\Sigma(E_2)
       \overset{\hat D_2}{\longrightarrow}
     \Gamma_\Sigma(E_3)
  $$
  is the composite
  \begin{align*}
  	\underline{\hat{D}_2 \circ \hat{D}_1} =
    D_2 \circ p^\infty D_1
    &\colon ~
    J^\infty_\Sigma E_1
      \overset{p^\infty D_1}{\longrightarrow}
    J^\infty_\Sigma E_2
      \overset{D_2}{\longrightarrow}
    E_3 \,,
  \end{align*}
	(with $p^\infty D_1$ the infinite prolongation of $D_1$ from
	definition \ref{InfiniteProlongation}) of $D_2$ with the image under
	$J^\infty_\Sigma$ of $D_1$ and with the coproduct $\Delta$ of the jet
	comonad.

  In terms of infinite prolongations this is
  $$
    p^\infty
    \left(
      \underline{\hat{D}_2 \circ \hat{D}_1}
    \right)
    =
    (p^\infty D_2) \circ (p^\infty D_1)
  $$
\end{proposition}
\begin{proof}
	Regarding the first statement: Applying the formula in definition
	\ref{DifferentialOperatorsAsMapsOutOfJetBundle} twice gives that, for
	$\sigma \in \Gamma_\Sigma(E_1)$ any section, we have that $(\hat D_2 \circ
	\hat D_1)(\sigma)$ is the section given by the total top horizontal
	composite in the following diagram:
  $$
    \begin{gathered}
    \xymatrix{
       \Sigma \ar[r]^-{ \simeq}
       \ar[dr]_-\simeq
       & J^\infty_\Sigma J^\infty_\Sigma \Sigma
       \ar[rr]^-{J^\infty_\Sigma J^\infty_\Sigma \sigma}
       &&
       J^\infty_\Sigma J^\infty_\Sigma E_1
       \ar[rr]^-{J^\infty_\Sigma D_1}
       &&
       J^\infty_\Sigma E_2
       \ar[rr]^-{D_2}
       &&
       E_3
       \\
       & J^\infty_\Sigma \Sigma
       \ar[rr]_-{J^\infty_\Sigma \sigma}
       \ar[u]_{\Delta_\Sigma}^\simeq
       &&
       J^\infty_\Sigma E_1
       \ar[u]^{\Delta_{J^\oo_\Sigma E_1}}
       \ar@/_1pc/[rrrru]_{D_2 \circ p^\infty D_1}
       \ar[rru]|{p^\infty D_1}
    }
    \end{gathered}
    \, .
  $$
	Here the square commutes by the naturality of the coproduct $\Delta$
	of the jet comonad, and the equivalences on the left are the unique
	ones from example \ref{JetBundleOfSigma}. Hence the total top
	composite is equivalent to the total bottom composite, which is the
	formula to be proven.

  Regarding the second statement, consider the diagram
  $$
    \begin{gathered}
    \xymatrix{
      J^\infty_\Sigma E_1
       \ar[rr]^-{\Delta_{E_1}}
       \ar@/^1.7pc/[rrrr]^{p^\infty D_1}
       \ar[d]_{\Delta_{E_1}}
       &&
      J^\infty_\Sigma J^\infty_\Sigma E_1
       \ar[rr]^{J^\infty D_1}
       \ar[d]|{\Delta_{J^\infty_\Sigma E_1}}
       &&
      J^\infty_\Sigma E_2
        \ar[d]|{\Delta_{E_2}}
       \ar@/^1.7pc/[rrd]^{p^\infty D_2}
       \\
       J^\infty_\Sigma J^\infty_\Sigma E_1
       \ar[rr]_{J^\infty_\Sigma (\Delta_{E_1})}
       \ar@/_1.7pc/[rrrr]|{J^\infty_\Sigma p^\infty D_1}
       \ar@/_2.7pc/[rrrrrr]_{J^\infty_\Sigma (\underline{\hat{D}_2 \circ \hat{D}_1})}
       &&
       J^\infty_\Sigma J^\infty_\Sigma J^\infty_\Sigma E_1
       \ar[rr]_{J^\infty_\Sigma J^\infty_\Sigma D_1}
        &&
      J^\infty_\Sigma J^\infty_\Sigma E_2
        \ar[rr]_{J^\infty_\Sigma D_2}
        &&
      J^\infty E_3
    }
    \end{gathered}
    \, .
  $$
	Here both squares are naturality squares of the coproduct $\Delta$.
	The total top composite is $p^\infty D_2 \circ p^\infty D_1$, while
	the total left and bottom composite is $p^\infty(\underline{\hat{D}_2
	\circ \hat{D}_1})$.
\end{proof}
\begin{definition}
  \label{DifferentialOperatorsCategory}
	For $\Sigma$ a smooth manifold, we write
	$\mathrm{DiffOp}_{\downarrow\Sigma}(\mathrm{LocProMfd})$ for the
	category whose objects are fibered manifolds over $\Sigma$ (definition
	\ref{FiberedManifold}) and whose morphisms are differential operators
	between their spaces of sections, according to definition
	\ref{DifferentialOperatorsAsMapsOutOfJetBundle}.
\end{definition}

\begin{proposition}
	\label{BundlesDiffOpsAndPDEs}
	There is an equivalence of categories
   $$
     \mathrm{DiffOp}_{\downarrow\Sigma}(\mathrm{LocProMfd})
     \simeq
		 \mathrm{Kl}(J^\infty_\Sigma|_{\mathrm{LocProMfd}_{\downarrow \Sigma}})
   $$
	 between the category of differential operators over $\Sigma$
	 (definition \ref{DifferentialOperatorsCategory}) and the co-Kleisli
	 category $\mathrm{Kl} (J^\infty_\Sigma)$ of the jet comonad
	 $J^\infty_\Sigma$ restricted to fibered manifolds (according to
	 proposition \ref{JetBundleEndoFunctorOnSuitableManifolds}).
\end{proposition}
\begin{proof}
This is a direct consequence of definition~\ref{coKleisli} and
remark~\ref{compositionalaKleisli}, together with
definition~\ref{DifferentialOperatorsAsMapsOutOfJetBundle} and
proposition \ref{FormalDifferentialOperatorsComposition}.
\end{proof}
\begin{remark}
	\label{DifferentialOperatorGeneralized}
	Proposition \ref{BundlesDiffOpsAndPDEs} implies that it makes sense to
	regard the full Kleisli category of the $J^\oo_\Sigma$ comonad on all
	of $\mathbf{H}_{/\Sigma}$ as the category of differential operators on
	sections of all objects of $\mathbf{H}_{/\Sigma}$, which we will
	denote by $\mathrm{DiffOp}_{/\Sigma}(\mathbf{H}) :=
	\mathrm{Kl}(J^\infty_\Sigma)$. Hence a general morphism of the form
	$$
		D : J^\infty_\Sigma E \longrightarrow F
	$$
	in $\mathbf{H}_{/\Sigma}$ may be thought of as a \emph{generalized
	differential operator} (the generalization being that $E$ and $F$ may
	both be far from having the structure of fibered manifolds).
\end{remark}

\subsection{Partial differential equations}
\label{PDEs}

We now give a general synthetic definition of partial differential
equations (PDEs) in a differentially cohesive topos $\mathbf{H}$
(definition \ref{DifferentialCohesion} below). Throughout this section
we make detailed comments on how to connect this synthetic definition to
traditional concepts.

We are concerned with the geometric perspective on partial differential
equations, from the point of view of jet bundles, sometimes known as the
\emph{formal theory of PDEs}. This has a long history going back to the
works of Riquier, E.~Cartan, Janet and many others, with the influential
ideas of Spencer~\cite{Spencer} and Vinogradov~\cite{VinKras}
(eventually developed by them and many others) being essential for its
modern formulation. The idea here is that, just as with ordinary
(algebraic) equations, which we can put in correspondence with the loci
of solutions that they carve out inside the domain of definition of
their variables, so we may think of partial differential equations as
being embodied by the sub-loci which they carve out inside spaces of
partial derivatives, hence inside jet bundles:
\begin{definition}[generalized PDEs]
  \label{generalpdes}
	Let $\Sigma \in \mathbf{H}$ be a $V$-manifold
	(definition~\ref{VManifold}), and let $Y \in \mathbf{H}_{/\Sigma}$ be
	an object in the slice over $\Sigma$ (thought of as a bundle, example
	\ref{SectionsOfBundles}). Then a \emph{generalized partial
	differential equation} (PDE) on the space of sections of $Y$ is an
	object $\mathcal{E} \in \mathbf{H}_{/\Sigma}$ together with a
	monomorphism into the infinite jet bundle of $Y$
	(definition~\ref{jetcomonad}),
	$$
		\mathcal{E}
		  \hookrightarrow
		J^\infty_\Sigma Y
       \,.
	$$
    Note that we are not excluding the possibility
	that $Y = \mathcal{E}$.
	We omit the \emph{generalized} attribute when $\Sigma$ is a smooth
	finite-dimensional manifold and $\mathcal{E} \in
	\mathrm{LocProMfd}_{\downarrow\Sigma} \hookrightarrow
	\mathrm{FormalSmoothSet}_{/\Sigma} = \mathbf{H}_{/\Sigma}$ is a
	fibered manifold (definition~\ref{FiberedManifold}).

	We say that a section $\sigma \in \Gamma_{\Sigma}(Y)$ (definition
	\ref{SectionsOfBundles}) is a \emph{solution} to this partial
	differential equation if its jet prolongation $j^\infty_\Sigma \sigma$
	(definition \ref{jetprolongation}) factors through $\mathcal{E}$,
	i.e.,\ if there exists the dashed morphism $s$ making the following
	diagram commute:
	$$
    \begin{gathered}
		\xymatrix{
			\mathcal{E}
				\ar@{^{(}->}[r]
			&
			J^\infty_\Sigma Y
			\ar[d]
			\\
			\Sigma \ar@{-->}[u]^{s}
				\ar@{=}[r]
				\ar[ur]|{j^\infty\sigma}  & \Sigma
		}
    \end{gathered}
    \, .
	$$
	We write
	$$
		\mathrm{Sol}_{\Sigma}(\mathcal{E})
			\hookrightarrow
		\Gamma_\Sigma(Y)
	$$
	for the subset of solutions to the PDE, inside the set of sections
	(example \ref{SectionsOfBundles}).
\end{definition}

Below in prop. \ref{CoAlgebraMorphismsAsSolutions} we give a comonadic characterization of 
  the solutions $\mathrm{Sol}_{\Sigma}(\mathcal{E})$.
But before we get to that, we now pass attention to \emph{formal theory of PDEs}.
The reason that it is referred to as \emph{formal} is that in it center
stage is taken not by 
  actual
solutions (in the sense of
definition~\ref{generalpdes} above), but \emph{formal solutions} (definition \ref{FamilyOfFormal} below),
essentially formal Taylor series that satisfy the PDE order by order.

The concept of formal solutions is natural since jets do not see past the
``formal power series horizon.''
While formal solutions have been used as a guiding heuristic in
the traditional literature on PDEs, their treatment has
traditionally been only informal (meaning heuristic). What here allows us to
work with formal solutions directly is the flexibility of synthetic differential geometry
in the guise of differential cohesion, which makes formal infinitesimal neighbourhoods
in a manifold a reality.
As the central role of families of formal solutions becomes apparent below, the
ability to work with them directly clearly becomes an advantage of the
synthetic formalization of PDEs.

Before precisely stating definition~\ref{FamilyOfFormal} below, let us
recall the traditional notion of formal solution. Working in local
coordinates $(y^i)$ centered at a point $x\in \Sigma \in \Mfd$, the jet
$j^\infty\sigma(x)$ of a section $\sigma\colon \Sigma \to Y$ can of
course be identified with its Taylor series $\sigma(x+y) = \sigma(x) +
\sigma_i(x) y^i + \frac{1}{2!} \sigma_{ij}(x) y^i y^j + \cdots$. One
could be tempted to directly call it a formal solution at $x$ if
$j^\infty\sigma(x) \in e_Y(\mathcal{E})$ for a PDE embedding $e_Y\colon
\mathcal{E} \hookrightarrow J^\infty_\Sigma Y$. But if $\sigma$ is a
local solution at $x$, that is, $j^\infty\sigma(y) \in e_Y(\mathcal{E})$
for all $y$ in a neighborhood of $x$ in $\Sigma$, it is much stronger
property satisfied by $j^\infty\sigma(x+y)$ that takes the local
behavior of $e_Y \colon \mathcal{E} \to J^\infty_\Sigma Y$ into account
already at the local level, and it remains a non-trivial condition
directly on $j^\infty\sigma(x)$ when the local neighborhood at $x$ is
shrunk to an infinitesimal one. Namely, replacing $y$ by $y+z$, the
Taylor series can be rearranged into
$$
	\sigma(x+y+z)
	= (\sigma(x) + \sigma_i(x) z^i + \frac{1}{2!} \sigma_{ij}(x) z^i z^j + \cdots)
	+ \left(\sigma_i(x) + \sigma_{ik}(x) z^k + \cdots\right) y^i
	+ \frac{1}{2!} \left(\sigma_{ij}(x) + \cdots\right) y^i y^j
	+ \cdots ,
$$
where the parenthesized series are respectively the Taylor series of the
higher derivatives $\sigma_{i\cdots l}(x+z)$, which altogether give the
Taylor expansion of thet jet $j^\infty\sigma(x+z)$. The extra condition
to be satisfied by $j^\infty\sigma(x)$ to be called a formal solution at
$x$ is that the corresponding $j^\infty\sigma(x+z)$ Taylor series
remains in the Taylor expanded image of the embedding $e_Y(\mathcal{E})
\subset J^\infty_\Sigma Y$. This informal idea precise below in
definition~\ref{FamilyOfFormal}, with the added complication that we
directly define families of formal solutions, parametrized by some other
bundle $E\to \Sigma$, which will be necessary to define the prolongation
and formal interability of PDEs in definition~\ref{FormallyIntegrable}. 

See remark~\ref{InterpretationOfHolonomicSections} and
example~\ref{RecoveringTraditionalFormalSolutions} below for how to
further interpret the abstract monadic definition \ref{FamilyOfFormal}.
In order to make this work in full generality it invokes
parameterization of families of formal solutions by infinitesimal disk
bundles as above in proposition \ref{ParameterizedSectionsOfJetBundles}.
Below in example~\ref{RecoveringTraditionalFormalSolutions} we show that
this general definition reduces pointwise to the expected one. In
remark~\ref{GeneralizedVsTraditionalPDE} we give pointers to the
literature where the traditional notions are discussed in a way that
inspired our generalizations.

\begin{definition}[formal solutions of PDEs]
	\label{FamilyOfFormal}
	Let $\Sigma$ be a $V$-manifold (definition
	\ref{ParameterizedSectionsOfJetBundles}). Consider a generalized PDE
	$\mathcal{E} \hookrightarrow J^\oo_\Sigma Y$ over $\Sigma$ according
	to definition \ref{generalpdes} and a further parameter object $E \in
	\mathbf{H}_{/\Sigma}$. Then:

	(a)
	An \emph{($E$-parametrized) family of formally holonomic sections} of
	$J^\oo_\Sigma Y$ is a morphism
	$$
		\sigma_E : T^\oo_\Sigma E \longrightarrow J^\oo_\Sigma Y
	$$
	from the infinitesimal disk bundle of $E$ (definition
	\ref{InfinitesimalDiskBundle}) such that its $(T^\oo_\Sigma \dashv
	J^\oo_\Sigma)$-adjunct $\widetilde{\sigma_E} : E \to J^\infty_\Sigma
	J^\infty_\Sigma Y$ (definition \ref{AdjointFunctor}) makes the
	following diagram commute:
	$$
    \begin{gathered}
		\xymatrix{
			T^\oo_\Sigma E
              \ar[r]^-{\sigma_E}_-{\ }="s"
            &
            J^\oo_\Sigma Y
               \ar[d]^{\Delta_Y}
            \\
			E \ar[u]^{\eta_E}
              \ar[r]_-{\widetilde{\sigma_E}}^-{\ }="t"
            &
            J^\oo_\Sigma J^\oo_\Sigma Y
		}
    \end{gathered}
		\, ,
	$$
	or equivalently that its $(T^\infty_\Sigma \dashv
	J^\infty_\Sigma)$-adjunct ``to the other side'' $\overline{\sigma_E} :
	T^\infty_\Sigma T^\infty_\Sigma E \to Y$ makes the following diagram
	commute:
	$$
    \begin{gathered}
		\xymatrix{
			T^\oo_\Sigma T^\oo_\Sigma E
               \ar[d]_{\nabla_E}
               \ar[r]^-{\overline{\sigma_E}}_{\ }="t" & Y\\
			T^\oo_\Sigma E
               \ar[r]_{\sigma_E}^{\ }="s"
            &
            J^\oo_\Sigma Y \ar[u]_{\epsilon_Y}
		}
    \end{gathered}
		\, .
	$$

	(b)
	An \emph{($E$-parametrized) family of formal solutions} of
	$\mathcal{E} \hookrightarrow J^\oo_\Sigma Y$ is a morphism
	$$
		s_E : T^\oo_\Sigma E \longrightarrow \mathcal{E}
	$$
	such that the composite
	$$
		\sigma_E : T^\oo_\Sigma E \overset{s_E}{\longrightarrow} \mathcal{E} \hookrightarrow J^\oo_\Sigma Y
	$$
	is an $E$-parametrized family of formally holonomic sections of
	$J^\oo_\Sigma Y$.
\end{definition}
\begin{remark}
  \label{InterpretationOfHolonomicSections}
	By corollary \ref{ComponentVersionOfDiskBundleMonadProduct} the two
	conditions in definition \ref{FamilyOfFormal} for a morphism
	${\sigma_E}$ to be a formally holonomic section $\sigma_E$ mean in
	terms of local generalized elements (definition
	\ref{ComponentNotationForInfinitesimalDiskBundles}) equivalently that
	$\overline{\sigma_E}(x,a,b)$ depends ``symmetrically'' on its
	infinitesimal arguments $a$ and $b$: The first condition says that
  $$
    \widetilde{\sigma}_E(x)(a,b) = \sigma_E(x,0)(a+b)
  $$
  hence equivalently
  $$
    \sigma_E(x,a)(b) = \sigma_E(x,0)(a+b)
    \,,
  $$
  while the second similarly states that
  $$
    \overline{\sigma_E}(x,a,b) = \sigma_E(x,a+b)(0)
  $$
  hence equivalently
  $$
    \sigma_E(x,a)(b) = \sigma_E(x,a+b)(0)
    \,.
  $$
  This also shows that the two conditions are indeed equivalent:
  The first condition implies the second by replacing $a \mapsto a + b$ and $b \mapsto 0$,
  while the second implies the first by replacing $a \mapsto 0$ and $b \mapsto a + b$.
\end{remark}

\begin{example}[traditional formal solutions]
	\label{RecoveringTraditionalFormalSolutions}
	Let $\mathbf{H} = \mathrm{FormalSmoothSet}$ be the Cahiers topos (definition \ref{FormalSmoothSet})
	and let $\Sigma$ be an ordinary smooth manifold, according to example \ref{SmoothManifoldsAreVManifolds}.
	Let $E = {*}$ be the abstract point, with the slice morphism $x\colon {*} \to \Sigma$
	picking some point in $\Sigma$. Recall, by example
	\ref{FormalDiskBundleOfPoint}, that its formal neighborhood bundle
	$T^\oo_\Sigma \{*\} \to \Sigma$ is the inclusion $\mathbb{D}_x \to
	\Sigma$ of the formal neighborhood of the point $x$.

	Then, a ${*}$-parametrized family of formal solutions of $\mathcal{E}
	\hookrightarrow J^\oo_\Sigma Y$, in the general sense of definition
	\ref{FamilyOfFormal}, is precisely what is usually known as a
	\emph{formal solution at $x\in \Sigma$}, namely a formal power series
	$T^\oo_\Sigma \{*\} \simeq \mathbb{D}_x \stackrel{\sigma}{\to} Y$ at
	$x \in \Sigma$ valued in $Y$, such that its jet extension
	$T^\oo_\Sigma \{*\} \simeq \mathbb{D}_x
	\stackrel{j^\oo\sigma}{\hookrightarrow} J^\oo_\Sigma Y$ factors
	through the solution locus $\mathcal{E} \hookrightarrow J^\oo_\Sigma
	Y$.
\end{example}
\begin{proof}
	Consider a ${*}$-parametrized family of formal solutions $T^\oo_\Sigma
	\{*\} \stackrel{\sigma^\oo}{\longrightarrow} \mathcal{E}$ and the two
	commutative diagrams equivalently exhibiting this property
	(definition~\ref{FamilyOfFormal}):
	$$
    \begin{gathered}
		\xymatrix{
			& \mathcal{E} \ar@{^{(}->}[d] \\
			T^\oo_\Sigma \{*\} \simeq \mathbb{D}_x
				\ar[r]^-{j^\oo\sigma}
				\ar[ur]^{\sigma^\oo} &
			J^\oo_\Sigma Y \ar[d]^{\Delta_Y}
			&
			T^\oo_\Sigma T^\oo_\Sigma \{*\} \ar[d]_{\nabla_\ast}
				\ar[r]^{\overline{j^\oo\sigma}} &
			Y \\
			{*}
               \ar[u]^{\eta_\ast}
               \ar[ur]^-{\tilde \sigma}
               \ar[r]_-{\widetilde{j^\oo\sigma}} &
			J^\oo_\Sigma J^\oo_\Sigma Y
			&
			T^\oo_\Sigma \{*\} \simeq \mathbb{D}_x
				\ar[r]_-{j^\infty \sigma}
                \ar[dr]_{\sigma^\oo}
				\ar[ur]^-{\sigma} &
			J^\oo_\Sigma Y \ar[u]_{\epsilon_Y} \\
			& & & \mathcal{E} \ar@{_{(}->}[u]
		}
    \end{gathered}
		\, .
	$$
	By remark~\ref{InterpretationOfHolonomicSections}, we can recover the
	formal solution $\sigma^\oo$ from the knowledge of the diagonal
	morphism in the diagram on the right, which we have denoted by
	$\sigma$. By proposition \ref{JetProlongationReproduced} the top
	horizontal morphism on the right is the $(T^\infty_\Sigma \dashv
	J^\infty_\Sigma)$-adjunct $\overline{j^\infty \sigma}$ of the jet
	prolongation $j^\infty \sigma $ (definition~\ref{jetprolongation}).
	This identifies the bottom horizontal morphism on the right and hence
	the top horizontal morphism on the left with the jet prolongation
	$j^\infty \sigma$ itself, as shown. Now the commutativity of the
	triangle diagrams shows that $j^\oo\sigma$ factors as $\mathbb{D}_x
	\stackrel{\sigma^\oo}{\longrightarrow} \mathcal{E} \to J^\oo_\Sigma
	Y$, which concludes the proof in one direction.

	In the other direction, suppose that we have a section $\sigma \colon
	T^\oo_\Sigma\{*\} \to Y$, such that $j^\oo\sigma$ factors through
	$\mathcal{E}$. Then, by
	remark~\ref{InterpretationOfHolonomicSections}, we recover a
	commutative diagram like the one above on the right and hence a
	$*$-parametrized family of formal solutions.
\end{proof}

\begin{example}
	\label{SolFactorsThroughFormalSol}
	Given a section $\sigma \colon \Sigma \to Y$, it is a solution in
	the sense of definition~\ref{generalpdes}, precisely if the
    $(T^\infty_\Sigma \dashv J^\infty_\Sigma)$-adjunct
    $\sigma^\infty : T^\infty \Sigma \to J^\infty_\Sigma Y$
    (definition \ref{AdjointFunctor})
    of its
	double jet prolongation $j^\oo j^\oo\sigma \colon \Sigma \to
	J^\oo_\Sigma J^\oo_\Sigma Y$ (definition \ref{jetprolongation}), is a
	$\Sigma$-parametrized family of formal solutions
    according to definition \ref{FamilyOfFormal}.
\end{example}
\begin{proof}
	By proposition \ref{JetProlongationInTermsOfEv}, an equivalent formula
	for $\sigma^\oo$ is $\sigma^\oo\colon T^\oo_\Sigma \Sigma
	\stackrel{\mathrm{ev}}{\longrightarrow} \Sigma
	\stackrel{j^\oo\sigma}{\longrightarrow} J^\oo_\Sigma Y$.

	Then, consider the diagram
	$$
    \begin{gathered}
		\xymatrix{
			& \mathcal{E} \ar@{^{(}->}[d] \\
			T^\oo_\Sigma \Sigma
				\ar[r]^{\sigma^\oo}
				\ar@{-->}[ur]
				&
			J^\oo_\Sigma Y
				\ar[d]^{\Delta_Y}
			\\
			\Sigma
				\ar[u]^{\eta_\Sigma}
				\ar[ur]^-{j^\oo \sigma}
				\ar[r]_-{j^\oo j^\oo \sigma} &
			J^\oo_\Sigma J^\oo_\Sigma Y
		}
    \end{gathered}
		\, ,
	$$
	where we do not a priori know whether the dashed morphism exists. By
	the definition of $\sigma^\oo$ and the identity $j^\oo j^\oo \sigma =
	\Delta_Y \circ j^\oo \sigma$
	(proposition~\ref{AbstractJetCoproductIsMichalsJetCoproduct}), the solid
	arrows are known to commute. Hence, $\sigma^\oo$ is a
	$\Sigma$-parametrized family of formally holonomic sections. If
	$\sigma^\oo$ happens to be a family of formal solutions, then the
	dashed arrow exists and the whole diagram commutes, implying that
	$j^\oo\sigma$ factors through $\mathcal{E} \hookrightarrow
	J^\oo_\Sigma Y$, hence an
    %
      actual
    solution. On the other hand, if $\sigma$
	is a 
    %
      actual
    solution, then $j^\oo\sigma$ factors through $\mathcal{E}
	\hookrightarrow J^\oo_\Sigma Y$. Precomposing this factorization with
	$T^\oo_\Sigma \Sigma \to \Sigma$ then shows that $\sigma^\oo$ also
	factors through $\mathcal{E} \hookrightarrow J^\oo_\Sigma Y$, meaning
	that the dashed morphism exists and commutes with the rest of the
	diagram, or in other words that $\sigma^\oo$ is a family of formal
	solutions.
\end{proof}

We close this section with remarks on how to connect the
above synthetic formalization to traditional concepts.

\begin{remark}
	\label{GeneralizedVsTraditionalPDE}
	In the case that $\mathbf{H} = \mathrm{FormalSmoothSet}$ is the Cahiers topos (definition \ref{FormalSmoothSet})
    and when  $\mathcal{E}$ and solutions $\sigma\colon \Sigma \to \mathcal{E}$
	belong to $\mathrm{LocProMfd}_{\downarrow \Sigma} \hookrightarrow
	\mathbf{H}_{/\Sigma}$ (definition~\ref{FiberedManifold}), the above
	definitions coincide with the one given by
	Vinogradov~\cite{Vinogradov86,VinKras,Marvan86}. See for instance the
	discussion of formal solutions in~\cite[\S V.2]{Vinogradov86}.
	A slightly more traditional
	definition~\cite{Goldschmidt,GuPf13} restricts Vinogradov's notion of
	a PDE on the sections of $Y\to \Sigma$ by the extra requirement that
	image of $\mathcal{E} \hookrightarrow J^\oo_\Sigma Y$ is a closed
	submanifold; sometimes the restriction that the composition with the
	natural projection $\mathcal{E} \hookrightarrow J^\oo_\Sigma Y \to Y$
	is surjective is also invoked. It should be noted that, by ignoring
	the latter requirements, Vinogradov's notion of PDE admits also
	examples that have been called \emph{partial differential
	relations}~\cite{Gromov}, which may be specified by inequalities
	(rather than equalities) between differential operators.
\end{remark}

\begin{remark}
  \label{PDEsRegardedAsSubobjectsOfJetBundles}
	In a category of sheaves such as the slice topos
	$\mathbf{H}_{/\Sigma}$, the difference between equations and
	inequalities gets blurred. There, every monomorphism is a
	\emph{regular monomorphism} (prop. \ref{RegularMonomorphismInSheafTopos}), which means that every subobject is
	characterized by an equation, in that every subobject inclusion like
	$\mathcal{E} \hookrightarrow J^\infty_\Sigma Y$ is part of an
	equalizer diagram of the form
	$$
		\xymatrix{
			\mathcal{E}
				\ar@{^{(}->}[rr]
			&&
			J^\infty_\Sigma Y
				\ar@<+3pt>[rr]^{D_1}
				\ar@<-3pt>[rr]_{D_2}
			&&
			Z
		}
	$$
	for some object $Y$ and some morphisms $D_1, D_2$ in
	$\mathbf{H}_{/\Sigma}$. Since by remark
	\ref{DifferentialOperatorGeneralized} we may view the morphisms $D_1$
	and $D_2$ above as generalized differential operators, this exhibits
	every PDE in the general sense of definition \ref{generalpdes} as the
	locus carved out by an equation between differential operators.
	However, even if the inclusion $\mathcal{E} \hookrightarrow
	J^\oo_\Sigma Y$ lives in $\mathrm{LocProMfd}_{\downarrow \Sigma}
	\hookrightarrow \mathbf{H}_{/\Sigma}$, the target object $Z$ of the
	corresponding morphisms $D_1$, $D_2$ may be quite far from a standard
	smooth manifold. For instance, the inclusion $(0,1) \hookrightarrow
	\mathbb{R}$ is the equalizer of two maps $\mathbb{R} \mapsto
	\mathbb{R} \sqcup_{(0,1)} \mathbb{R}$ into the non-Hausdorff space
	constructed by glueing two copies of $\mathbb{R}$ along the $(0,1)$
	subinterval.
\end{remark}

\begin{remark}
It is worth dwelling a bit on the condition that a generalized PDE
$\mathcal{E} \hookrightarrow J^\oo_\Sigma Y$ be included in a jet bundle
by a monomorphism. In our definition, this is a monomorphism in the slice
category $\mathbf{H}_{/\Sigma}$. As will be seen below, no further
conditions on this morphism will be necessary. On the other hand,
although we have a full inclusion of the sub-category fibered manifolds,
$\mathrm{LocProMfd}_{\downarrow \Sigma} \hookrightarrow
\mathbf{H}_{/\Sigma}$, the notion of an
$\mathbf{H}_{/\Sigma}$-monomorphism is strictly stronger than a
monomorphism in fibered manifolds. Since Marvan was working in fibered
manifolds~\cite{Marvan86,Marvan89}, he had to require an additional
condition%
	\footnote{This condition was originally and erroneously omitted from
	Proposition~1.4 of~\cite{Marvan86}, which was later corrected in
	Theorem~1.3 of~\cite{Marvan89}. The condition itself can be seen as a
	strengthened version of being an \emph{immersion} or as a weakened
	version of being \emph{transversal}.}%
: the inclusion monomorphism needed to remain a monomorphism under the
$V$-functor (vertical tangent bundle functor). The main consequence of
this extra hypothesis is that such monomorphisms are then also preserved
by the $J^\oo_\Sigma$ functor. In our case,
$\mathbf{H}_{/\Sigma}$-monomorphisms (and even more generally all
limits) are preserved by $J^\oo_\Sigma$ because it is defined as a right
adjoint (definition~\ref{jetcomonad}, proposition~\ref{rightadjointpreserveslimits}).
\end{remark}

\subsection{Formally integrable PDEs}
\label{FormallyIntegrablePDEs}
A PDE is to be called \emph{integrable} if given a solution jet at one
point, it may be extended to a jet prolongation of a local section on an
open neighbourhood of that point. Accordingly, a PDE is to be called
\emph{formally integrable} if given a solution jet at one point, it may
be extended to a jet prolongation of a section defined at least on a
\emph{formal infinitesimal neighbourhood} of that point. We now give an
abstract synthetic definition (not actually referring to points) of this
concept of formally integrable PDEs, this is definition
\ref{FormallyIntegrable} below.

Then we prove in this generality that the category formally integrable
PDEs in a differentially cohesive topos is equivalent to the category of
coalgebras over the jet-comonad. This is theorem \ref{PDEIsEM} below. For
the special case that $\mathbf{H} = \mathrm{FormalSmoothSet}$ is Dubuc's
Cahiers topos (definition \ref{FormalSmoothSets}), and that $\mathcal{C}
= \mathrm{LocProMfd}_{\downarrow \Sigma}$ is the category of locally
pro-manifolds fibered over an ordinary manifold $\Sigma$ (definition
\ref{FiberedManifold}) we recover $\mathrm{PDE}(\mathcal{C})
\hookrightarrow \mathrm{PDE}_{/\Sigma}(\mathbf{H})$ as the image of the
embedding of Vinogradov's category of PDEs over $\Sigma$, since by the
main result of~\cite{Marvan86,Marvan89} they are both equivalent to the
category of coalgebras over the jet-comonad in fibered manifolds
(corollary~\ref{GeneralizedPDEVsVinogradovCorollary}).

Finally we use this equivalence in order to discuss finite limits in the
category of formally integrable PDEs (theorem~\ref{PDEHasKernels} and
corollary~\ref{PDEHasFiniteLimits} below). This is expected to be crucial (in
future work) for the discussion of variational (Euler-Lagrange) PDEs,
which characterize the vanishing locus (hence a certain kernel) of
certain PDE morphisms, more precisely given by variational derivatives
of Lagrangian densities.

\begin{definition}[formally integrable PDEs]
	\label{FormallyIntegrable}
	Consider $\mathcal{E}, Y \in \mathbf{H}_{/\Sigma}$ and a generalized
	PDE $e_Y\colon \mathcal{E} \hookrightarrow J^\oo_\Sigma Y$ (definition
	\ref{generalpdes}).

	(a)
  Its \emph{(infinite) prolongation} is the generalized PDE denoted by
  the $\mathbf{H}_{/\Sigma}$-monomorphism
  $$
    e^\oo_Y\colon \mathcal{E}^\infty \hookrightarrow J^\infty_\Sigma Y
  $$
  and defined by the pullback square%
		\footnote{Note that, because $e^\oo_Y$ and $\Delta_Y$ are
		monomorphisms, the rest of the morphisms in the diagram are also
		monomorphisms, since $J^\oo_\Sigma$ preserves monomorphisms
		(proposition~\ref{rightadjointpreserveslimits}) and pullbacks of
		monomorphisms are monomorphisms.} %
	$$
    \begin{gathered}
		\xymatrix{
			\mathcal{E}^\oo
				\ar[rr]^{e^\oo_Y}
				\ar[d]^{\rho^\oo_\mathcal{E}}
                \ar@{}[drr]|{\mbox{\tiny (pb)}}
                &&
			J^\oo_\Sigma Y
				\ar[d]_{\Delta_Y}
			\\
			J^\oo_\Sigma \mathcal{E}
				\ar[rr]_{J^\oo_\Sigma e_Y}^{\ }="t"
            &&
			J^\oo_\Sigma J^\oo_\Sigma Y
		}
    \end{gathered}
		\,.
	$$
	We call the $\mathbf{H}_{/\Sigma}$-monomorphism $e^\oo_\mathcal{E} =
	\rho^\oo_\mathcal{E}\circ \epsilon_\mathcal{E}\colon \mathcal{E}^\oo
	\hookrightarrow \mathcal{E}$ the \emph{canonical inclusion} of the
	prolongation in the original PDE.

	(b) If the canonical inclusion is in fact an isomorphism,
	$e^\oo_\mathcal{E}\colon \mathcal{E}^\oo
	\stackrel{\simeq}{\longrightarrow} \mathcal{E}$, we say that the
	generalized PDE $\mathcal{E} \hookrightarrow J^\oo_\Sigma Y$ is
	\emph{formally integrable}.

	(c) If $e'_Y\colon \mathcal{E}' \hookrightarrow J^\oo_\Sigma Y'$ is
	another generalized PDE, then an $\mathbf{H}_{/\Sigma}$-morphism
	$\phi\colon \mathcal{E} \to \mathcal{E}'$ is said to \emph{preserve
	formal solutions} if for any parametrized family of formal solutions
	$\sigma^\oo\colon T^\oo_\Sigma E \to \mathcal{E}$, the composition
	$\phi\circ \sigma^\oo \colon T^\oo_\Sigma E \to \mathcal{E}'$ is still
	a parametrized family of formal solutions.

	(d) Denote by $\mathrm{PDE}_{/\Sigma}(\mathbf{H})$ the category whose
	objects are formally integrable generalized PDEs $\mathcal{E}
	\hookrightarrow J^\oo_\Sigma Y$, with $\mathcal{E}, Y \in
	\mathbf{H}_{/\Sigma}$, and whose morphisms are
	$\mathbf{H}_{/\Sigma}$-morphisms $\phi\colon \mathcal{E} \to
	\mathcal{E}'$ that preserve formal solutions.
\end{definition}

\begin{remark}
	(a)
	Definition~\ref{FormallyIntegrable}(a) says that the prolongation of a
	PDE locus $\mathcal{E} \hookrightarrow J^\oo_\Sigma Y$
	(definition~\ref{generalpdes}) yields another PDE locus
	$\mathcal{E}^\infty \hookrightarrow \mathcal{E} \hookrightarrow
	J^\oo_\Sigma Y$ that is \emph{smaller or equal} to the original one.
	In the sense of remark \ref{PDEsRegardedAsSubobjectsOfJetBundles} this
	means that there are in general \emph{more}/\emph{stronger} equations
	characterizing $\mathcal{E}^\infty$ than there are equations
	characterizing $\mathcal{E}$, and it is in this sense that the
	differential equation is \emph{prolonged} as we pass to
	$\mathcal{E}^\infty$.

	(b) Combining the definition of the prolongation $\mathcal{E}^\oo$ of
	a PDE with example~\ref{SolFactorsThroughFormalSol} shows that both
	$\mathcal{E}$ and its prolongation have the same
    %
      actual
    solutions,
	because by definition they have the same formal solutions. But it is
	important to note that dealing with formal solutions is indispensable
	in this context. There are examples of PDEs that have lots of formal
	solutions, but none that can be extended even to local
	sections~\cite{Lewy} (see~\cite{Zworski} for an insightful
	exposition). Thus, if we were to define the prolongation
	$\mathcal{E}^\oo$ of $\mathcal{E}$ as the largest sub-object that
	admits the same 
    %
      actual
    solutions, it would be empty in those cases.
	However, under such a definition the prolongation would be much more
	difficult to compute, because jets do not see past the ``horizon'' of
	formal solutions.

	(c) Finally, a comment on the term \emph{formally integrable}. One way
	construct a formal solution $\sigma^\oo \colon T^\oo_\Sigma(*) \simeq
	\mathbb{D}_* \to \mathcal{E}$ is as a limit $\sigma^\oo =
	\varprojlim_k \sigma^k$, where $\mathbb{D}_* = \varinjlim_k
	\mathbb{D}_*(k)$ and $\sigma^k \colon \mathbb{D}_*(k) \to \mathcal{E}$
	are formal solutions of finite orders. The construction would proceed
	inductively, by starting with a $\sigma^0\colon {*} \to \mathcal{E}$
	and successively constructing higher orders. However, if the
	$\sigma^0$ does not factor through the canonical inclusion
	$\mathcal{E}^\oo \hookrightarrow \mathcal{E}$, this inductive
	procedure will be obstructed at some higher order and it is said that
	such a finite order formal solution is \emph{not integrable}. Thus,
	another way to interpret $\mathcal{E}^\oo$ is as the largest
	sub-object of $\mathcal{E}$ that consists of $0$-th order formal
	solutions that are integrable
	(lemma~\ref{ProlongationIsUniversalFamily} makes this more precise by
	proving that $\mathcal{E}^\oo$ parametrizes a universal family of
	formal solutions). These ideas have lead to a finite order version of
	the notion of formally integrable (for instance, as expressed in
	Definition~7.2 of~\cite{Goldschmidt}), which when extended to infinite
	order coincides with our definition~\ref{FormallyIntegrable}.
\end{remark}

\begin{remark}
	\label{RemarksOnPDECategory}
	(a) Consider any full subcategory $\mathcal{C} \hookrightarrow
	\mathbf{H}_{/\Sigma}$ such that we can restrict $J^\oo_\Sigma \colon
	\mathcal{C} \to \mathcal{C}$. One example is the
	subcategory $\mathcal{C} = \mathrm{LocProMfd}_{\downarrow
	\Sigma}$ of locally pro-manifolds fibered over $\Sigma$
	(definition~\ref{FiberedManifold}).

	(b) We can then denote by $\mathrm{PDE}(\mathcal{C})
	\hookrightarrow \mathrm{PDE}_{/\Sigma}(\mathbf{H}) =:
	\mathrm{PDE}(\mathbf{H}_{/\Sigma})$ the full subcategory
	where $\mathcal{E} \hookrightarrow J^\oo_\Sigma Y$ and $\phi\colon
	\mathcal{E} \to \mathcal{E}'$ are all morphisms in
	$\mathcal{C}$. It is important to note that the definition of
	$\mathrm{PDE}(\mathcal{C})$ still refers to the larger category
	$\mathbf{H}_{/\Sigma}$, which is needed to check the
	$\mathbf{H}_{/\Sigma}$-monomorphism property of $\mathcal{E}
	\hookrightarrow J^\oo_\Sigma$ as well as to introduce families of
	formal solutions (for instance, $\mathcal{C}$ itself might not
	admit any manifolds with infinitesimal dimensions), to verify that
	they are preserved under morphisms and to check formal integrability.

	(c) The example of $\mathrm{PDE}_{\downarrow
	\Sigma}(\mathrm{LocProMfd}) :=
	\mathrm{PDE}(\mathrm{LocProMfd}_{\downarrow \Sigma})$ was explicitly
	considered in~\cite{Marvan86,Marvan89} where it was identified with a
	subcategory of Vinogradov's category of
	PDEs~\cite{Vinogradov80,VinKras}. Strictly speaking, morphisms in
	Vinogradov's category are defined differently: instead of preserving
	formal solutions they preserve an alternative geometric structure (the
	\emph{Cartan distribution}). On the other hand, informal comments in
	the literature (such as those at the top of~\cite[\textsection
	8.6]{Vinogradov80}) indicate that our definition is the intuitively
	preferred one. However, since infinitesimals are not strictly part of
	classical differential geometry, the difficulty of describing formal
	solutions in a precise and concise way has made the Cartan
	distribution a favored proxy in the existing literature. We will not
	go into the details here of how the Cartan distribution relates to
	formal solutions. It suffices to note that our definition of the PDE
	category ultimately coincides with Vinogradov's
	(remark~\ref{GeneralizedPDEVsVinogradov} and
	corollary~\ref{GeneralizedPDEVsVinogradovCorollary}), at least in the
	$\mathrm{PDE}_{\downarrow \Sigma}(\mathrm{LocProMfd})$ case. Later, in
	remark~\ref{CoAlgebraMorphismsAsSolutions}, we will comment on how
	identifying the coalgebra structure underlying a PDE could be
	interpreted as putting the PDE in ``canonical form.''
\end{remark}

Next we prove some lemmas containing the main technical results needed
to establish our main theorem~\ref{PDEIsEM} below on the structure of
the category of generalized PDEs.

\begin{lemma}
	\label{CompositionOfFormalFamilies}
	Consider $E, E', \mathcal{E}, \in \mathbf{H}_{/\Sigma}$, a generalized
	PDE $e_Y\colon \mathcal{E} \hookrightarrow J^\oo_\Sigma Y$ (definition
	\ref{generalpdes}) and an $E$-parametrized family of formal solutions
	$s_E\colon T^\oo_\Sigma E \to \mathcal{E}$. Then, for any morphism
	$\phi\colon E' \to E$, the composite $s_E \circ T^\oo_\Sigma \phi
	\colon T^\oo_\Sigma E' \to \mathcal{E}$ is also an $E'$-parametrized
	family of formal solutions.
\end{lemma}
\begin{proof}
	The only property that we need to check is that the composition $e_Y
	\circ s_E \circ T^\oo_\Sigma \phi \colon T^\oo_\Sigma E'
	\to J^\oo_\Sigma Y$ is an $E'$-parametrized family of formally
	holonomic sections. This property is precisely captured by the
	commutativity of the outer part of the following diagram:
	$$
    \begin{gathered}
		\xymatrix{
			T^\oo_\Sigma E'
				\ar@/^2pc/@{-->}[rrrr]^-{\ }
				\ar[rr]^{T^\oo_\Sigma \phi} &&
			T^\oo_\Sigma E
				\ar[rr]^{e_Y\circ s_E}="s" &&
			J^\oo_\Sigma Y
				\ar[d]^{\Delta_Y}
			\\
			E'
				\ar@/_2pc/@{-->}[rrrr]_{\widetilde{(e_Y\circ s_E \circ T^\oo_\Sigma \phi)}}
				\ar[u]^{\eta_{E'}}
				\ar[rr]_\phi &&
			E
				\ar[u]^{\eta_E}
				\ar[rr]_-{ \widetilde{e_Y\circ s_E} }="t" &&
			J^\oo_\Sigma J^\oo_\Sigma Y
		}
    \end{gathered}
		\, .
	$$
	Here the square on the right commutes by hypothesis; the left square
	commutes by naturality of the unit of the monad $T^\oo_\Sigma$; while
	the top triangle commutes by definition. Finally, the bottom triangle
	commutes by the natural hom-isomorphism of the $(T^\oo_\Sigma \dashv
	J^\oo_\Sigma)$-adjunction, which  provides us with a bijection between
	the following kinds of commutative triangles:
	$$
    \begin{gathered}
		\xymatrix{
			E'
				\ar[d]_{\phi}
				\ar[dr] &
			{} \ar@{}[d]^{\ }="s"
			&&
			T^\oo_\Sigma E'
				\ar[d]_{T^\oo_\Sigma \phi}="t"
				\ar[dr] &
			\\
			E
				\ar[r] &
			J^\oo_\Sigma Z
			&&
			T^\oo_\Sigma E
				\ar[r] &
			Z
			\ar@{}|{\leftrightarrow} "s"; "t"
		}
    \end{gathered}
    \, .
	$$
\end{proof}

\begin{proposition}
  \label{InputForPDEisEM}
	Consider a formally integrable generalized PDE $e_Y \colon
	\mathcal{E} \hookrightarrow J^\oo_\Sigma Y$ (definition \ref{FormallyIntegrable}).

	(a) Using the canonical inclusion isomorphism, $e^\oo_\mathcal{E}
	\colon \mathcal{E}^\oo \stackrel{\simeq}{\longrightarrow}
	\mathcal{E}$, a formally integrable PDE parametrizes its own universal
	family of formal solutions, $\overline{\rho_\mathcal{E}}\colon
	T^\oo_\Sigma\mathcal{E} \to \mathcal{E}$, where $\rho_\mathcal{E} =
	\rho^\oo_\mathcal{E} \circ e^\oo_\mathcal{E}$
	(cf.~lemma~\ref{ProlongationIsUniversalFamily}).

	(b) The adjunct morphism $\rho_\mathcal{E} \colon \mathcal{E}
	\hookrightarrow J^\oo_\Sigma \mathcal{E}$ to
	$\overline{\rho_\mathcal{E}}$ defines a coalgebra structure
	(definition \ref{EMOverComonad}) over the jet comonad $J^\oo_\Sigma$.

	(c) Conversely, any coalgebra structure $\rho_\mathcal{E}\colon \mathcal{E}
	\hookrightarrow J^\oo_\Sigma \mathcal{E}$ is also a formally
	integrable generalized PDE.

	(d) For $e'_Y\colon \mathcal{E}' \hookrightarrow J^\oo Y'$ another
	formally integrable generalized PDE, an
	$\mathbf{H}_{/\Sigma}$-morphism $\phi \colon \mathcal{E} \to
	\mathcal{E}'$ preserves formal solutions iff it is a morphism of
	corresponding coalgebras, meaning that it fits into the commutative
	diagram
	$$
    \begin{gathered}
		\xymatrix{
			\mathcal{E} \ar[r]^{\phi} \ar[d]_{\rho_\mathcal{E}} &
			\mathcal{E}' \ar[d]^{\rho_{\mathcal{E}'}} \\
			J^\oo_\Sigma \mathcal{E} \ar[r]_{J^\oo_\Sigma \phi} &
			J^\oo_\Sigma \mathcal{E}'
		}
    \end{gathered}
		\, .
	$$
\end{proposition}
\begin{proof}
	(a)
	This is just a restatement (placed here for convenience) of the result
	of lemma~\ref{ProlongationIsUniversalFamily} in light of the
	definition of formal integrability
	(definition~\ref{FormallyIntegrable}(b)).

	(b)
	First to see the counitality condition, let us identify
	$\mathcal{E}^\oo \simeq \mathcal{E}$ and $e_Y \simeq e_Y^\infty$ by
	formal integrability, and consider the diagram
	$$
  \begin{gathered}
	\xymatrix{
		\mathcal{E}
			\ar@{^{(}->}[r]^{e_Y}="s"
			\ar[d]_{\rho_\mathcal{E}}
			\ar@{}[dr]|{\mbox{\tiny (pb)}}
			&
		J^\infty_\Sigma Y \ar[d]|{\Delta_Y}
	  	\ar@/^2pc/[dd]^{\mathrm{id}}
	  \\
	  J^\infty_\Sigma \mathcal{E}
	  	\ar@{^{(}->}[r]_-{J^\infty_\Sigma e_Y}
	  	\ar[d]_{\epsilon_\mathcal{E}}
	  	&
		J^\infty_\Sigma J^\infty_\Sigma Y
	  	\ar[d]|{\epsilon_{J^\infty_\Sigma Y}}
	  	\\
		\mathcal{E}
			\ar@{^{(}->}[r]_{e_Y}="s"
			&
		J^\infty_\Sigma Y
	}
  \end{gathered}
  \, .
	$$
	Here the top square is the defining pullback square of the formally
	integrable PDE, while the bottom square is the naturality square for
	the counit of the jet comonad. The right vertical composite is the
	identity by counitality of the coproduct $\Delta$. A consequence of
	the commutativity of the whole diagram is the factorization identity
	$e_Y \circ (\epsilon_\mathcal{E} \circ \rho_\mathcal{E}) = e_Y$. But,
	since $e_Y$ is a monomorphism and $e_Y \circ \mathrm{id} = e_Y$ is
	another factorization, the two factorizing morphisms must be the same,
  $$
    \epsilon_\mathcal{E}\circ \rho_\mathcal{E}
      =
    \mathrm{id}
    \,,
  $$
	which is the counitality condition for a coalgebra structure
	$\rho_\mathcal{E}$. It also means that $\rho_\mathcal{E}$ is a split
	monomorphism with retraction morphism $\epsilon_\mathcal{E}$.

  Now, to see the coaction property, consider the diagram
  $$
    \xymatrix{
      J^\infty_\Sigma J^\infty_\Sigma \mathcal{E}
      \ar[rr]^-{J^\infty_\Sigma J^\infty_\Sigma e_Y}
      &&
      J^\infty_\Sigma J^\infty_\Sigma J^\infty_\Sigma Y
      &&
      J^\infty_\Sigma J^\infty_\Sigma Y
      \ar[ll]_-{J^\oo_\Sigma\Delta_Y}
      \\
      J^\infty_\Sigma \mathcal{E}
      \ar[rr]_-{J^\infty_\Sigma e_Y}
      \ar[u]_{\Delta_\mathcal{E}}
      &&
      J^\infty_\Sigma J^\infty_\Sigma Y
      \ar[u]|{\Delta_{J^\infty_\Sigma Y}}
      &&
      J^\infty_\Sigma Y
      \ar[u]_{\Delta_Y}
      \ar[ll]^-{\Delta_Y}
    }
  $$
	made from naturality squares of the coproduct $\Delta$ (it consists
	of two faces of the ``commuting cube'' diagram
	in~\cite[\textsection\textsection2.3--4]{Marvan86}). Since
	$J^\infty_\Sigma$ is a right adjoint it preserves fiber products. By
	definition of a formally integrable PDE, the fiber product over the
	bottom arrows is $\mathcal{E}$, and since $J^\oo_\Sigma$ preserves
	limits the fiber product over the top morphisms is $J^\infty_\Sigma
	\mathcal{E}$. By the universal property of the latter fiber product,
	there is hence a unique dashed morphism $e$ making the two rectangles
	in the following diagram commute:
	$$
    \begin{gathered}
		\xymatrix{
			J^\infty_\Sigma \mathcal{E} &&
			J^\infty_\Sigma J^\infty_\Sigma \mathcal{E}
				\ar[ll]|-{J^\oo_\Sigma \epsilon_\mathcal{E}} &
			J^\infty_\Sigma \mathcal{E}
				\ar[l]^-{J^\infty_\Sigma \rho_\mathcal{E}}
				\ar[r]^-{J^\infty_\Sigma e_Y}
				\ar@/_1.5pc/[lll]_{J^\oo_\Sigma \mathrm{id} = \mathrm{id}}
		            &
			J^\infty_\Sigma J^\infty_\Sigma Y
			\\
			&& J^\infty_\Sigma \mathcal{E}
				\ar[ull]^{\mathrm{id}}
				\ar[u]_{\Delta_\mathcal{E}} &
			\mathcal{E}
				\ar[l]^{\rho_\mathcal{E}}
				\ar@{-->}[u]^{e} \ar[r]_{e_Y} &
			J^\infty_\Sigma Y
				\ar[u]_{\Delta_Y}
		}
    \end{gathered}
		\,,
	$$
	where we have also added the triangle on the left, which commutes by
	the counitality of $\Delta$, and have also noted that the two
	left-most horizontal morphisms compose to the identity, by the
	counitality of $\rho_\mathcal{E}$ just established above. Hence, from
	commutativity of the left part of this diagram, we can conclude that
	$$
	   e = \rho_\mathcal{E}
	   \,.
	$$
	With this identification, the commutativity of the left square is
	exactly the coaction property of $\rho_\mathcal{E}$.

	(c)
	Suppose that $\rho_\mathcal{E} \colon \mathcal{E} \to J^\oo_\Sigma
	\mathcal{E}$ gives $\mathcal{E}$ a coalgebra structure. Then, by the
	counitality condition, $\epsilon_\mathcal{E} \circ \rho_\mathcal{E} =
	\mathrm{id}$, meaning that $\rho_\mathcal{E}$ is a split monomorphism
	(in $\mathbf{H}_{/\Sigma}$ in either case), in the same way as we
	concluded in part (a). Now, the Beck equalizer theorem
	(proposition~\ref{BeckEqualizer}) says that $\rho_\mathcal{E}$ is an
	equalizer of the diagram
	$$
		\xymatrix{
			\mathcal{E}
				\ar@{^{(}->}[rr]^-{\rho_\mathcal{E}}
				&&
			J^\oo_\Sigma \mathcal{E}
        \ar@<+3pt>[rr]^-{\Delta_\mathcal{E}}
        \ar@<-3pt>[rr]_-{J^\oo_\Sigma \rho_\mathcal{E}}
        &&
			J^\oo_\Sigma J^\oo_\Sigma \mathcal{E}
		} \, .
	$$
	Recall the canonical inclusion $e^\oo_\mathcal{E} \colon \mathcal{E}^\oo
	\hookrightarrow \mathcal{E}$ (definition~\ref{FormallyIntegrable}) of
	the prolongation $\mathcal{E}^\oo$ of $\mathcal{E}$. The pullback
	diagram defining $\mathcal{E}^\oo$
	(lemma~\ref{ProlongationIsUniversalFamily}) can be combined with
	the Beck equalizer diagram to give the following commutative diagram:
	$$
    \begin{gathered}
		\xymatrix{
			\mathcal{E}
				\ar@/^1pc/@{^{(}->}[drrr]^{\rho_\mathcal{E}}
				\ar@/_1pc/@{_{(}->}[ddr]_{\rho_\mathcal{E}}
				\ar@<-3pt>@{^{(}-->}[dr]
				\ar@<+4pt>@{<-^{)}}[dr]^{e^\oo_\mathcal{E}}
			\\
		  & \mathcal{E}^\oo
		    \ar@{^{(}->}[rr]
		    \ar[d]
				\ar@{}[drr]|{\mbox{\tiny (pb)}}
		  &&
		  J^\infty_\Sigma \mathcal{E}
		  	\ar[d]^{\Delta_\mathcal{E}}
		  \\
		  & J^\infty_\Sigma \mathcal{E}
		    \ar@{^{(}->}[rr]_{J^\infty_\Sigma \rho_{\mathcal{E}}}
		  &&
		  J^\infty_\Sigma J^\infty_\Sigma \mathcal{E}
		}
    \end{gathered}
		\, ,
	$$
	where we have illustrated by a dashed arrow the unique monomorphism
	that factors $\rho_\mathcal{E}$ through $\mathcal{E}^\oo$ (note that
	monomorphisms always factor through monomorphisms), which exists by
	the pullback property. Composing the morphisms $\mathcal{E}^\oo
	\hookrightarrow \mathcal{E} \hookrightarrow \mathcal{E}^\oo$ gives a
	factorization of the cone with vertex $\mathcal{E}^\oo$ through
	itself. By the pullback property, there is only one way for such a
	factorization to be done, which hence must coincide with the identity
	$\mathcal{E}^\oo \stackrel{\mathrm{id}}{\longrightarrow}
	\mathcal{E}^\oo$. If the composition of two monomorphisms is the
	identity, then they both had to be isomorphisms to begin with. Thus,
	$\mathcal{E} \simeq \mathcal{E}^\oo$ and we have shown that
	$\rho_\mathcal{E} \colon \mathcal{E} \hookrightarrow J^\oo_\Sigma
	\mathcal{E}$ is a generalized formally integrable PDE.

	(d)
	In one direction, suppose that $\phi\colon \mathcal{E} \to
	\mathcal{E}'$ preserves formal solutions. Then, using formal
	integrability and lemma~\ref{ProlongationIsUniversalFamily}, we
	get the commutative diagram
	$$
    \begin{gathered}
		\xymatrix{
			T^\oo_\Sigma \mathcal{E}
				\ar[r]^{T^\oo_\Sigma \phi}
				\ar[d]_{\overline{\rho_\mathcal{E}}} &
			T^\oo_\Sigma \mathcal{E}'
				\ar[d]^{\overline{\rho_{\mathcal{E}'}}}
			\\
			\mathcal{E}
				\ar[r]_{\phi} &
			\mathcal{E'}
		}
    \end{gathered}
		\, ,
	$$
	which expresses the fact that the composition of $\phi$ with the
	universal family of formal solutions of $\mathcal{E}$ factors through
	the universal family of solutions of $\mathcal{E}'$. By the
	$T^\oo_\Sigma \dashv J^\oo_\Sigma$ adjunction, we get precisely the
	commutative diagram showing that $\phi$ is a morphism of
	corresponding coalgebras.

	In the other direction, reversing the preceding argument, applying the
	adjunction to the commutative square showing that $\phi$ is a morphism
	of coalgebras, we get the commutative square with the $T^\oo_\Sigma
	\phi$ morphism as illustrated above, implying the identity $\phi \circ
	\overline{\rho_\mathcal{E}} = \overline{\rho_{\mathcal{E}'}} \circ
	T^\oo_\Sigma \phi$. By lemma~\ref{CompositionOfFormalFamilies}, this
	means that $\phi$ preserves the universal family of formal solutions
	$\overline{\rho_\mathcal{E}}$ and hence preserves all families of
	formal solutions.
\end{proof}

Now the main theorem:
\begin{theorem}
	\label{PDEIsEM}
	Let $\mathbf{H}$ be a differentially cohesive topos (definition \ref{DifferentialCohesion})
	and let $\Sigma$ be a $V$-manifold in $\mathbf{H}$ (definition \ref{VManifold}).
	Then there is an equivalence of categories
	$$
		\mathrm{PDE}_{/\Sigma}(\mathbf{H}) \simeq \mathrm{EM}(J^\oo_\Sigma)
	$$
	between the category of formally integrable generalized PDEs in $\mathbf{H}$
	(definition~\ref{FormallyIntegrable}(d)) and the Eilenberg-Moore category
	of coalgebras (definition~\ref{EMOverComonad}) over the jet comonad $J^\oo_\Sigma\colon \mathbf{H}_{/\Sigma} \to
	\mathbf{H}_{/\Sigma}$  (definition \ref{jetcomonad}), all over $\Sigma$. Moreover,
	for any full subcategory $\mathcal{C} \hookrightarrow
	\mathbf{H}_{/\Sigma}$
    %
    closed under $J^\infty_\Sigma$,
    the equivalence restricts to the corresponding
	subcategories $\mathrm{PDE}(\mathcal{C}) \simeq
	\mathrm{EM}(J^\oo_\Sigma|_{\mathcal{C}})$, meaning that the
	following diagram commutes:
  $$
    \begin{gathered}
  	\xymatrix{
  		\mathrm{PDE}(\mathcal{C})
  			\ar@{^{(}->}[d]
  			\ar[r]^{\simeq}
  		&
			\mathrm{EM}(J^\infty_\Sigma|_{\mathcal{C}})
  			\ar@{^{(}->}[d]
  		\\
  		\mathrm{PDE}_{/\Sigma}(\mathbf{H})
  			\ar[r]^{\simeq}
  		&
			\mathrm{EM}(J^\infty_\Sigma)
    }
    \end{gathered}
    \, .
  $$
\end{theorem}
\begin{proof}
	Proposition \ref{InputForPDEisEM} shows that there are functors in both directions
	between $\mathrm{PDE}_{/\Sigma}(\mathbf{H})$ and
	$\mathrm{EM}(J^\oo_\Sigma)$. It remains only to check that they
	compose to identity, up to natural isomorphism, in both directions. It
	is obvious from part (c) that the coalgebra $\rho_\mathcal{E}\colon
	J^\oo_\Sigma \hookrightarrow \mathcal{E}$ gets sent back to itself
	under composition in one direction. In the other direction, a formally
	integrable PDE $e_Y \colon \mathcal{E} \hookrightarrow J^\oo_\Sigma Y$
	gets sent back to $\rho_\mathcal{E} \colon \mathcal{E} \hookrightarrow
	J^\oo_\Sigma \mathcal{E}$. However, the natural morphism
	$\mathrm{id}_\mathcal{E}\colon \mathcal{E} \to \mathcal{E}$ clearly
	preserves formal solutions and is an isomorphism. Hence, we have the
	desired equivalence of categories.

	Finally, consider a full subcategory $\mathcal{C}
	\hookrightarrow \mathbf{H}_{/\Sigma}$ that is closed under $J^\oo_\Sigma$.
	If $\mathcal{E} \in \mathcal{C}$, then all the morphisms needed
	to define a PDE or coalgebra structure on it are also
	$\mathcal{C}$, because it is a full subcategory. Similarly, any
	morphism in $\mathrm{PDE}_{/\Sigma}(\mathbf{H})$ or
	$\mathrm{EM}(J^\oo_\Sigma)$ breaks down to morphisms in
	$\mathcal{C}$ if its source and target have underlying objects
	in $\mathcal{C}$. Hence, clearly,
	$\mathrm{PDE}(\mathcal{C}) \hookrightarrow
	\mathrm{PDE}_{/\Sigma}(\mathbf{H})$ and
	$\mathrm{EM}(J^\oo_\Sigma|_{\mathcal{C}}) \hookrightarrow
	\mathrm{EM}(J^\oo_\Sigma)$ are full subcategories that are respected
	by the equivalence.
\end{proof}

\begin{remark}
\label{CoalgebraIsCanonicalForm}
We introduced the notion of a PDE (definition~\ref{generalpdes}) by
relying on the extrinsic structure of its inclusion $\mathcal{E}
\hookrightarrow J^\oo_\Sigma Y$ in a jet bundle, which also plays a
crucial role in the corresponding notion of a solution. While this
extrinsic structure seems natural from the traditional point of view on
PDEs, a complete equivalence of the category of PDEs based on this
definition (definition~\ref{FormallyIntegrable}(d)) with coalgebras
over the $J^\oo_\Sigma$ comonad shows that this extrinsic information
is not actually necessary and the only inclusion that counts is that of
a (formally integrable) PDE $\rho_\mathcal{E} \colon \mathcal{E}
\hookrightarrow J^\oo_\Sigma \mathcal{E}$ in its own jet bundle, with
$\rho_\mathcal{E}$ precisely giving it the structure of a coalgebra.
This completely intrinsic structure can be still seen as a traditional
PDE when written out in the following way:
$$
	j^\oo \phi = \rho_\mathcal{E}(\phi) \, ,
$$
where $\phi\colon \Sigma \to \mathcal{E}$ is a section. In a sense, this
form can be seen as the canonical form of a PDE achieved by ``solving
for the highest derivatives,'' in analogy with how it is commonly done
for ordinary differential equations. Of course, any such equation may
have integrability conditions, following from the differential
consequences of the equation as written above. From this point of view,
the coaction identity (definition~\ref{EMOverComonad}) satisfied by
$\rho_\mathcal{E}$ is a necessary and sufficient condition for the
absence of non-trivial integrability conditions. That is, plugging the
above equation into the universal integrability condition $j^\oo j^\oo
\phi = \Delta_\mathcal{E} (j^\oo \phi)$ implies the identity
$p^\oo\rho_\mathcal{E} = \Delta_\mathcal{E} \circ \rho_\mathcal{E}$, but
it does not yield any new conditions on the section $\phi$, because
$p^\oo \rho_\mathcal{E} = J^\oo_\Sigma \rho_\mathcal{E} \circ
\Delta_\mathcal{E}$ by definition~\ref{InfiniteProlongation} and
$\rho_\mathcal{E}$ already satisfies the identity $J^\oo_\Sigma
\rho_\mathcal{E} \circ \Delta_\mathcal{E} = \Delta_\mathcal{E} \circ
\rho_\mathcal{E}$ by virtue of being a coalgebra structure on
$\mathcal{E}$.
\end{remark}

\begin{remark}
	\label{GeneralizedPDEVsVinogradov}
	As mentioned earlier in remark~\ref{RemarksOnPDECategory}(c), while
	our definition for the category of generalized PDEs shares the same
	heuristic foundations with the one in Vinogradov's
	approach~\cite{Vinogradov80,VinKras}, they are a priori distinct,
	though ours could be said to be intuitively closer to these
	heuristics. Both have an intrinsic description, without requiring an
	extrinsic embedding into a jet bundle, ours in terms of jet coalgebra
	structures and Vinogradov's in terms of the Cartan distribution. We
	are finally at a point where we can state a precise relation between
	these two categories. Instead of relating the two definitions
	directly, which would require a detailed synthetic formalization of
	the notion of the Cartan distribution, we will simply rely on Marvan's
	original and insightful observation~\cite{Marvan86,Marvan89} of the
	equivalence of Vinogradov's category with the Eilenberg-Moore
	category of coalgebras over the jet comonad in fibered manifolds
	(definition~\ref{FiberedManifold}), which by our main
	theorem~\ref{PDEIsEM} embeds fully and faithfully in our category of
	generalized PDEs in the Cahiers topos $\mathbf{H} =
	\mathrm{FormalSmoothSet}$.
\end{remark}

\begin{corollary}
	\label{GeneralizedPDEVsVinogradovCorollary}
	Vinogradov's category of PDEs over an ordinary manifold $\Sigma$
	embeds fully and faithfully in
	$\mathrm{PDE}_{/\Sigma}(\mathrm{FormalSmoothSet})$, with image
	$\mathrm{EM}(J^\oo_\Sigma|_{\mathrm{LocProMfd}_{\downarrow \Sigma}})$.
\end{corollary}

\begin{example}
  \label{PDEDiffOp}
  There are several immediate sources of examples of morphisms in the PDE
  category.

  (a) Recall that the co-Kleisli category of cofree coalgebras over a
  comonad (definition~\ref{coKleisli}) embeds as a full subcategory of
  the Eilenberg-Moore category of coalgebras over a comonad
  (definition~\ref{EMOverComonad}) and, in
  particular, we have
  $\mathrm{Kl}(J^\oo_\Sigma) \hookrightarrow \mathrm{EM}(J^\oo_\Sigma)$.
  Putting together the notation from
  definition~\ref{InfiniteProlongation} and the equivalences
  $\mathrm{DiffOp}_{/\Sigma}(\mathbf{H}) \simeq \mathrm{Kl}(J^\oo_\Sigma)$
  (proposition~\ref{BundlesDiffOpsAndPDEs}) and
  $\mathrm{PDE}_{/\Sigma}(\mathbf{H}) \simeq \mathrm{EM}(J^\oo_\Sigma)$, we
  have the full embedding
  $$
  	\xymatrix{
			\mathrm{DiffOp}_{/\Sigma}(\mathbf{H})
				\ar@{^{(}->}[rr]^-{(J^\oo_\Sigma,p^\oo)} &&
			\mathrm{PDE}_{/\Sigma}(\mathbf{H})
		} \, ,
  $$
	where the notation means that objects are embedded as \emph{cofree
	PDEs}, $E \mapsto (J^\oo_\Sigma E \stackrel{\Delta_E}{\longrightarrow}
	J^\oo_\Sigma J^\oo_\Sigma E)$, and morphisms as infinitely prolonged
	differential operators, $(J^\oo_\Sigma E_1
	\stackrel{\phi}{\longrightarrow} E_2) \mapsto (J^\oo_\Sigma E_1
	\stackrel{p^\oo\phi}{\longrightarrow} J^\oo_\Sigma E_2)$. Since the
	embedding is full (remark~\ref{compositionalaKleisli}), it means that
	the only allowed morphisms between cofree PDEs are prolonged
	differential operators.

	(b) The trivial bundle $[\Sigma
	\stackrel{\mathrm{id}}{\longrightarrow} \Sigma] \in
	\mathbf{H}_{/\Sigma}$ embeds in $\mathrm{DiffOp}_{/\Sigma}(\mathbf{H})$
	and hence as a cofree object $\Sigma \in
	\mathrm{PDE}_{/\Sigma}(\mathbf{H})$. Recall that $\Sigma \simeq
	J^\oo_\Sigma \Sigma$ (example~\ref{JetBundleOfSigma}). Hence, every
	PDE morphism $\sigma^\oo\colon \Sigma \to J^\oo_\Sigma Y$ into a
	cofree object must be of the form $\sigma^\oo = p^\oo \sigma = j^\oo
	\sigma$, for some section $\sigma \colon \Sigma \to Y$.

	(c) The inclusion morphism of a formally integrable generalized PDE in
	a jet bundle, $e_Y \colon \mathcal{E} \hookrightarrow J^\oo_\Sigma Y$
	is another example of a morphism in $\mathrm{PDE}_{/\Sigma}(\mathbf{H})$.
	Keeping in mind the isomorphism $\mathcal{E}^\oo \simeq \mathcal{E}$,
	this can be seen directly from the top square of the commutative
	diagram in lemma~\ref{ProlongationIsUniversalFamily}, which
	precisely shows that $e_Y$ is a morphism of coalgebras.

	(d) For any full subcategory $\mathcal{C} \hookrightarrow
	\mathbf{H}_{/\Sigma}$, such as $\mathcal{C} =
	\mathrm{LocProMfd}_{\downarrow \Sigma}$, to which we can restrict the
	$J^\oo_\Sigma$ comonad, all of the following embeddings of categories
	are compatible:
  $$
    \begin{gathered}
  	\xymatrix{
			\mathrm{DiffOp}(\mathcal{C})
				\ar@{^{(}->}[d]
				\ar@{^{(}->}[rr]^{(J^\oo_\Sigma,p^\oo)}
			&&
			\mathrm{PDE}(\mathcal{C})
				\ar@{^{(}->}[d]
			\\
			\mathrm{DiffOp}_{/\Sigma}(\mathbf{H})
				\ar@{^{(}->}[rr]_{(J^\oo_\Sigma,p^\oo)}
			&&
			\mathrm{PDE}_{/\Sigma}(\mathbf{H})
    }
    \end{gathered}
    \, .
  $$
\end{example}

\begin{theorem}[Products and equalizers]
	\label{PDEHasKernels}
	Consider two formally integrable PDEs $\mathcal{E}, \mathcal{F} \in
	\mathrm{PDE}_{/\Sigma}(\mathbf{H})$, presented as coalgebras over
	$J^\oo_\Sigma$ with coaction morphisms $\rho_\mathcal{E} \colon
	\mathcal{E} \to J^\oo_\Sigma\mathcal{E}$ and $\rho_\mathcal{F} \colon
	\mathcal{E} \to J^\oo_\Sigma\mathcal{F}$.

	(a) The $\mathbf{H}_{/\Sigma}$-product%
		\footnote{Since $\Sigma \in \mathbf{H}_{/\Sigma}$ is the terminal
		object, the product in $\mathbf{H}_{/\Sigma}$ must be computed as
		the $\mathcal{E} \times_\Sigma \mathcal{F}$ product in $\mathbf{H}$.} %
	$\mathcal{E}\times \mathcal{F}$ has the coalgebra structure
	$$
		\rho_{\mathcal{E}\times \mathcal{F}}
			= (\rho_\mathcal{E}, \rho_\mathcal{F}) \colon
			\mathcal{E} \times \mathcal{F}
				\to J^\oo_\Sigma \mathcal{E} \times J^\oo_\Sigma \mathcal{F}
			\simeq J^\oo_\Sigma (\mathcal{E} \times \mathcal{F})
		\, ,
	$$
	which makes it into the $\mathrm{PDE}_{/\Sigma}(\mathbf{H})$-product of
	$\mathcal{E}$ and $\mathcal{F}$.

	(b)
	For any pair of coalgebra morphisms $f,g\colon \mathcal{E} \to
	\mathcal{F}$, the $\mathbf{H}_{/\Sigma}$-equalizer
	$$
		\xymatrix{
			\mathcal{K} \ar@{^{(}->}[r]^{e} &
			\mathcal{E} \ar@<+3pt>[r]^{f} \ar@<-3pt>[r]_{g} &
			\mathcal{F}
		}
	$$
	has a unique coalgebra structure, $\rho_\mathcal{K} \colon \mathcal{K}
	\to J^\oo_\Sigma \mathcal{K}$, that makes $e\colon \mathcal{K} \to
	\mathcal{E}$ into a morphism of coalgebras and moreover the
	$\mathrm{PDE}_{/\Sigma}(\mathbf{H})$-equalizer of $f$ and $g$.

	(c)
	When, $\mathcal{E} = J^\oo_\Sigma Y$, $\mathcal{F} = J^\oo_\Sigma Z$
	and $f = p^\oo f_0$, $g = p^\oo g_0$ for some differential operators
	$f,g\colon J^\oo_\Sigma Y \to Z$. The
	$\mathrm{PDE}_{/\Sigma}(\mathbf{H})$-equalizer $\mathcal{K}$ of $p^\oo
	f_0$ and $p^\oo g_0$ coincides with the prolongation of the
	$\mathbf{H}_{/\Sigma}$-equalizer
	$$
		\xymatrix{
			\mathcal{K}_0 \ar@{^{(}->}[r]^{e_0} &
			J^\oo_\Sigma Y \ar@<+3pt>[r]^{f_0} \ar@<-3pt>[r]_{g_0} &
			Z
		} \, .
	$$
	That is, $\mathcal{K} \simeq (\mathcal{K}_0)^\oo$
	(lemma~\ref{ProlongationIsUniversalFamily}), as
	formally integrable generalized PDEs.
\end{theorem}
\begin{proof}
(a)
This part follows directly from the observation that $J^\oo_\Sigma$
preserves $\mathbf{H}_{/\Sigma}$-products.

(b)
Given that $e = \ker_{\mathbf{H}_{/\Sigma}} (f,g) \colon \mathcal{K} \to
\mathcal{E}$, it is a monomorphism. Since equalizers are preserved by
$J^\oo_\Sigma$ (from definition~\ref{jetcomonad} it is right adjoint and
hence preserves all limits), $J^\oo_\Sigma e\colon J^\oo_\Sigma
\mathcal{K} \to J^\oo_\Sigma \mathcal{E}$ is also a monomorphism. By the
properties of morphisms of coalgebras, the solid arrows in the following
diagram in $\mathbf{H}_{/\Sigma}$ commute:
$$
  \begin{gathered}
	\xymatrix{
		\mathcal{K}
			\ar@{^{(}-->}[d]_{\rho_\mathcal{K}}
			\ar@{^{(}->}[r]^{e} &
		\mathcal{E}
			\ar@{^{(}->}[d]_{\rho_\mathcal{E}}
			\ar@<+3pt>[r]^{f}
			\ar@<-3pt>[r]_{g} &
		\mathcal{F}
			\ar@{^{(}->}[d]_{\rho_\mathcal{F}}
		\\
		J^\oo_\Sigma \mathcal{K}
			\ar@{^{(}->}[r]^{J^\oo_\Sigma e} &
		J^\oo_\Sigma \mathcal{E}
			\ar@<+3pt>[r]^{J^\oo_\Sigma f}
			\ar@<-3pt>[r]_{J^\oo_\Sigma g} &
		J^\oo_\Sigma \mathcal{F}
	}
  \end{gathered}
	\, .
$$
And, because $J^\oo_\Sigma e$ is a monomorphism, there exists the dashed
arrow $\rho_{\mathcal{K}}$ that is unique in making the whole diagram
commute. It follows that the left square in
the diagram is a pullback square for $\mathcal{K}$. Pasting it together
with the coaction square for $\rho_\mathcal{E}$ (which is also a
pullback square, by Proposition \ref{PDEIsEM}(c)), we get the diagram
$$
  \begin{gathered}
	\xymatrix{
		\mathcal{K}
			\ar@{}[dr]|{\mbox{\tiny (pb)}}
			\ar[d]_{\rho_\mathcal{K}}
			\ar[r]^{e} &
		\mathcal{E}
			\ar@{}[dr]|{\mbox{\tiny (pb)}}
			\ar[d]_{\rho_\mathcal{E}}
			\ar[r]^{\mathcal{E}} &
		J^\oo_\Sigma \mathcal{E}
			\ar[d]^{\Delta_\mathcal{E}}
		\\
		J^\oo_\Sigma \mathcal{K}
			\ar[r]_{J^\oo_\Sigma e} &
		J^\oo_\Sigma \mathcal{E}
			\ar[r]_{J^\oo_\Sigma \rho_\mathcal{E}} &
		J^\oo_\Sigma J^\oo_\Sigma \mathcal{E}
	}
  \end{gathered}
  \, ,
$$
whose outer square is hence also a pullback square for $\mathcal{K}$, by the
pasting law (Proposition \ref{PastingLaw}). In
other words $\rho_\mathcal{E} \circ e \colon \mathcal{K} \to
J^\oo_\Sigma \mathcal{E}$ is a formally integrable PDE and hence, by
proposition~\ref{PDEIsEM}(b), the morphism
$\rho_\mathcal{K}\colon \mathcal{K} \to J^\oo_\Sigma \mathcal{K}$ endows
$\mathcal{K}$ it with the structure of a coalgebra. The commutativity of
the left square of the above diagram then means that $e\colon
\mathcal{K} \to \mathcal{E}$ is a morphism of coalgebras, with
$\rho_\mathcal{K}$ the unique coalgebra structure making it true.

It remains to show that any other coalgebra morphism $h\colon
\mathcal{H} \to \mathcal{E}$ that equalizes $f$ and $g$ must uniquely
factor through $\mathcal{K}$. Since both $e$ and $J^\oo_\Sigma e$ are
equalizers in $\mathbf{H}_{/\Sigma}$, there exists dashed arrow $u$ that
is unique in making the top loop in the following diagram commute, with
the bottom loop also commuting since it is the $J^\oo_\Sigma$ image of
the top one:
$$
  \begin{gathered}
	\xymatrix{
		\mathcal{H}
			\ar[d]_{\rho_\mathcal{H}}
			\ar@{-->}[r]_{u}
			\ar@/^1pc/[rr]^{h}
			&
		\mathcal{K}
			\ar[d]_{\rho_\mathcal{K}}
			\ar[r]_{e}
			&
		\mathcal{E}
			\ar[d]_{\rho_\mathcal{E}}
		\\
		J^\oo_\Sigma \mathcal{H}
			\ar@{-->}[r]^{J^\oo_\Sigma u}
			\ar@/_1pc/[rr]_{J^\oo_\Sigma h}
			&
		J^\oo_\Sigma \mathcal{K}
			\ar[r]^{J^\oo_\Sigma e}
			&
		J^\oo_\Sigma \mathcal{E}
	}
  \end{gathered}
	\, .
$$
We already know that the right square and the outer square commute.
Hence, we can conclude that
$$
	J^\oo_\Sigma e\circ (\rho_\mathcal{K} \circ u)
	=
	J^\oo_\Sigma e \circ (J^\oo_\Sigma u \circ \rho_\mathcal{H})
	\, .
$$
But, since $J^\oo_\Sigma e$ is a monomorphism, it must be true that
$\rho_\mathcal{K} \circ u = J^\oo_\Sigma u \circ \rho_\mathcal{K}$ and
therefore that the left square commutes as well. This shows that
$u\colon \mathcal{H} \to \mathcal{K}$ is a morphism of coalgebras.
Therefore, we have shown that any $h$ as above must uniquely factor
through $e$ in $\mathrm{PDE}_{/\Sigma}(\mathbf{H})$, we can conclude that
$e = \ker_{\mathrm{PDE}_{/\Sigma}(\mathbf{H})} (f,g)$.

(c)
Consider the following diagram:
$$
  \begin{gathered}
	\xymatrix{
		\mathcal{K}
			\ar@<-4pt>@{^{(}-->}[dr]
			\ar@/^1pc/@{^{(}->}[drr]^e
			\ar@/_1pc/@{_{(}.>}[ddr]
		\\
		&
		(\mathcal{K}_0)^\oo
			\ar@<-3pt>@{_{(}-->}[ul]
			\ar[d]
			\ar[r]^{(e_0)^\oo}
			\ar@{}[dr]|{\mbox{\tiny{(pb)}}}
			&
		J^\oo_\Sigma Y
			\ar[d]|{\Delta_Y}
		\\
		&
		J^\oo_\Sigma \mathcal{K}_0
			\ar@{^{(}->}[r]_{J^\oo_\Sigma e_0} &
		J^\oo_\Sigma J^\oo_\Sigma Y
			\ar@<+3pt>[r]^-{J^\oo_\Sigma f_0}
			\ar@<-3pt>[r]_-{J^\oo_\Sigma g_0}
			&
		J^\oo_\Sigma Z
	}
  \end{gathered}
	\, .
$$
The bottom line is an equalizer, since $J^\oo_\Sigma$ preserves
equalizers, while the inner square is the defining pullback square of
the prolongation $(\mathcal{K}_0)^\oo$. Since $e$ is the equalizer of
$p^\oo f_0 = J^\oo_\Sigma f_0 \circ \Delta_Y$ and $p^\oo g_0 =
J^\oo_\Sigma g_0 \circ \Delta_Y$, $\Delta_Y\circ e$ equalizes
$J^\oo_\Sigma f_0$ and $J^\oo_\Sigma g_0$. Hence $\Delta_Y \circ e$
uniquely factors through $J^\oo_\Sigma \mathcal{K}_0$, illustrated by
the dotted monomorphism above (recall that monomorphisms always factor
through monomorphisms), which hence commutes with all the solid arrows.
This commutativity and the pullback property of $(\mathcal{K}_0)^\oo$
implies the unique commuting factorization $\mathcal{K} \hookrightarrow
(\mathcal{K}_0)^\oo$, illustrated by one of the dashed morphisms above.
On the other hand, the commutativity of the pullback square implies that
$(e_0)^\oo$ equalizes $p^\oo f_0$ and $p^\oo g_0$, which implies the
unique commuting factorization $(\mathcal{K}_0)^\oo \hookrightarrow
\mathcal{K}$, illustrated by the other dashed morphism above.

Now, composing the two dashed morphisms as $(\mathcal{K}_0)^\oo
\hookrightarrow \mathcal{K} \hookrightarrow (\mathcal{K}_0)^\oo$ gives a
commuting factorization of $(\mathcal{K}_0)^\oo$ through itself, which
by uniqueness of commuting factorizations through a pullback must be
equal to the identity $(\mathcal{K}_0)^\oo
\stackrel{\mathrm{id}}{\longrightarrow} (\mathcal{K}_0)^\oo$. But if the
composition of two monomorphisms is the identity, they must have both
been isomorphisms to begin with. Thus, $e \simeq (e_0)^\oo$ and
$\mathcal{K}, (\mathcal{K}_0)^\oo \hookrightarrow J^\oo_\Sigma Y$ are
equivalent subobjects. Since each of them can be uniquely endowed with
the structure of a coalgebra such that the monomorphisms $e$ and
$(e_0)^\oo$ are morphisms of coalgebras, $\mathcal{K}$ and
$(\mathcal{K}_0)^\oo$ are also isomorphic as objects in
$\mathrm{PDE}_{/\Sigma}(\mathbf{H})$.
\end{proof}

If all pairwise products and equalizers exist, then all finite products
and equalizers exist. Moreover, it is well-known that any category with
a terminal object, finite products and finite equalizers has all finite
limits~\cite[cor.V.2.1]{MacLane}. Hence, from parts (a) and (b) of the
preceding theorem~\ref{PDEHasKernels} we have the
\begin{corollary}
	\label{PDEHasFiniteLimits}
	The category $\mathrm{PDE}_{/\Sigma}(\mathbf{H})$ has all finite limits
	and they can be computed as $\mathbf{H}_{/\Sigma}$ limits.
\end{corollary}

We conclude this section by underscoring one of the central advantages
of the equivalence $\mathrm{PDE}_{/\Sigma}(\mathbf{H}) \simeq \mathrm{EM}(J^\oo_\Sigma)$
that we have established in theorem \ref{PDEIsEM}. Namely with  this equivalence at hand, we
may precisely characterize solutions of a PDE using only its intrinsic
coalgebra structure, this is proposition~\ref{CoAlgebraMorphismsAsSolutions}
below.

\begin{proposition}
  \label{CoAlgebraMorphismsAsSolutions}
	Given a formally integrable generalized PDE $e_Y\colon \mathcal{E}
	\hookrightarrow J^\oo_\Sigma Y$ over $\Sigma$
	(definition~\ref{FormallyIntegrable}), the
	induced $J^\infty_\Sigma$-coalgebra $\rho_{\mathcal{E}} \colon
	\mathcal{E} \hookrightarrow J^\oo_\Sigma \mathcal{E}$
	(proposition~\ref{PDEIsEM}) has the property that
	coalgebra morphisms from $\Sigma$ (regarded as a jet coalgebra via
	example \ref{PDEDiffOp}(b)) into it are in natural bijection
	with solutions to the PDE according to definition \ref{generalpdes}:
  $$
    \mathrm{Sol}_\Sigma(\mathcal{E})
      \simeq
    \mathrm{Hom}_{\mathrm{EM}(J^\infty_\Sigma)}(\Sigma, \mathcal{E})
    \,.
  $$
\end{proposition}
\begin{proof}
	Suppose that $\sigma^\oo_\mathcal{E} \colon \Sigma \to \mathcal{E}$ is
	a morphism in $\mathrm{PDE}_{/\Sigma}(\mathbf{H})$. By
	example~\ref{PDEDiffOp}(c), so is $e_Y\colon \mathcal{E} \to
	J^\oo_\Sigma Y$. Hence, so is $\sigma^\oo_Y = e_Y\circ
	\sigma^\oo_\mathcal{E} \colon \Sigma \to J^\oo_\Sigma Y$. But now, by
	example~\ref{PDEDiffOp}(b), any such morphism must be of the form
	$\sigma^\oo_Y = j^\oo \sigma_Y$ for some section $\sigma_Y \colon
	\Sigma \to Y$. Noticing now that $j^\oo\sigma_Y = e_Y \circ
	\sigma^\oo_\mathcal{E}$ immediately implies that $\sigma_Y \in
	\mathrm{Sol}_\Sigma(\mathcal{E})$.

	In the other direction, suppose that $\sigma_Y\colon \Sigma \to Y$ has
	a jet extension that factors through $\mathcal{E}$, that is $j^\oo
	\sigma_Y = e_Y \circ \sigma^\oo_\mathcal{E}$ for some
	$\sigma^\oo_\mathcal{E} \colon \Sigma \to \mathcal{E}$. Hence, it is
	easy to see that the following is a pullback diagram in
	$\mathbf{H}_{/\Sigma}$:
	$$
    \begin{gathered}
		\xymatrix{
			\Sigma \ar[r]^{\sigma^\oo_\mathcal{E}}="s" \ar[d]_{\mathrm{id}}
			&
			\mathcal{E} \ar[d]^{e_Y}
			\\
			\Sigma \ar[r]_{j^\oo\sigma_Y}="t"
			&
			J^\oo_\Sigma Y
      \ar@{}|{\mbox{\tiny (pb)}} "s"; "t"
		}
    \end{gathered}
		\, .
	$$
	By corollary~\ref{PDEHasFiniteLimits}, the same diagram is also a
	pullback diagram in $\mathrm{PDE}_{/\Sigma}(\mathbf{H}) \simeq
	\mathrm{EM}(J^\oo_\Sigma)$, meaning that $\sigma^\oo_\mathcal{E} \in
	\mathrm{Hom}_{\mathrm{EM}(J^\infty_\Sigma)}(\Sigma, \mathcal{E})$.
\end{proof}

\subsection{The topos of synthetic PDEs}
\label{TheToposOfSyntheticPDEs}

We have seen above in section \ref{PDEs} (theorem \ref{PDEIsEM}) that for $\Sigma$ a $V$-manifold, the Eilenberg-Moore
category (definition \ref{EMOverComonad})
of coalgebras for the jet comonad (definition \ref{jetcomonad}), acting on
any object of a differentially cohesive topos $\mathbf{H}$ (definition \ref{DifferentialCohesion})
is equivalently the category of formally integrable generalized PDEs,
generalized in the sense that both their
underlying bundles as well as their solution loci may again be arbitrary objects of $\mathbf{H}$:
$$
  \mathrm{PDE}_{/\Sigma}(\mathbf{H}) \simeq \mathrm{EM}(J^\infty_\Sigma)
  \,.
$$
It is a general fact that for $J$ a comonad, such that
\begin{enumerate}
  \item $J$ acts on an elementary topos,
  \item $J$ is a right adjoint,
\end{enumerate}
the Eilenberg-Moore category $\mathrm{EM}(J)$ is itself an
elementary topos \cite[V8, cor. 7]{MacLaneMoerdijk92}. Here we show
(theorem \ref{PDESite} below) that at least for $\mathbf{H} = \mathrm{FormalSmoothSet}$
then  $\mathrm{PDE}_\Sigma(\mathbf{H})$
is in fact a category of sheaves (definition \ref{sheaf}) over the category of ``ordinary'' PDEs, i.e.,\ of PDEs
in the category of (infinitesimally thickened) fibered manifolds.

This result shows that the generalized formally integrable PDEs of
definition \ref{FormallyIntegrable} are related to ordinary PDEs in the
same way that the (formal) smooth sets of definition
\ref{SmoothManifolds} and definition \ref{FormalSmoothSets} are related
to (formal) smooth manifolds. In particular this implies that every PDE
in the generalized sense is but a colimit of PDEs in the ordinary sense.
We expect that this observation will be useful in followup works in
order to identify a synthetic formulation for the variational
bicomplex~\cite{Anderson-book} construction (which lies at the
foundation of variational calculus and Lagrangian field theory) inside
$\mathrm{PDE}_{/\Sigma}(\mathbf{H})$.

\medskip

Let throughout this section $\mathbf{H} := \mathrm{FormalSmoothSet} := \mathrm{Sh}(\mathrm{FormalCartSp})$ be the
Cahiers topos (definition \ref{FormalSmoothSets}) with $\Im \colon
\mathbf{H} \to \mathbf{H}$ denoting the infinitesimal shape monad from
proposition \ref{ReImEtAdjunctionOnFormalSmoothSets}, and the natural
transformation $\eta \colon \mathrm{id} \to \Im$ its unit.

\begin{theorem}
  \label{generalPDEsIsSliceOverInfinitesimalShape}
  Let $\Sigma \in \mathrm{FormalSmoothSet}$ be any object.
  Then the functor
  $$
    (\eta_\Sigma)^\ast
      \;\colon\;
    \mathrm{FormalSmoothSet}_{/\Im \Sigma}
      \stackrel{\simeq}{\longrightarrow}
    \mathrm{EM}(J^\infty_\Sigma)
      \simeq
    \mathrm{PDE}_{/\Sigma}(\mathrm{FormalSmoothSet})
  $$
	(from the slice over $\Im\Sigma$ to the Eilenberg-Moore category of
	theorem~\ref{PDEIsEM}) which equips the pullback $(\eta_\Sigma)^\ast E$
	with the coalgebra structure given by
  $$
    \xymatrix{
      (\eta_\Sigma)^\ast E
        \ar[rr]^-{(\eta_\Sigma)(\epsilon_E)}
        &&
      (\eta_\Sigma)^\ast (\eta_\Sigma)_\ast (\eta_\Sigma)^\ast E
      =
      J^\infty_\Sigma \left((\eta_\Sigma)^* E\right)
    }
  $$
	(where $\epsilon \colon \mathrm{id} \to (\eta_\Sigma)_\ast
	(\eta_\Sigma)^\ast$ is the unit of the adjunction $(\eta_\Sigma)^\ast
	\dashv (\eta_\Sigma)_\ast$) is an equivalence of categories.
\end{theorem}
\begin{proof}
	By proposition \ref{comonadicdescent} it is sufficient that
	$(\eta_\Sigma) \colon \Sigma \longrightarrow \Im \Sigma$ is an
	epimorphism. This is the case by proposition
	\ref{UnitOfInfinitesimalShapeIsEpiInCahiersTopos}.
\end{proof}
\begin{proposition}
  \label{SliceSiteForHPDE}
	For $\Sigma \in \mathrm{FormalSmoothSet}$ any object, there is an equivalence of categories
  $$
    \mathrm{PDE}_{/\Sigma}(\mathrm{FormalSmoothSet})
      \simeq
    \mathrm{Sh}(\mathrm{FormalCartSp}/\Im\Sigma)
  $$
	which identifies the category of generalized PDEs
	(definition~\ref{FormallyIntegrable}(d)) with the
	category of sheaves (definition \ref{sheaf}) over the slice site
	(proposition \ref{SliceSite}) of that of formal Cartesian spaces
	(definition \ref{FormalSmoothSets}) over $\Sigma$.
\end{proposition}
\begin{proof}
  This follows via theorem \ref{generalPDEsIsSliceOverInfinitesimalShape}
  by proposition \ref{SliceSite}.
\end{proof}
By theorem \ref{generalPDEsIsSliceOverInfinitesimalShape} we may
identify the objects in the slice site $\mathrm{FormalCartSp}/\Im\Sigma$ of proposition
\ref{SliceSiteForHPDE} as generalized PDEs. But in fact these turn out
to be just very mildly generalized, in that their solution locus is a
locally pro-manifold that may admit a formal infinitesimal thickening:
\begin{proposition}
  \label{SliceOfLocProMfdLandsInFormallyThickenedPDEs}
	For $\Sigma \in \mathrm{SmthMfd} \hookrightarrow \mathbf{H} := \mathrm{FormalSmoothSet}$ a smooth
	manifold, the equivalence of theorem
	\ref{generalPDEsIsSliceOverInfinitesimalShape} restricted to the
	objects in the image under the Yoneda embedding of the slice site of
	proposition \ref{SliceSiteForHPDE} factors through jet coalgebra
	structures (definition \ref{EMOverComonad}) in formal locally
	pro-manifolds (definition \ref{FormalLocalProManifolds}). That is, there
	exists a dashed arrow that makes the following diagram commute:
  $$
    \begin{gathered}
    \xymatrix{
      \mathrm{FormalCartSp}_{/\Im\Sigma}
      \ar@{^{(}-->}[rr]
      \ar[d]_y
      &&
      \mathrm{EM}(J^\infty_\Sigma|_{\mathrm{FormalLocProMfd}_{\downarrow \Sigma}})
      \ar@{^{(}->}[d]
      \ar[r]^\simeq
      &
      \mathrm{PDE}_{\downarrow\Sigma}(\mathrm{FormalLocProMfd})
      \ar@{^{(}->}[d]
      \\
      \mathbf{H}_{/\Im \Sigma}
        \ar[rr]_-\simeq
        &&
      \mathrm{EM}(J^\infty_\Sigma)
      \ar[r]_\simeq
      &
      \mathrm{PDE}_{/\Sigma}(\mathbf{H})
    }
    \end{gathered}
    \, .
  $$
\end{proposition}
\begin{proof}
  By theorem~\ref{generalPDEsIsSliceOverInfinitesimalShape} the bottom left equivalence in the
  above diagram sends an object
  $$
    X = \mathbb{R}^k \times \mathbb{D} \in \mathrm{FormalCartSp}
  $$
  in the slice over $\Im \Sigma$
  $$
    [X \stackrel{f}{\longrightarrow} \Im \Sigma]
    \in \mathrm{FormalCartSp}_{/\Im \Sigma}
  $$
  to its pullback along $\Sigma \overset{\eta_\Sigma}{\to} \Im \Sigma$ into the slice over $\Sigma$
  $$
    [\eta_\Sigma^* X \longrightarrow \Sigma]
    \in
    \mathrm{FormalSmoothSet}_{/\Sigma}
  $$
  and equipped there with some coalgebra structure. This coalgebra structure is
  of no concern for the proof, all we need to check is that this pullback actually lands in the
  image of the inclusion
  $$
    \mathrm{FormalLocProMfd}_{\downarrow \Sigma} \hookrightarrow \mathrm{FormalSmoothSet}_{/\Sigma}
    \,.
  $$
  To this end, we  claim that:

  {\bf Claim:} Every morphism of the form
  $$
    f \colon X \longrightarrow \Im \Sigma
  $$
	for $X = \mathbb{R}^k \times \mathbb{D} \in \mathrm{FormalCartSp}$
	factors as
  $$
    X \overset{f'}{\longrightarrow} \Sigma \stackrel{\eta_\Sigma}{\longrightarrow} \Im \Sigma
    \,.
  $$

  With this claim the desired statement follows:
  By the pasting law (proposition \ref{PastingLaw}) the pullback of any
  $f$ that is so factored is given by
  the pasting of two pullbacks as in the following diagram
  $$
    \begin{gathered}
    \xymatrix{
      T^\infty_\Sigma X
        \ar[r]
        \ar[d]
        \ar@{}[dr]|{\mbox{\tiny (pb)}}
        &
      X
      \ar[d]
      \ar@/^1pc/[dd]^f
      \\
      T^\infty \Sigma
      \ar[d]_{\pi}
      \ar[r]|{\mathrm{ev}}
      \ar@{}[dr]|{\mbox{\tiny (pb)}}
      &
      \Sigma
      \ar[d]|{\eta_\Sigma}
      \\
      \Sigma \ar[r]_{\eta_\Sigma} & \Im \Sigma
    }
    \end{gathered}
    \, .
  $$
  This identifies the pullback with the infinitesimal disk bundle from definition
  \ref{InfinitesimalDiskBundle}
  $$
    \eta_{\Sigma}^\ast X \simeq T_\Sigma^\infty X
  $$
  as shown. Moreover, by example \ref{SmoothManifoldsAreVManifolds} and proposition \ref{FormalDiskBundleLocallyTrivializableOverVManifold},
	the object $T^\infty \Sigma$ is locally a Cartesian product of a chart
    $$
      \mathbb{R}^n \hookrightarrow \Sigma
    $$
    with the formal neighbourhood $\mathbb{D}^n$ of any point in $\mathbb{R}^n$, and according to corollary \ref{ComponentVersionOfDiskBundleMonadProduct}
    the morphism $\mathrm{ev}$ is given on generalized elements
   (definition \ref{ComponentNotationForInfinitesimalDiskBundles})
	by $(s,d) \mapsto s-d$. Therefore $T^\infty_\Sigma X$ is locally
    the pullback in the following diagram:
    $$
      \xymatrix{
        {X|_{\mathbb{R}^n}} \times \mathbb{D}^n
        \ar[rr]^{(x,d) \mapsto x}
        \ar[d]_{(x,d) \mapsto  (f'(x)+d,d)}
        \ar@{}[drr]|{\mbox{\tiny (pb)}}
        &&
        X|_{\mathbb{R}^n}
        \ar[d]^{x \mapsto  f'(x)}
        \\
        \mathbb{R}^n \times \mathbb{D}^n
        \ar[rr]_{(s,d) \mapsto s-d}
        &&
        \mathbb{R}^n
      }
    $$
    Now since ${X|_{\mathbb{R}^n}} \times \mathbb{D}^n$ is the Cartesian product of
    a formal manifold with an infinitesimally thickened point, it is itself a formal manifold
    and so in particular a formal locally pro-manifold
    according to definition~\ref{FormalLocalProManifolds}. This is the desired conclusion.

	It remains to prove the above claim that every $f$ factors through $f'$ in this
	way. To that end, recall the adjoint quadruple
    $$
      i_! \dashv i^\ast \dashv i_\ast \dashv i_!
        \;:\;
      \mathrm{SmoothSet} \leftrightarrow \mathrm{FormalSmoothSet}
    $$
	from proposition~\ref{diffcohesionofFormalSmoothSet} and the
	identification $\Im = i_\ast i^\ast$ from definition
	\ref{ReImEtAdjunctionOnFormalSmoothSets}. Using this and forming adjuncts,
    (definition \ref{AdjointFunctor}) the morphisms in $\mathrm{FormalSmoothSet}$ of the form $f : X \to
	\Im \Sigma$ are in natural bijection with morphisms in $\mathrm{SmoothSet}$ of the form
  $$
    \phi : i^\ast X \longrightarrow i^\ast \Sigma
    \,.
  $$
	Here $i^\ast X = \mathbb{R}^k$ is just the ordinary Cartesian space
	underlying the formal Cartesian space $X = \mathbb{R}^k \times
	\mathbb{D}$, and $i^\ast \Sigma$ is just the smooth manifold $\Sigma$ itself, regarded in
	$\mathrm{SmoothSet}$. Now if we denote by $\tilde f$ the composite
   $$
     \tilde f : X \to \Re X = i_! i^\ast X \overset{i_! \phi}{\longrightarrow} i_! i^\ast \Sigma = \Sigma
   $$
   then
   $$
     \phi = i^\ast \tilde f
   $$
   by the fact that $i_!$ is fully faithful, and using proposition \ref{AdjointsOfFullEmbeddings}.

  Now by the formula for adjuncts (proposition \ref{AdjunctsInTermsOfUnitAndCounit})
  the fact that $\phi$ is, by definition, the $(i^\ast \dashv i_\ast)$-adjunct of $f$ means that the
  left triangle in the following diagram commutes
    $$
    \begin{gathered}
    \xymatrix{
      X \ar[dr]|f \ar[r]^{\tilde f} \ar[d]_{\eta_X}
        &
      \Sigma \ar[d]^{\eta_\Sigma}
      \\
      i_\ast i^\ast X
      \ar[r]_{{i_\ast \phi} \atop {= \atop {i_\ast i^\ast \tilde f}} }
      &
      i_\ast i^\ast \Sigma
    }
    \end{gathered}
    \,.
  $$
  But the equality which we just established, shown on the bottom edge of the diagram, now exhibits
  this triangle as being one half of the naturality square of the unit $\eta : \mathrm{id} \to i_\ast i^\ast$
  on $\tilde f$. The other half is the right triangle shown in the above diagram. Therefore this
  right triangle also commutes, and this is the factorization to be shown.
\end{proof}
This now allows us to use the category of ordinary PDEs as a site that
presents the category of generalized PDEs:
\begin{theorem}
  \label{PDESite}
	For $\Sigma \in \mathrm{SmthMfd} \hookrightarrow \mathrm{FormalSmoothSet}$
	a smooth manifold, there
	is a subcanonical coverage (definition \ref{subcanonical}) on the category $\mathrm{PDE}_{\downarrow
	\Sigma}(\mathrm{FormalLocProMfd})$ of formally integrable PDEs in
	formal manifolds (definition \ref{FormalLocalProManifolds}), making it
	a site (definition \ref{site}), such that there is an equivalence of
	categories
  $$
    \mathrm{PDE}_{/\Sigma}(\mathrm{FormalSmoothSet})
      \simeq
    \mathrm{Sh}(\mathrm{PDE}_{/\Sigma}(\mathrm{FormalLocProMfd}))
  $$
	which identifies the category of generalized formally integrable PDEs $\mathrm{PDE}_{/\Sigma}(\mathbf{H}) \simeq
	\mathrm{EM}(J^\infty_\Sigma)$ from theorem~\ref{PDEIsEM} with the
	category of sheaves (definition \ref{sheaf}) over the category of PDEs in the ordinary sense.
\end{theorem}
\begin{proof}
	By proposition~\ref{SliceSiteForHPDE} there is an equivalence of
	categories
  $$
    \mathrm{PDE}_{/\Sigma}(\mathrm{FormalSmoothSet})
      \simeq
    \mathrm{Sh}(\mathrm{FormalCartSp}/\Im\Sigma)
  $$
	with sheaves on the slice site of formal Cartesian spaces over $\Im
	\Sigma$. By
	proposition~\ref{SliceOfLocProMfdLandsInFormallyThickenedPDEs} there
	is a full inclusion
  $$
    \mathrm{FormalCartSp}/\Im\Sigma
      \hookrightarrow
    \mathrm{PDE}_{\downarrow \Sigma}(\mathrm{FormalLocProMfd})
  $$
	of the slice site into the category of ordinary PDEs. Moreover, since
	$\mathrm{FormalCartSp}/\Im\Sigma$ is a site of definition for PDEs,
	the objects in $\mathrm{PDE}_{\downarrow
	\Sigma}(\mathrm{FormalLocProMfd})$ are faithfully tested on
	$\mathrm{FormalCartSp}/\Im\Sigma$. Thus
	proposition~\ref{AmbientCategoryGrothedieckTopology} implies the
	claim.
\end{proof}

\appendix

\section{Formal solutions of PDEs}

We state and prove here two technical lemmas regarding the general
concept of formal solutions of PDEs (definition \ref{FamilyOfFormal}).

In order to prevent an explosion of notation, we will use the
following device in the sequel.  Since the adjunction $T^\oo_\Sigma \dashv
J^\oo_\Sigma$ (definition \ref{jetcomonad}) provides us with a natural bijection
$\mathrm{Hom}(T^\oo_\Sigma (-), -) \simeq \mathrm{Hom}(-,J^\oo_\Sigma
(-))$ (definition \ref{AdjointFunctor}), we may freely use one kind of morphism to label the other kind.
For instance, recalling the notation for adjunct morphisms introduced in
definition~\ref{jetcomonad}, we might use $\overline{\alpha}$ to refer
to an arbitrary morphism $T^\oo_\Sigma E \to Y$, which is labeled by
the corresponding morphism $\alpha\colon E \to J^\oo_\Sigma Y$.

\begin{lemma}
	\label{FamilyOfFormalAdjunctSquare}
	Consider $E, \mathcal{E}, Y \in \mathbf{H}_{/\Sigma}$, together with a
	monomorphism $e_Y \colon \mathcal{E} \hookrightarrow J^\oo_\Sigma Y$.

	(a) For an $E$-parametrized family of formal solutions
	$\overline{\rho^E}\colon T^\oo_\Sigma E \to \mathcal{E}$, the morphism
	$\rho^E \colon E \to J^\oo_\Sigma \mathcal{E}$ fits into a commutative
	diagram of the following form
	$$
    \begin{gathered}
		\xymatrix{
			E
				\ar[rr]^{e^E_Y}
				\ar[d]^{\rho^E}
				\ar@/_2pc/[dd]_{e^E_\mathcal{E}} &&
			J^\oo_\Sigma Y
				\ar[d]_{\Delta_Y}
				\ar@/^2pc/[dd]^{\mathrm{id}}
			\\
			J^\oo_\Sigma \mathcal{E}
				\ar[rr]_-{J^\oo_\Sigma e_Y}
				\ar[d]^{\epsilon_\mathcal{E}} &&
			J^\oo_\Sigma J^\oo_\Sigma Y
				\ar[d]_{\epsilon_{J^\oo_\Sigma Y}}
			\\
			\mathcal{E}
				\ar[rr]_{e_Y} &&
			J^\oo_\Sigma Y
		}
    \end{gathered}
		\, ,
	$$
	where
    \begin{enumerate}
      \item $e^E_\mathcal{E} := \epsilon_\mathcal{E}\circ \rho^E$ as shown;
      \item $e^E_Y := e_Y \circ \overline{\rho^E} \circ \eta_E$.
    \end{enumerate}

	(b) For a pair of morphisms $\rho^E \colon E \to J^\oo_\Sigma
	\mathcal{E}$ and $e^E_Y \colon E\to J^\oo_\Sigma Y$ that fit into a
	commutative diagram of the form
	$$
    \begin{gathered}
		\xymatrix{
			E
				\ar[rr]^{e^E_Y}
				\ar[d]_{\rho^E} &&
			J^\oo_\Sigma Y
				\ar[d]^{\Delta_Y}
			\\
			J^\oo_\Sigma \mathcal{E}
				\ar[rr]_{J^\oo_\Sigma e_Y} &&
			J^\oo_\Sigma J^\oo_\Sigma Y
		}
    \end{gathered}
		\, ,
	$$
	the adjunct morphism $\overline{\rho^E}\colon T^\oo_\Sigma E \to
	\mathcal{E}$ is an $E$-parametrized family of formal solutions and the
	above commutative square coincides with the top square of the
	commutative diagram associated to $\overline{\rho^E}$ in part (a).
\end{lemma}
\begin{proof}
	(a)
	Let us start with the defining property of $\overline{\rho^E}$ as a
	family of formal solutions. Namely, $\overline{\rho^E}$ and its adjunct
	$\rho^E$ together fit into the following commutative diagram:
	$$
    \begin{gathered}
		\xymatrix{
			&& \mathcal{E} \ar[d]^{e_Y} \\
			T^\oo_\Sigma E
				\ar[urr]^{\overline{\rho^E}}
				\ar[rr] &&
			J^\oo_\Sigma Y \ar[d]^{\Delta_Y} \\
			E
				\ar[u]^{\eta_E}
				\ar[rr]
				\ar@{-->}[urr]^{e^E_Y}
				\ar[drr]_{\rho^E} &&
			J^\oo_\Sigma J^\oo_\Sigma Y \\
			&& J^\oo_\Sigma \mathcal{E} \ar[u]_{J^\oo_\Sigma e_Y}
		}
    \end{gathered}
		\, .
	$$
	The middle square commutes because by hypothesis the composition $e_Y
	\circ \overline{\rho^E}$ is a family of formally holonomic sections of
	$J^\oo_\Sigma Y$, which factors through $e_Y\colon \mathcal{E}
	\hookrightarrow J^\oo_\Sigma Y$ by virtue of consisting of formal
	solutions. The top and bottom triangle are related by the
	$T^\oo_\Sigma \dashv J^\oo_\Sigma$ adjunction, which explains the
	appearance of the morphism $J^\oo_\Sigma e_Y$. The dashed diagonal
	arrow, which we have chosen to denote by $e^E_Y$, is the unique one
	that commutes with the rest of the diagram, namely $e^E_Y = e_Y \circ
	\overline{\rho^E} \circ \eta_E$.

	This shows that the subdiagram consisting of the morphisms $\rho^E$,
	$e^E_Y$, $J^\oo_\Sigma e_Y$ and $\Delta_Y$ commutes and is
	identical to the top square of the desired diagram in part (a) of the
	theorem. We can now conclude that this square commutes.

    The rest of
	the desired diagram is constructed by pasting to it the naturality
	square of the $\epsilon$ counit (definition~\ref{comonad}) of
	$J^\oo_\Sigma$ (bottom), the counit-coproduct commutative triangle
	(definition~\ref{EMOverComonad}.2) of $J^\oo_\Sigma$ (right) and the
	composite morphism $e^E_\mathcal{E} = \epsilon_\mathcal{E} \circ
	\rho^E$ (left). Since all the pasted subdiagrams commute, the whole
	diagram commutes as well, as was desired.

	(b)
	Let us split the commutative square from the hypothesis in part
	(b) of the theorem into two triangles, with $\tau$ denoting the
	common morphism between them:
	$$
    \begin{gathered}
		\xymatrix{
			E \ar[d]_{\rho^E} \ar[drr]^{\tau} &&
			&
			E \ar[rr]^{e^E_Y} \ar[drr]_{\tau} && J^\oo_\Sigma Y \ar[d]^{\Delta_Y}
			\\
			J^\oo_\Sigma \mathcal{E} \ar[rr]_{J^\oo_\Sigma e_Y} && J^\oo_\Sigma J^\oo_\Sigma Y
			&
			&& J^\oo_\Sigma J^\oo_\Sigma Y
		}
    \end{gathered}
		\, .
	$$
	Now observe that the  $(T^\oo_\Sigma \dashv J^\oo_\Sigma)$-adjuncts of
	these diagrams are the following triangles shown with solid arrows
	(for the left triangle this is just the naturality of forming
	adjuncts, for the right triangle this is by proposition
	\ref{PropertiesOfAdjointPairsFromAdjointTriples}):
	$$
    \begin{gathered}
		\xymatrix{
			&&
			\mathcal{E}
				\ar[d]^{e_Y}
			&
			T^\oo_\Sigma T^\oo_\Sigma E
				\ar[rr]^{\overline{\overline{\tau}}}
				\ar[d]_{\nabla_E} &&
			Y
			\\
			T^\oo_\Sigma E
				\ar[urr]^{\overline{\rho^E}}
				\ar[rr]_{\overline{\tau}} &&
			J^\oo_\Sigma Y
			&
			T^\oo_\Sigma E
				\ar[urr]^{\overline{e^E_Y}}
				\ar@{-->}[rr]^{\overline{\tau}} &&
			J^\oo_\Sigma Y
				\ar@{-->}[u]_{\epsilon_Y}
				\ar@{-->}[d]^{\Delta_Y}
			\\
			&& &
			E
				\ar@{-->}[u]^{\eta_E}
				\ar@{-->}[rr]^{\tau} &&
			J^\oo_\Sigma J^\oo_\Sigma Y
		}
    \end{gathered}
		\, .
	$$

	Now we argue with the method of local generalized elements as in
	remark \ref{InterpretationOfHolonomicSections} (using that $\Sigma$ is
	assumed to be a $V$-manifold) that the diagrams shown with dashed
	morphisms exist and commute: Since $\tau$ is the image of composing
	with the jet coproduct over a $V$-manifold, it depends on its
	infinitesimal arguments symmetrically, $\tau(x)(a,b) = e^E_Y(x)(a+b)$.
	Hence,
	\begin{align*}
		\epsilon_Y \circ \overline{\tau} \circ \nabla_E (x,a,b)
		&= \epsilon_Y \circ \overline{\tau} (x,a+b)
		= \overline{\tau}(x,a+b)(0)
		= \tau(x)(a+b,0) \\
		&= e^E_Y(x)(a+b)
		= \tau(x)(a,b)
		= \overline{\overline{\tau}}(x,a,b)
		\, .
	\end{align*}
	This shows that the square on the top right commutes. On the other
	hand, the commutative triangle on the left shows that
	$\overline{\tau}$ factors through $\overline{\rho^E} \colon
	T^\oo_\Sigma E \to \mathcal{E}$. In other words, as was desired, we
	have shown that $\overline{\rho^E}$ is a family of formal solutions.

	Finally, it is obvious that applying to $\overline{\rho^E}$ the
	argument of part (a) we recover the same commutative square that we
	have started with.
\end{proof}

\begin{lemma}
	\label{ProlongationIsUniversalFamily}
	Consider $\mathcal{E}, Y \in \mathbf{H}_{/\Sigma}$ and a generalized
	PDE $e_Y \colon \mathcal{E} \hookrightarrow J^\oo_\Sigma Y$.

	Every parametrized family of formal solutions of this PDE factors
	uniquely through a universal one,
	$\overline{\rho^\oo_\mathcal{E}}\colon T^\oo_\Sigma \mathcal{E}^\oo
	\to \mathcal{E}$, which is defined by the pullback square in the
	following commutative diagram:
	$$
    \begin{gathered}
		\xymatrix{
			\mathcal{E}^\oo
				\ar[rr]^{e^\oo_Y}_{\ }="s"
				\ar[d]^{\rho^\oo_\mathcal{E}}
				\ar@/_2pc/[dd]_{e^\oo_\mathcal{E}} &&
			J^\oo_\Sigma Y
				\ar[d]_{\Delta_Y}
				\ar@/^2pc/[dd]^{\mathrm{id}}
			\\
			J^\oo_\Sigma \mathcal{E}
				\ar[rr]_{J^\oo_\Sigma e_Y}^{\ }="t"
				\ar[d]^{\epsilon_\mathcal{E}} &&
			J^\oo_\Sigma J^\oo_\Sigma Y
				\ar[d]_{\epsilon_{J^\oo_\Sigma Y}}
      \ar@{}|{\mbox{\tiny (pb)}} "s"; "t"
      \\
      \mathcal{E}
      	\ar[rr]_{e_Y} &&
      J^\oo_\Sigma Y
		}
    \end{gathered}
		\, ,
	$$
	which defines the $e^\oo_\mathcal{E}\colon \mathcal{E}^\oo \to
	\mathcal{E}$ morphism and where each of $e^\oo_Y$,
	$\rho^\oo_\mathcal{E}$ and $e^\oo_\mathcal{E}$ is a monomorphism.

	If $\overline{\rho^E}\colon T^\oo_\Sigma E \to \mathcal{E}$, with $E \in
	\mathbf{H}_{/\Sigma}$ is a parametrized family of formal solutions,
	then it factors through the universal family
	$\overline{\rho^\oo_\mathcal{E}}$ as illustrated in the commutative
	diagram
	$$
		\xymatrix{
			T^\oo_\Sigma E
				\ar@{-->}[r]_-{T^\oo_\Sigma \phi}
				\ar@/^1.2pc/[rr]^-{\overline{\rho^E}}
				&
			T^\oo_\Sigma \mathcal{E}^\oo
				\ar[r]_-{\overline{\rho^\oo_\mathcal{E}}}
				&
			\mathcal{E}
		} \, ,
	$$
	where $\phi\colon E \to \mathcal{E}^\oo$ is the unique morphism such
	that the above diagram is commutative.
\end{lemma}
\begin{proof}
	As demonstrated in lemma~\ref{FamilyOfFormalAdjunctSquare}, an
	equivalent way of presenting an $E$-parametrized family of formal
	solutions $\overline{\rho^E} \colon T^\oo_\Sigma E \to J^\oo_\Sigma Y$
	is by the pair of morphisms $\rho^\oo_\mathcal{E}$ and $e^E_Y =
	\epsilon_Y\circ (J^\oo_\Sigma e_Y) \circ \rho^E$ that fit into the
	commutative diagram illustrated in part (b) of that proposition. But
	by the definition of $\mathcal{E}^\oo$ through the pullback diagram
	above and by the universality property of pullbacks, there exists a
	unique morphism $\phi$ (dashed below) that makes the following diagram
	commute:
	$$
    \begin{gathered}
		\xymatrix{
			E
				\ar@/^1pc/[drr]^{e^E_Y}
				\ar@/_1pc/[ddr]_{\rho^E}
				\ar@{-->}[dr]^\phi
			\\
			&
			\mathcal{E}^\oo
				\ar[r]^{e^\oo_Y}
				\ar[d]_{\rho^\oo_\mathcal{E}} &
			J^\oo_\Sigma Y
				\ar[d]^{\Delta_Y}
			\\
			&
			J^\oo_\Sigma \mathcal{E}
				\ar[r]_{J^\oo_\Sigma e_Y} &
			J^\oo_\Sigma J^\oo_\Sigma Y
		}
    \end{gathered}
		\, .
	$$
	By the $T^\oo_\Sigma \dashv J^\oo_\Sigma$ adjunction, it then follows
	that the morphism $T^\oo_\Sigma \phi$ (dashed below) making the
	following diagram commute:
	$$
		\xymatrix{
			T^\oo_\Sigma E
				\ar@{-->}[r]_-{T^\oo_\Sigma \phi}
				\ar@/^1.2pc/[rr]^-{\overline{\rho^E}}
				&
			T^\oo_\Sigma \mathcal{E}^\oo
				\ar[r]_-{\overline{\rho^\oo_\mathcal{E}}}
				&
			\mathcal{E}
		} \, .
	$$
	Finally, the commutativity of the pullback square defining
	$\mathcal{E}^\oo$ and lemma~\ref{FamilyOfFormalAdjunctSquare}(b)
	imply that $\overline{\rho^\oo_\mathcal{E}} \colon T^\oo_\Sigma
	\mathcal{E}^\oo \to \mathcal{E}$ is itself a parametrized family of
	formal solutions. In other words, $\overline{\rho^\oo_\mathcal{E}}$ is
	the desired universal family of formal solutions through which every
	other one factors uniquely, in the manner indicated above.

	Now we conclude by checking the monomorphism conditions. The
	commutative diagram in the hypothesis is obtained by taking the
	pullback square defining $\mathcal{E}^\oo$ and pasting to it the
	naturality square of the $\epsilon$ counit (definition~\ref{comonad})
	of $J^\oo_\Sigma$ (bottom), the counit-coproduct commutative triangle
	(definition~\ref{EMOverComonad}.2) of $J^\oo_\Sigma$ (right) and the
	composite morphism $e^\oo_\mathcal{E} = \epsilon_\mathcal{E} \circ
	\rho^\oo_\mathcal{E}$ (left). Since all the pasted subdiagrams
	commute, the whole diagram commutes as well, as was desired. Recall
	that both $\Delta_Y$ and $J^\oo_\Sigma e_Y$ are monomorphisms, hence
	their pullbacks $\rho^\oo_\mathcal{E}$ and $e^\oo_Y$ are also
	monomorphisms. The commutativity of the diagram in the hypothesis
	implies the identity $e^\oo_Y = e_Y \circ e^\oo_\mathcal{E}$. Hence,
	since $e^\oo_Y$ is a monomorphism, its factorization map through
	$e_Y$, that is $e^\oo_\mathcal{E}$, must also be a monomorphism.
\end{proof}

\section{Category theoretic background}

For reference, we collect here some standard facts from category theory
that are referred to in the main text. Unless indicated otherwise, proof
of these facts may be found in \cite{MacLane, Borceux}.

\subsection{Universal constructions}

We list some basic statements about (co-)limits and Kan extensions, see for instance \cite[vol 1, section 2]{Borceux}

\begin{proposition}[pasting law, e.g.,\ {\cite[vol 1, prop. 2.5.9]{Borceux}}]
  \label{PastingLaw}
  Consider in any category a commuting diagram of the from
  $$
    \xymatrix{
      A \ar[r] \ar[d] & B \ar[r]_{\ }="s" \ar[d] & C \ar[d]
      \\
      D \ar[r] & E \ar[r]^{\ }="t" & F
      \ar@{}|{\mbox{\tiny (pb)}} "s"; "t"
    }
  $$
  such that the right square is Cartesian (is a pullback square). Then the left
  square is Cartesian precisely if the total rectangle is.
\end{proposition}

\begin{definition}[e.g.,\ {\cite[III.6]{MacLane}}]
  \label{group}
  For $\mathcal{C}$ a category with finite products
  (hence with a terminal object $\ast$ and with binary Cartesian products $(-)\times (-)$),
  a \emph{monoid object}
  in $\mathcal{C}$ is an object $G \in \mathcal{C}$ equipped with
  \begin{enumerate}
    \item (unit) a morphism $e : \ast \longrightarrow G$
    \item (binary product) a morphism $(-)\cdot (-) : G \times G \longrightarrow G$;
  \end{enumerate}
  such that
  \begin{itemize}
    \item (associativity) the following diagram commutes:
    $$
      \xymatrix{
         G \times G \times G
         \ar[rr]^{(\mathrm{id},(-)\cdot (-))}
         \ar[d]_{((-)\cdot(-),\mathrm{id})}
         &&
         G \times G
         \ar[d]^{(-)\cdot (-)}
         \\
         G\times G
         \ar[rr]_{(-)\cdot (-)}
         &&
         G
      }
    $$
    \item
     (unitality) the following diagram commutes:
     $$
       \xymatrix{
         \ast \times G
         \ar[dr]_{\simeq}
         \ar[rr]^{e \times \mathrm{id}}
         &&
         G \times G
         \ar[dl]^{(-)\cdot (-)}
         \\
         & G
       }
     $$
  \end{itemize}
  Such a monoid object is called \emph{commutative} if
  \begin{itemize}
    \item the following diagram commutes:
    $$
      \xymatrix{
        G \times G
          \ar[dr]_{(-)\cdot (-)}
          \ar[rr]^{\tau}
          &&
        G \times G
        \ar[dl]^{(-)\cdot (-)}
        \\
        & G
      }
    $$
    (where the top morphism exchanges the two factors $(g_1,g_2) \mapsto (g_2,g_1)$).
  \end{itemize}
  A monoid object is a \emph{group object} if
  \begin{itemize}
    \item (inverses) there exists a morphism $(-)^{-1} : G \longrightarrow G$
  \end{itemize}
  such that
  \begin{itemize}
    \item
      (invertibility) the following diagram commutes
      $$
        \xymatrix{
          G
            \ar[rr]^{\mathrm{id} \times \mathrm{id}}
            \ar@{=}[d]
            &&
          G \times G
          \ar[d]^{\mathrm{id} \times (-)^{-1}}
          \\
          G
           \ar@{<-}[rr]_{(-) \cdot (-)}
           &&
          G \times G
        }
      $$
  \end{itemize}
  A monoid object that is both commutative as well as a group object is also called
  an \emph{abelian group object}.
\end{definition}

\begin{proposition}[nonabelian Mayer-Vietoris lemma]
  \label{MayerVietoris}
	Let $\mathcal{C}$ be a category with finite products, and let $G \in
	\mathcal{C}$ be equipped with the structure of a group object (definition
	\ref{group}). Then for $f : X \to G$ and $g : Y \to G$ two morphisms
	in $\mathcal{C}$, their fiber product also makes the following square
	Cartesian:
  $$
    \begin{gathered}
    \xymatrix{
      X \times_G Y
      \ar[d]_{(\mathrm{pr}_1, \mathrm{pr}_2)}
      \ar[rr]_-{\ }="s"
      &&
      \ast
      \ar[d]^{e}
      \\
      X \times Y
      \ar[r]_-{f \times g}
      &
      G \times G
      \ar[r]_-{(-)\cdot (-)^{-1}}
      &
      G
      \ar@{}|{\mbox{\tiny (pb)}} "s"; "s"+(0,-10)
    }
    \end{gathered}
    \, .
  $$
\end{proposition}

\begin{definition}[e.g.,\ {\cite[vol 1, section 3]{Borceux}}]
  \label{AdjointFunctor}
  A pair of \emph{adjoint functors}
  denoted $(L \dashv R)$ (or an \emph{adjunction} between two functors), is a pair of functors of the form
  $$
    \xymatrix{
      \mathcal{C}
      \ar@{<-}@<+4pt>[rr]^{L}
      \ar@<-4pt>[rr]_{R}
      &&
      \mathcal{D}
    }
  $$
  equipped with a natural isomorphism (``forming adjuncts'') between their hom-functors of the form
  $$
    \mathrm{Hom}_{\mathcal{C}}(L(-),-)
    \simeq
    \mathrm{Hom}_{\mathcal{D}}(-,R(-))
    \,.
  $$
  Here $L$ is called \emph{left adjoint to $R$} and $R$ is called \emph{right adjoint to $L$}.
  The image $\eta_d$ of $\mathrm{id}_{L d}$ under this isomorphism is called the \emph{unit} of the adjunction at $d \in \mathcal{D}$
  $$
    \eta_d : d \longrightarrow R(L(d))
    \,,
  $$
  while, conversely, the image $\epsilon_d$ of $\mathrm{id}_{R c}$ is called the \emph{counit}
  $$
    \epsilon_d : L(R(d)) \longrightarrow d
    \,.
  $$
  (Unit and counit are themselves natural transformations $\eta : \mathrm{id}_{\mathcal{D}} \longrightarrow R\circ L $
  and $\epsilon : L \circ R \longrightarrow \mathrm{id}_\mathcal{C}$.)
\end{definition}
One also writes horizontal lines for indicating these bijections between sets of adjunct morphisms:
\begin{center}
\begin{tabular}{rcl}
  $d$ & $\stackrel{}{\longrightarrow}$ & $R c$
  \\
  \hline
  $L d$ & $\stackrel{}{\longrightarrow}$ & $c$
\end{tabular}
\end{center}
There are various equivalent definitions of adjoint functors:
\begin{proposition}
  \label{AdjunctionInTermsOfUnitAndCounit}
  An adjunction $L \dashv R$ between two functors (definition \ref{AdjointFunctor})
  is equivalent to two natural transformations, the unit
  $$
    \eta : \mathrm{id} \longrightarrow R \circ L
  $$
  and the counit
  $$
    \epsilon : L \circ R \longrightarrow \mathrm{id},
  $$
  such that the following ``zig-zag''-identities hold:
  $$
    \xymatrix{
      L
        \ar[rr]^-{L (\eta)}
        \ar@/_2pc/[rrrr]_-{\mathrm{id}}
      &&
      L \circ R \circ L
        \ar[rr]^-{\epsilon_{L}}
      &&
      L
    }
  $$
  and
  $$
    \xymatrix{
      R \ar@/_2pc/[rrrr]_-{\mathrm{id}} \ar[rr]^{\eta_R} && R \circ L \circ R \ar[rr]^{R (\epsilon)} && R
    }
    \,.
  $$
\end{proposition}
Here is a list of basic properties of adjoint functors:
\begin{proposition}[e.g.,\ {\cite[vol 1, prop. 3.4.1]{Borceux}}]
  \label{AdjointsOfFullEmbeddings}
  For a pair of adjoint functors  $(L \dashv R)$ (definition \ref{AdjointFunctor}) the following are equivalent
  \begin{enumerate}
    \item the right adjoint $R$ is a fully faithful functor;
    \item the adjunction counit $L \circ R \longrightarrow \mathrm{id}$ is a natural isomorphism
  \end{enumerate}
  and similarly the following are equivalent:
  \begin{enumerate}
    \item the left adjoint $L$ is a fully faithful functor;
    \item the adjunction unit $\mathrm{id} \longrightarrow R \circ L$ is a natural  isomorphism.
  \end{enumerate}
\end{proposition}
\begin{proposition}[e.g.,\ {\cite[thm.V.5.1]{MacLane}}{\cite[vol 1, prop. 3.2.2]{Borceux}}]
  \label{rightadjointpreserveslimits}
	A right adjoint functor (definition~\ref{AdjointFunctor}) preserves
	all small limits. Dually, a left adjoint functor preserves all small
	colimits.
\end{proposition}
\begin{proposition}
  \label{AdjunctsInTermsOfUnitAndCounit}
  Given an adjunction $(L \dashv R)$ as in definition~\ref{AdjointFunctor},
  then
  \begin{itemize}
     \item the adjunct of a morphism of the form $f :d \longrightarrow R c$ is
       equivalently the composite
       $$
         \xymatrix{
           L d \ar[r]^{L(f)} & L R c \ar[r]^-{\epsilon_c} &  c
         };
       $$
     \item the adjunct of a morphism of the form $g : L c \longrightarrow d$
     is equivalently the composite
     $$
       \xymatrix{
         c \ar[r]^-{\eta_c} & R L c \ar[r]^{R(g)} & R d
       }
       \,.
     $$
  \end{itemize}
\end{proposition}

Key examples of adjoint pairs and adjoint triples are Kan extensions:
\begin{proposition}[Kan extension, e.g.,\ {\cite[vol 1, section 3.7]{Borceux}}]
  \label{FactsAboutLeftKanExtension}
  Given a functor $f : C \longrightarrow D$ between small categories, then
  the induced functor on categories of presheaves $f^\ast : \mathrm{PSh}(D) \longrightarrow \mathrm{PSh}(C)$
  (given by precomposing a presheaf with $f$) has both a left and a right adjoint (definition~\ref{AdjointFunctor}), denoted
  $f_!$ and $f_\ast$ respectively, and called the operations of left and right Kan extension along $f$.
  $$
    (f_! \dashv f^\ast \dashv f_\ast)
    :
    \xymatrix{
      \mathrm{PSh}(C)
      \ar@<+6pt>[rr]^{f_!}
      \ar@{<-}@<+0pt>[rr]|{f^\ast}
      \ar@<-6pt>[rr]_{f_\ast}
      &&
      \mathrm{PSh}(D)
    }
    \,.
  $$
  Moreover, the left Kan extension of a presheaf $A \in \mathrm{PSh}(C)$ is equivalently the presheaf which to any object $d\in D$
  assigns the set expressed by the coend formula
  $$
    (f_! A)(d) \simeq \int^{c \in C} \mathrm{Hom}_D(d,f(c))\times \mathrm{Hom}_{\mathrm{PSh}(C)}(c,A)
    \,,
  $$
  where on the right we are identifying $c$ with the presheaf that it represents.
  Explicitly, this coend gives the set of equivalence classes of pairs of morphisms
  $$
    (d \to f(c), c \to A)
  $$
  where two such pairs are regarded as equivalent if there is a morphism $\phi : c_1 \to c_2$ in $C$ such that
  the following two triangles commute
  $$
    \begin{gathered}
    \xymatrix@R=8pt{
      & d
      \ar[dddl]
      \ar[dddr]
      \\
      \\
      \\
      f(c_1) \ar[rr]^{f(\phi)} && f(c_2)
      \\
      c_1 \ar[dddr] \ar[rr]^-{\phi} && c_2 \ar[dddl]
      \\
      \\
      \\
      & A
    }
    \end{gathered}
    \, .
  $$
  In particular the left Kan extension of a representable presheaf $y(c)$ is the presheaf represented by the image under
  the given functor $f$ of the representing object $c$:
  $$
    f_!(y(c)) \simeq y(f(x))
    \,.
  $$
\end{proposition}
\begin{remark}
For a presheaf $A$ on $\mathcal{C},$ the presheaf $f_*\left(A\right)$ can also be described explicitly. Namely, for $D$ an object of $\mathcal{D},$ one has $f_*\left(A\right)\left(D\right)\cong\mathrm{Hom}_{\mathrm{Psh}\left(\mathcal{C}\right)}\left(f^*y\left(D\right),A\right).$ This follows easily from the Yoneda lemma.
\end{remark}
\begin{proposition}
  \label{LeftKanExtensionAlongEmbedding}
  If $f : C \hookrightarrow D$ is a fully faithful functor, then so is its left Kan extension
  (proposition \ref{FactsAboutLeftKanExtension})
  $f_! : \mathrm{PSh}(C) \hookrightarrow \mathrm{PSh}(D)$, hence (by proposition \ref{AdjointsOfFullEmbeddings} then
  the adjunction unit $\mathrm{id} \overset{\simeq}{\longrightarrow} f^\ast f_!$
  is a natural isomorphism.
\end{proposition}
\begin{definition}
  \label{SliceCategory}
  Given a category $\mathcal{C}$ and an object
  $c\in \mathcal{C}$, then the \emph{slice category} $\mathcal{C}_{/c}$
  has as objects the morphisms of $\mathcal{C}$ into $c$, and as morphisms between these the
  commuting triangles in $\mathcal{C}$ of the form
  $$
    \begin{gathered}
    \xymatrix{
      a_1 \ar[dr]_{f_1} \ar[rr] && a_2 \ar[dl]^{f_2}
      \\
      & c
    }
    \end{gathered}
    \,.
  $$
\end{definition}
\begin{proposition}
  The hom-spaces in a slice category $\mathcal{C}_{/c}$, definition \ref{SliceCategory}
  are equivalently given by the fiber product:
  $$
    \mathcal{C}_{/c}(f_1,f_2)
      \simeq
    \mathcal{C}(a_1,a_2) \underset{\mathcal{C}(a_2,c)}{\times} \{f_2\}
  $$
  of hom-spaces in $\mathcal{C}$:
  $$
    \begin{gathered}
    \xymatrix{
      \mathcal{C}_{/c}(f_1,f_2)
      \ar@{}[dr]|{\mbox{\tiny (pb)}} \ar[d]\ar[r]
      &
      \mathcal{C}(a_1, a_2)
      \ar[d]
      \\
      \ast
      \ar[r]_{\tilde f_2}
      &
      \mathcal{C}(a_1,c)
    }
    \end{gathered}
    \,.
  $$
  where $\tilde f_2$ picks the element $f_2$ in $\mathcal{C}(a_2,c)$.
\end{proposition}
\begin{example}
  If $\ast \in \mathcal{C}$ is a terminal object, then there is an equivalence of categories
  $$
    \mathcal{C}_{/\ast}\simeq \mathcal{C}
  $$
  between the slice category over $\ast$ (definition \ref{SliceCategory}) and the original category.
\end{example}
\begin{example}
  \label{SliceCartesianProduct}
  If $\mathcal{C}$ is a category with finite limits, then for every object $c \in \mathcal{C}$
  the slice category $\mathcal{C}_{/c}$ (definition \ref{SliceCategory}) has terminal object given by
  $$
    [ c \stackrel{\mathrm{id}}{\longrightarrow} c ]
  $$
  and with Cartesian product given by the fiber product over $c$ in $\mathcal{C}$:
  $$
    [ a \to c ]
    \times
    [ b \to c ]
    \;\simeq\;
    [ a \times_c b \to c ]
    \, .
  $$
\end{example}
\begin{example}
  \label{SectionsOfBundles}
  If $\mathcal{C}$ is a category thought of as a category of spaces
  (as in section \ref{CategoriesOfSPaces}) then for $\Sigma \in \mathcal{C}$ any object,
  thought of as a base space, we may think of the slice category $\mathcal{C}_{/\Sigma}$
  (definition \ref{SliceCategory}) as the \emph{category of bundles} over $\Sigma$,
  in the generality where bundles are not required to be fiber bundles and in particular may have empty fibers.
  For
  $$
    [ E \stackrel{p}{\longrightarrow} \Sigma ]
    \in
    \mathcal{C}_{/\Sigma}
  $$
  any such bundle, we have that bundle morphisms from the terminal bundle
  $$
    [ \Sigma \stackrel{\mathrm{id}}{\longrightarrow} \Sigma ]
    \in
    \mathcal{C}_{/\Sigma}
  $$
  are equivalently sections of the bundle $ E \stackrel{p}{\to} \Sigma$. We write
  $$
    \Gamma_{\Sigma}(E)
     :=
    \mathrm{Hom}_{\mathcal{C}_{/\Sigma}}\left(
    [ \Sigma \stackrel{\mathrm{id}}{\longrightarrow} \Sigma ]
    \;,\;
    [ E \stackrel{p}{\longrightarrow} \Sigma ]
    \right)
  $$
\end{example}

\subsection{Categories of sheaves}

For reference, we collect here some facts about categories of sheaves (Grothendieck toposes) that we use in the main text,
see for instance \cite[vol 3]{Borceux}

\begin{definition}[site]
 \label{site}
 For $\mathcal{C}$ a small category, a \emph{coverage} or
 \emph{Grothendieck pre-topology} on $\mathcal{C}$ is for each
 object $X \in \mathcal{C}$ a set of families of morphisms
 $\{U_i \overset{\phi_i}{\longrightarrow} X\}_{i \in I}$ into $X$,
 called the \emph{covering families}, which are such that for each morphism
 $Y \longrightarrow X$ there exists a covering family
 $\{V_i \overset{\psi_j}{\longrightarrow} Y\}_{j \in J}$ such that for each
 $i \in I$ there is a commuting square of the form
 $$
   \begin{gathered}
   \xymatrix{
     V_{j(i)}
     \ar[d]_{\psi_{j(i)}}
     \ar[r]
     &
     U_i
     \ar[d]^{\phi_i}
     \\
     Y \ar[r] & X
   }
   \end{gathered}
   \,.
 $$
 A small category $\mathcal{C}$ equipped with a coverage is called a \emph{site}.
\end{definition}
\begin{definition}
  \label{sheaf}
  For $\mathcal{S}$ a small site (definition \ref{site})
  then a presheaf $F \colon \mathcal{S}^{\mathrm{op}} \longrightarrow \mathrm{Set}$
  is called a \emph{sheaf} if for all covering families $\{U_i \overset{\phi_i}{\longrightarrow} U\}_{i \in I}$
  in $\mathcal{S}$ and for all tuples of elements $(x_i \in F(U_i))_{i \in I}$
  and for all pairs of morphisms $g \colon V \to U_{i_1}$ and $h \colon V \to U_{i_2}$
  in $\mathcal{S}$ such that $\phi_{i_1} \circ g = \phi_{i_2} \circ h$  and such that
  $F(g)(x_{i_1}) = F(h)(x_{i_2})$ then there exists a unique $x \in F(U)$ such that
  $x_i = F(\phi_i)(x)$ for all $i \in I$.

  We write
  $$
    \mathrm{Sh}(\mathcal{S})
      \hookrightarrow
    \mathrm{PSh}(\mathcal{S})
  $$
  for the full inclusion of the sheaves into the category of presheaves. We call this the
  \emph{category of sheaves} or the \emph{sheaf topos} or the \emph{Grothendieck topos} over $\mathcal{S}$.
\end{definition}
The following is an important structure theorem for categories of sheaves:
\begin{proposition}[{e.g.,\ \cite[vol. 3, cor. 3.5.5 with vol 1, def. 3.5.5]{Borceux}}]
  \label{IncludingSheavesIntoPresheaves}
	Given a small category $\mathcal{S}$ with the structure of a site
	(definition \ref{site}) then the full inclusion of the sheaves into
	the presheaves (definition \ref{sheaf})
	has a left adjoint functor $L_{\mathcal{S}}$ (`sheafification`'')
	which preserves finite limits
  $$
    \xymatrix{
      \mathrm{Sh}(\mathcal{S})
        \ar@{^{(}->}@<-5pt>[rr]_-{i}
        \ar@{<--}@<+5pt>[rr]^-{L_{\mathcal{S}}}_-{\bot}
        &&
      \mathrm{PSh}(\mathcal{S})
    }
    \,.
  $$
  Conversely, every full inclusion of this form into a category of presheaves on some small category $\mathcal{C}$
  is the inclusion of sheaves for some site structure on $\mathcal{C}$.
\end{proposition}
Here is a list of some basic extra conditions on sites:
\begin{definition}
  \label{subcanonical}
  A site $\mathcal{C}$ (definition \ref{site}) is called \emph{sub-canonical}
  if every representable presheaf is a sheaf (definition \ref{sheaf}), i.e.,\ if the Yoneda embedding
  factors through the full inclusion of sheaves (proposition \ref{IncludingSheavesIntoPresheaves})
  $$
    \xymatrix{
      \mathcal{C}
        \ar@{-->}[r]
        \ar@/_1pc/[rr]_{y}
        &
      \mathrm{Sh}(\mathcal{C})
        \ar@{^{(}->}[r]
        &
      \mathrm{PSh}(\mathcal{C})
    }
    \,.
  $$
\end{definition}
\begin{definition}
  \label{SitePoint}
  Let $\mathcal{C}$ be a small site (definition \ref{site}).
  \begin{enumerate}
  \item A \emph{point of the site}
  is a geometric morphism from the base topos to its sheaf topos (definition \ref{sheaf}), hence a pair of adjoint functors $(x^\ast \dashv x_\ast)$ (definition \ref{AdjointFunctor})
  of the form
  $$
    \xymatrix{
      \mathrm{Set}
        \ar@<-3pt>[rr]_{x_\ast}^{\bot}
        \ar@{<-}@<+7pt>[rr]^{x^\ast}
        &&
      \mathrm{Sh}(\mathcal{C})
    }
  $$
  and such that $x^\ast$ preserves finite limits. In this case left adjoint $x^\ast$ is called
  \emph{forming the stalk at $x$}. Hence for $X \in \mathrm{Sh}(\mathcal{C})$ any sheaf,
  $x^\ast X \in Set$ is called the \emph{stalk} of $X$ at $x$.
  \item
  The site $\mathcal{C}$ is said to have \emph{enough points} if there exists a set $\{x_i\}_{i \in I}$
  of points, in the above sense, such that a morphism $ f: X \to Y$ in $\mathrm{Sh}(\mathcal{C})$
  is an isomorphism precisely if all its stalks $x_i^\ast f : x_i^\ast X \to x_i^\ast Y$ are bijections
  of sets.
  \end{enumerate}
\end{definition}
The following is a list of some basic properties of categories of sheaves that we need in the main text:
\begin{proposition}
  \label{EpimorphismsInCategoriesOfSheaves}
  Let $S$ be a small site (definition \ref{site}), and let $f : X \to Y$ be a morphism in the
  category of sheaves $\mathrm{Sh}(S)$ over it (definition \ref{sheaf}). Then  in generality:
  \begin{enumerate}
    \item $f$ is a monomorphism or isomorphism precisely it is so globally, hence
    if for all object $U \in S$ its component $f_U : X(U) \to Y(U)$ is an
       injection or bijection of sets, respectively;
    \item $f$ is an epimorphism precisely if it is so locally, hence if for each object $U \in S$ there exists a
     cover $\{U_i \stackrel{\phi_i}{\to} U\}_i$ such that for each element $y_U \in Y(U)$
     its restriction $y_{U_i} := f(\phi_i)(y_U)$ is in the image of $f_{U_i} : X(U_i) \to Y(U_i)$
     for all $i$.
  \end{enumerate}
  But if $S$ has \emph{enough points} $\{x_i\}_{i \in I}$ in the sense of definition \ref{SitePoint},
  then $f$ is an epi-/iso-/mono-morphisms precisely if it is so stalkwise, hence precisely if for
  each $i \in I$ then $x_i^\ast f : x_i^\ast X \to x_i^\ast Y$ is a surjection/bijection/injection of sets,
  respectively.
\end{proposition}
\begin{proposition}[{e.g.,\ {\cite[vol 3, prop. 3.4.11, 3.4.13, 3.4.15]{Borceux}}}]
  \label{RegularMonomorphismInSheafTopos}
  In a category of sheaves $\mathrm{Sh}(\mathcal{C})$ (definition \ref{sheaf})
  \begin{enumerate}
    \item every epimorphism $f \colon X \to Y$ is regular and indeed effective, meaning that it is the
    coequalizer of its kernel pair, hence that
    $\xymatrix{ X \underset{Y}{\sqcup} X \ar@<+3pt>[r] \ar@<-3pt>[r] & X \ar[r]^-f & Y  }$
    is a colimiting co-cone;
    \item every monomorphism $f \colon X \to Y$ is regular, meaning that it is the equalizer of some pair
    of parallel morphisms.
  \end{enumerate}
\end{proposition}
\begin{proposition}[{universal colimits, e.g.,\ \cite[vol 3, prop. 3.4.4]{Borceux}}]
  \label{UniversalColimits}
	In a category of sheaves $\mathrm{Sh}(\mathcal{C})$ (definition
	\ref{sheaf}), colimits are compatible with fiber products : If $X : I
	\longrightarrow \mathrm{Sh}(\mathcal{C})$ is a diagram and
	$\varinjlim_i X_i \longrightarrow B$ a morphism
	out of its colimit, then for every morphism $f : A \longrightarrow B$
	the following square is Cartesian (is a pullback square)
  $$
    \begin{gathered}
    \xymatrix{
      \varinjlim_i f^\ast X_i
        \ar[r] \ar[d] \ar@{}[dr]|{\mbox{\tiny (pb)}} &
      \varinjlim_i X_i
      \ar[d]
      \\
      A
       \ar[r]_f
       &
      B
    }
    \end{gathered}
    \,.
  $$
\end{proposition}

\begin{proposition}[base change]
 \label{basechange}
  \label{BaseChangeAdjunctions}
 Let $\mathbf{H} = \mathrm{Sh}(\mathcal{C})$ be a category of sheaves (definition \ref{sheaf}). Then for
 $$
   f : X \longrightarrow Y
 $$
 any morphism in $\mathbf{H}$, there is an adjoint triple of functors (definition \ref{AdjointFunctor}) between the slice categories
 (definition \ref{SliceCategory})
 $$
   (f_! \dashv f^\ast \dashv f_\ast)
   \;:\;
   \xymatrix{
     \mathbf{H}_{/X}
       \ar@<+10pt>[rr]^-{f_!}
       \ar@{<-}[rr]|-{f^\ast}
       \ar@<-10pt>[rr]_-{f_!}
     &&
     \mathbf{H}_{/Y}
   }
   \,,
 $$
 where
 \begin{enumerate}
  \item $f_!$ (left push-forward) is given by post-composition with $f$ in $\mathbf{H}$;
  \item $f^\ast$ is given by pullback along $f$ in $\mathbf{H}$.
\end{enumerate}
Accordingly, by example \ref{AdjunctionGivesMonad} composition of left and right pushforward with pullback, respectively, yields
an adjoint pair of a monad $L_f$ and a comonad $R_f$ (definition \ref{monad})
$$
  (L_f \dashv R_f)
   \; := \;
   \left(
     (\eta_f)^\ast \circ (\eta_f)_!
     \,\dashv\,
     (\eta_f)^\ast \circ (\eta_f)_\ast
   \right)
   \;:\;
     \mathbf{H}_{/X}
     \longrightarrow
     \mathbf{H}_{/X}
   \,.
$$
\end{proposition}

\begin{proposition}
 \label{SliceSite}
 Let $\mathcal{C}$ be a small category equipped with a subcanonical coverage. Let $X \in \mathrm{Sh}(\mathcal{C})$
 be an object in the category of sheaves over $\mathcal{C}$. Write $\mathcal{C}/X$ for the category
 whose objects are morphisms $y(c) \to X$ in $\mathrm{Sh}(\mathcal{C})$, with $c\in \mathcal{C}$
 (and $y$ denoting the Yoneda embedding) and whose morphisms are commuting triangles
 $$
   \begin{gathered}
   \xymatrix{
     y(c_1) \ar[dr] \ar[rr] && y(c_2) \ar[dl]
     \\
     & X
   }
   \end{gathered}
   \,.
 $$
 Say that a set of morphism in $\mathcal{C}/X$ is a covering family if the underlying horizontal morphisms
 in $\mathcal{C}$ form a covering family. Then sheaves on $\mathcal{C}/X$ are equivalently sheaves on $\mathcal{C}$
 sliced over $X$:
 $$
   \mathrm{Sh}(\mathcal{C}/X) \simeq \mathrm{Sh}(\mathcal{C})/X
   \,.
 $$
\end{proposition}

\begin{proposition}
  \label{AmbientCategoryGrothedieckTopology}
  Let $\mathcal{S}$ be a small site (definition \ref{site}) and let $i : \mathcal{S} \hookrightarrow \mathcal{C}$
  be a fully faithful inclusion of the underlying category into a small category $\mathcal{C}$.
  Then there exists a Grothendieck topology on $\mathcal{C}$
  and an equivalence of categories of sheaves (definition \ref{sheaf}) of the form
  $$
    \xymatrix{
      \mathrm{Sh}(\mathcal{S})
      \ar[r]^-{\simeq}
      &
      \mathrm{Sh}(\mathcal{C})
    }
    \,.
  $$

  If moreover the functor
  $$
    \mathcal{C}
      \overset{y_{\mathcal{C}}}{\longrightarrow}
    \mathrm{PSh}(\mathcal{C})
      \overset{i^\ast}{\longrightarrow}
    \mathrm{PSh}(\mathcal{S})
      \overset{L_{\mathcal{S}}}{\longrightarrow}
    \mathrm{Sh}(\mathcal{S})
  $$
  is fully faithful ($y$ denotes Yoneda embedding, $i^\ast$ denotes restriction of presheaves,
  proposition \ref{BaseChangeAdjunctions},
  and $L$ denotes sheafification, definition \ref{IncludingSheavesIntoPresheaves}),
  then this Grothendieck topology is subcanonical.
\end{proposition}
\begin{proof}
  Consider\footnote{We are grateful to Dave Carchedi for providing this argument.} the composition of pairs of adjoint functors
  $$
    \xymatrix{
      \mathrm{Sh}(\mathcal{S})
      \ar@<-5pt>@{^{(}->}[rr]^{\bot}_{i_{\mathcal{S}}}
      \ar@<+5pt>@{<-}[rr]^{L_{\mathcal{S}}}
      &&
      \mathrm{PSh}(\mathcal{S})
      \ar@<-5pt>@{^{(}->}[rr]_{i_\ast}^{\bot}
      \ar@<+5pt>@{<-}[rr]^{i^\ast}
      &&
      \mathrm{PSh}(\mathcal{C})
    }
    \,,
  $$
	where on the left we have the sheafification adjunction over
	$\mathcal{S}$ (proposition \ref{IncludingSheavesIntoPresheaves}),
	while on the right we have the pullback / right Kan extension
	adjunction along $i$ (proposition \ref{FactsAboutLeftKanExtension}).
	Here $L$ preserves finite limits (by proposition
	\ref{IncludingSheavesIntoPresheaves}) and $i^\ast$ preserves all
	limits (since it has a further left adjoint given by left Kan
	extension). Hence the composite adjunction is a geometric embedding
  $$
    \xymatrix{
      \mathrm{Sh}(\mathcal{S})
        \simeq
      \mathrm{Sh}(\mathcal{C})
        \ar@<-5pt>@{^{(}->}[rr]^-{\bot}_-{i_{\mathcal{C}}}
        \ar@<+5pt>@{<-}[rr]^-{L_{\mathcal{C}}}
        &&
      \mathrm{PSh}(\mathcal{C})
    }
  $$
  of the topos $\mathrm{Sh}(\mathcal{S})$
  into the presheaf topos $\mathrm{PSh}(\mathcal{C})$. By proposition \ref{IncludingSheavesIntoPresheaves} every such
  corresponds to the sheafification adjunction for some Grothendieck topology on
  $\mathcal{C}$, identifying the former with the category of sheaves
  over the latter.

  In particular
  $$
    L_{\mathcal{S}} \circ i^\ast \circ y_{\mathcal{C}}
    \simeq
    L_{\mathcal{C}} \circ y_{\mathcal{C}}
    \,,
  $$
  and hence if
  $L_{\mathcal{S}} \circ i^\ast \circ y_{\mathcal{C}}$ is fully faithful
  then so is $L_{\mathcal{C}} \circ y_{\mathcal{C}}$ and hence also
  $i_{\mathcal{C}} \circ L_{\mathcal{C}} \circ y_{\mathcal{C}}$.
  This means that for $C \in \mathcal{C}$, the sheafification of $y(\mathcal{C})$
  is the presheaf given on $D \in \mathcal{C}$ by
  $$
    \begin{aligned}
      (L_{\mathcal{C}} \circ y_{\mathcal{C}}(C))(D)
      & \simeq
      (i_{\mathcal{C}} \circ  L_{\mathcal{C}} \circ y_{\mathcal{C}}(C))(D)
      \\
      & \simeq
      \mathrm{Hom}_{\mathrm{PSh}(\mathcal{C})}(y_{\mathcal{C}}(D), i_{\mathcal{C}} \circ  L_{\mathcal{C}} \circ y_{\mathcal{C}}(C) )
      \\
      & \simeq
      \mathrm{Hom}_{\mathrm{Sh}(\mathcal{C})}(L_{\mathcal{C}} \circ y_{\mathcal{C}}(D), L_{\mathcal{C}} \circ y_{\mathcal{C}}(C) )
      \\
      & =
      \mathrm{Hom}_{\mathrm{PSh}(\mathcal{C})}(i_{\mathcal{C}} \circ L_{\mathcal{C}} \circ y_{\mathcal{C}}(D), i_{\mathcal{C}} \circ L_{\mathcal{C}} \circ y_{\mathcal{C}}(C) )
      \\
      & \simeq
      \mathrm{Hom}_{\mathcal{C}}(C,D)
      \\
      & \simeq
      y_{\mathcal{C}}(D)
      \,.
    \end{aligned}
  $$
  This says that $y_{\mathcal{C}}(C)$ coincides with its sheafification, hence that the Grothendieck
  topology on $\mathcal{C}$ is subcanonical.
\end{proof}

\subsection{Monads}
\label{monads}

We collect some basic facts on (co-)monads and (co-)monadic descent, see for instance \cite[vol 2, section 4]{Borceux}.

\begin{definition}
  \label{monad}
  \label{comonad}
  For $\mathcal{C}$ a category, a \emph{monad} on $\mathcal{C}$
  is an endofunctor
  $
    T \colon \mathcal{C} \to \mathcal{C}
  $
  equipped with natural transformations
  \begin{itemize}
    \item (product) $\nabla : T \circ T \longrightarrow T $;
    \item (unit) $\eta :  \mathrm{id}_{\mathcal{C}} \longrightarrow T$
  \end{itemize}
	such that these satisfy the following associativity and unitality
	properties:
  $$
    \begin{gathered}
  	\xymatrix{
  		&
  		T
				\ar@{<-}[dl]_{\mathrm{id}_\mathcal{C}}
				\ar@{<-}[dr]^{\mathrm{id}_\mathcal{C}}
				\ar@{<-}[d]^{\nabla}
			\\
  		T &
  		TT
  		\ar@{<-}[l]^{\eta(T)}
  		\ar@{<-}[r]_{T(\eta)} &
  		T
  	}
    \end{gathered}
  	\qquad \text{and} \qquad
    \begin{gathered}
  	\xymatrix{
  		T
				\ar@{<-}[r]^{\nabla}
				\ar@{<-}[d]_{\nabla} &
			TT
				\ar@{<-}[d]^{\nabla(T)}
			\\
  		TT
				\ar@{<-}[r]_{T(\nabla)} &
  		TTT
  	}
    \end{gathered}
  	\, .
  $$

	Dually, a \emph{comonad} on $\mathcal{C}$ is an endofunctor $J\colon
	\mathcal{C} \to \mathcal{C}$ equipped with natural transformations
  \begin{itemize}
    \item (coproduct) $\Delta : J \longrightarrow J \circ J$
    \item (counit) $\epsilon : J \longrightarrow \mathrm{id}_{\mathcal{C}}$
  \end{itemize}
	such that these satisfy the following coassociativity and counitality
	properties:
  $$
    \begin{gathered}
  	\xymatrix{
  		&
  		J
				\ar[dl]_{\mathrm{id}_\mathcal{C}}
				\ar[dr]^{\mathrm{id}_\mathcal{C}}
				\ar[d]^{\Delta}
			\\
  		J &
  		JJ
  		\ar[l]^{\epsilon(J)}
  		\ar[r]_{J(\epsilon)} &
  		J
  	}
    \end{gathered}
  	\qquad \text{and} \qquad
    \begin{gathered}
  	\xymatrix{
  		J
				\ar[r]^{\Delta}
				\ar[d]_{\Delta} &
			JJ
				\ar[d]^{\Delta(J)}
			\\
  		JJ
				\ar[r]_{J(\Delta)} &
  		JJJ
  	}
    \end{gathered}
  	\, .
  $$
\end{definition}
\begin{example}
  \label{AdjunctionGivesMonad}
	Given a pair of adjoint functors $(L \dashv R)$ (definition
	\ref{AdjointFunctor}) then
  \begin{enumerate}
    \item $R \circ L$ becomes a monad (definition \ref{monad})
    by taking the monad unit to be the adjunction unit, and taking the monad product to be
    the image of the adjunction counit on $L$ under $R$:
    $$
    	\xymatrix{
      (R \circ L) \circ (R \circ L)
        \ar[r]^-{R \epsilon_L}
      &
      R \circ L
      }
    $$
    \item $L \circ R$ becomes a comonad (definition \ref{comonad})
    by taking the comonad counit to be the adjunction counit, and taking the comonad coproduct to be
    the image of the adjunction unit on $R$ under $L$:
    $$
    	\xymatrix{
      L \circ R
        \ar[r]^-{L \eta_R}
      &
      (L \circ R) \circ (L \circ R)
      }
      \,.
    $$
  \end{enumerate}
\end{example}
\begin{definition}
  \label{EMOverComonad}
  Given a comonad $(J,\epsilon,\Delta)$ on $\mathcal{C}$, definition~\ref{comonad},
  a \emph{coalgebra} over the comonad is an object $E \in \mathcal{C}$ equipped with a morphism
  $$
    \rho \colon E \longrightarrow J E
  $$
  such that
  \begin{enumerate}
    \item (coaction property) the following diagram commutes:
    $$
      \xymatrix{
        E \ar[r]^{\rho} \ar[d]_\rho
        & J E
        \ar[d]^{J\rho}
        \\
        J E \ar[r]_{\Delta_E} & J J E
      }
    $$
    \item
     (counitality) the following diagram commutes:
     $$
       \xymatrix{
          & E
          \ar[dr]^{\rho}
          \ar[dl]_{\mathrm{id}}
          \\
          E & & J E \ar[ll]_{\epsilon_{E}}
       }
     $$
  \end{enumerate}

	A homomorphism of coalgebras $f : (E_1, \rho_1) \longrightarrow
	(E_2,\rho_2)$ is a morphism $f : E_1 \longrightarrow E_2$ in
	$\mathcal{C}$ which respects these coaction morphisms in that the
	following diagram commutes:
  $$
    \begin{gathered}
    \xymatrix{
      E_1 \ar[r]^f \ar[d]_{\rho_1} & E_2 \ar[d]^{\rho_2}
      \\
      J E_1 \ar[r]_{J f} & J E_2
    }
    \end{gathered}
    \, .
  $$
  The resulting category of coalgebras is denoted  $\mathrm{EM}(J)$
  (for ``Eilenberg-Moore category'').
\end{definition}
\begin{proposition}[Beck equalizer]
  \label{BeckEqualizer}
  For $J \colon \mathcal{C} \longrightarrow \mathcal{C}$
  any comonad (definition \ref{comonad}) and
  $\rho : E \longrightarrow J E$ a coalgebra over $J$ (definition \ref{EMOverComonad})
  then the diagram
  $$
    \xymatrix{
      E
       \ar[rr]^-\rho
       &&
      J E
        \ar@<+3pt>[rr]^-{\Delta_E}
        \ar@<-3pt>[rr]_-{J \rho}
        &&
      J J E
    }
  $$
  is an equalizer diagram, in fact it is an absolute equalizer
  (meaning that it is preserved by every functor $F : \mathcal{C} \to \mathcal{D}$).
  In particular therefore $\rho$ is a monomorphism.
\end{proposition}
\begin{proposition}
  \label{CoMonadsFromAdjunctions}
  For $(L \dashv R) : \xymatrix{
      \mathcal{C}
      \ar@{<-}@<+4pt>[r]^{L}
      \ar@<-4pt>[r]_{R}
      &
      \mathcal{D}
    }
$ an adjunction, definition~\ref{AdjointFunctor}, the endofunctor
$$
  T := L \circ R \;:\; \mathcal{C} \longrightarrow \mathcal{C}
$$
becomes a comonad on $\mathcal{C}$ (definition~\ref{comonad}) with
counit the adjunction counit $ L \circ R \to \mathrm{id}_{\mathcal{C}}$ (definition~\ref{AdjointFunctor}), and with coproduct
induced from the unit of the adjunction by
$$
  \Delta_T := L(\eta_{R(-)})
  \,.
$$

Dually, $R \circ L$ is canonically equipped with the structure of a monad.
\end{proposition}
\begin{example}
  \label{AdjointPairFromAdjointTriple}
  Given an adjoint triple $(L\ \dashv C \dashv R)$ of functors (definition \ref{AdjointFunctor})
  then the monad $T := C\circ L$ and the comonad $J := C \circ R$
  (induced via proposition~\ref{CoMonadsFromAdjunctions}) themselves form an adjoint pair:
  $$
    (T  \dashv J  ) : \mathcal{C} \longrightarrow \mathcal{C}
    \,.
  $$
\end{example}
\begin{proposition}
  \label{PropertiesOfAdjointPairsFromAdjointTriples}
	In the situation of example \ref{AdjointPairFromAdjointTriple} then
	the double $(T \dashv J)$-adjunct $\tilde{\tilde{f}}$ (definition
	\ref{AdjointFunctor}) of a morphism $f : X \to J J Y$ of the form
  $$
    \xymatrix{
    	&& J Y \ar[d]^{\Delta_Y}
    	\\
      X
        \ar[urr]^g
        \ar[rr]_{f}
        &&
      J J Y
    }
  $$
  is given by the $(T \dashv J)$-adjunct $\tilde g$ of $g$ via
  $$
    \begin{gathered}
    \xymatrix{
      T T X
      \ar[rr]^{\tilde{\tilde{f}}}
      \ar[d]_{\nabla_X}
      && Y
      \\
      T X \ar[urr]_-{\tilde g}
    }
    \end{gathered}
    \,.
  $$
	In fact both these kinds of morphisms are in natural bijection with
	those of the form
  $$
    \xymatrix{
      T X \simeq C L X \ar[rr]^-{C \overline{g}} && C R X \simeq J X
    }
    \,,
  $$
  where $\overline{g}$ denotes the $(C \dashv R)$-adjunct of $g$.
\end{proposition}
\begin{proof}
  By definition of the monad unit (definition \ref{CoMonadsFromAdjunctions}),
  $f$ is of the form
  $$
    \xymatrix{
      X \ar[rr]^-g && C R Y \ar[rr]^{C (\eta_{R Y}) } && C R C R Y
    }
    \,.
  $$
  By composition of adjunctions, the $(C L \dashv C R)$-adjunct of this morphism is the
  $(L \dashv C)$-adjunct of its $(C \dashv R)$-adjunct.
  Since the morphism on the right is in the image of $C$, the naturality of the adjunction isomorphism
  (definition \ref{AdjointFunctor}) implies that $f$ is in natural bijection to
  $$
    \xymatrix{
      L X \ar[rr]^-{\overline{g}} && R Y \ar[rr]^-{\eta_{R Y}} && R C R Y
      \,.
    }
  $$
  Now by the formula for adjuncts from proposition \ref{AdjunctsInTermsOfUnitAndCounit}, the further
  $(C \dashv R)$-adjunct of this morphism is
  $$
    \xymatrix{
      C L X
      \ar[rr]^-{C \overline{g}}
       &&
      C R Y
        \ar[rr]^-{C \eta_{R Y}}
        \ar@/_2pc/[rrrr]_{id}
      &&
      C R C R Y
        \ar[rr]^{\epsilon_{C R Y}}
      &&
      C R Y
    }
    \,.
  $$
  Here the composite on the right is the identity morphism, as shown, by proposition \ref{AdjunctionInTermsOfUnitAndCounit}.
  This shows that the single $(T \dashv J)$-adjunct of $f$ is $C \overline{g}$. In complete formal duality
  one finds that also the single $(T \dashv J)$-adjunct (in the other direction) of $T T X \overset{\nabla_X}{\longrightarrow} T X \overset{\tilde g}{\longrightarrow} Y$
  is $C \overline{g}$. Hence the statement follows.
\end{proof}
\begin{definition}
  \label{coKleisli}
  The full subcategory of $\mathrm{EM}(J)$ on the cofree coalgebras,
  i.e.,\ on the
  objects in the image of $\mathrm{F}$, proposition~\ref{CofreeForgetAdjunction},
  is denoted $\mathrm{Kl}(J)$ (from ``Kleisli category'', or more
  appropriately ``co-Kleisli category'' given that $J$ is a comonad).
\end{definition}
\begin{proposition}
  \label{CofreeForgetAdjunction}
	The category of coalgebras over a comonad on a category $\mathcal{C}$,
	definition~\ref{EMOverComonad}, is related to $\mathcal{C}$ by a pair
	of adjoint functors, definition~\ref{AdjointFunctor}, of the form
  $$
    (\mathrm{U} \dashv \mathrm{F})
    :
    \xymatrix{
      \mathcal{C}
      \ar@{<-}@<+3pt>[rr]^{\mathrm{U}}
      \ar@<-3pt>[rr]_{\mathrm{F}}
      \ar@{->>}[dr]
      &&
      \mathrm{EM}(J)
      \\
      & \mathrm{Kl}(J) \ar@{^{(}->}[ur]
    }
    \,,
  $$
	where the left adjoint $U$ (``underlying'') forgets the coalgebra
	structure, $U : (E,\rho) \mapsto E$, while the right adjoint $F$
	(``cofree'') sends an object $c \in \mathcal{C}$ to the object $J
	c$ with coaction given by the coproduct $\Delta_J$. The comonad
	induced from this adjunction via
	proposition~\ref{CoMonadsFromAdjunctions} coincides with $J$:
  $$
    J \simeq \mathrm{U} \circ \mathrm{F}
    \,.
  $$
\end{proposition}
\begin{remark}
  \label{compositionalaKleisli}
	In the situation of proposition~\ref{CofreeForgetAdjunction}, given
	objects $c_1,c_2 \in \mathcal{C}$, we have, by adjunction, a
	bijection of morphisms of the form
  $$
  \mbox{
  \begin{tabular}{rcl}
    $F c_1$ & $\stackrel{f}{\longrightarrow}$ &  $F c_2$
    \\
    \hline
    \llap{$Jc_1 \simeq {}$ }$U F c_1$ &  $\stackrel{\tilde{f}}{\longrightarrow}$ & $c_2$
  \end{tabular}
  }
  $$
	Hence morphisms $f$ in the co-Kleisli category $\mathrm{Kl}(J)$,
	definition~\ref{coKleisli}, are equivalently morphisms in
	$\mathcal{C}$ of the form $\tilde f : J c_1 \longrightarrow c_2$.
	According to proposition~\ref{AdjunctsInTermsOfUnitAndCounit}, the
	adjunct morphisms are related by $\tilde{f} = \epsilon_{c_2} \circ f$,
	where we have identified $f \simeq Uf$, and $f = J(\tilde{f}) \circ
	\Delta_{c_1}$, where we have identified $J(\tilde{f}) \simeq
	F(\tilde{f})$.

	Under identification between the adjunct morphisms $f$ and
	$\tilde{f}$, the composition of morphisms $g\circ f$ in
	$\mathrm{Kl}(J)$ is given by the ``co-Kleisli composite''
  $$
    \widetilde{g \circ f}
    :
    J c_1 \stackrel{\Delta_{c_1}}{\longrightarrow} J J c_1
    \stackrel{J (\tilde f)}{\longrightarrow}
    J c_2
    \stackrel{\tilde g}{\longrightarrow}
    c_3
    \,.
  $$

	The fact that a morphism between free coalgebras $f\colon Fc_1 \to
	Fc_2$ must be of the form $f = J(\tilde{f}) \circ \Delta_{c_1}$, as
	observed above, follows from the general
	proposition~\ref{CofreeForgetAdjunction}. But there is also an
	elementary way to see it. Consider the following diagram:
	$$
		\begin{gathered}
		\xymatrix{
			Jc_1
				\ar@/^2pc/[rrrr]^{\tilde{f} = \epsilon_{c_2} \circ f}
				\ar[rr]^f
				\ar[d]_{\Delta_{c_1}} &&
			Jc_2
				\ar[rr]^{\epsilon_{c_2}}
				\ar[d]^{\Delta_{c_2}}
				\ar[drr]^{\mathrm{id}} &&
			c_2
			\\
			JJc_1
				\ar@/_2pc/[rrrr]_{J(\tilde{f})}
				\ar[rr]_{Jf} &&
			JJc_2
				\ar[rr]_{\epsilon_{Jc_2}\simeq J\epsilon_{c_2}} &&
			Jc_2
		}
		\end{gathered}
		\, ,
	$$
	which commutes because $f$ is a morphism of free coalgebras and
	because of the counit-coproduct identities of a comonad. From the
	commutativity, as was desired, it follows that $f = J(\tilde{f}) \circ
	\Delta_{c_1}$, with $\tilde{f} = \epsilon_{c_2} \circ f$.
\end{remark}
\begin{definition}
  \label{conservativeFunctor}
  A functor $F : \mathcal{D}\longrightarrow \mathcal{C}$ is called \emph{conservative}
  if it reflects equivalences, hence if for a morphism $f$ in $\mathcal{D}$ we
  have that if $F(f)$ is an equivalence then already $f$ was an equivalence.
\end{definition}
\begin{theorem}[Beck monadicity theorem, e.g.,\ {\cite[vol. 4 sect. 2]{Borceux}}]
  \label{BeckMonadicityTheorem}
  Sufficient conditions for an adjunction $(L \dashv R)$, definition~\ref{AdjointFunctor},
  to be equivalent to a comonadic adjunction $(\mathrm{U}\dashv \mathrm{F})$ as in proposition~\ref{CofreeForgetAdjunction} is that
  \begin{enumerate}
    \item $L$ is conservative, definition~\ref{conservativeFunctor};
    \item $L$ preserves certain limits called \emph{equalizers of $L$-split pairs}.
  \end{enumerate}
\end{theorem}
Hence it is useful to record some facts about conservative functors:
\begin{proposition}[e.g.,\ {\cite[lemma 1.3.2]{Johnstone02}}]
  \label{PullbackAlongEpisIsConservative}
  For $\mathbf{H}$ a category of sheaves (Definition \ref{sheaf}) and $f : X \longrightarrow Y$ an epimorphism
  in $\mathbf{H}$, the pullback functor $f^\ast : \mathbf{H}_{/Y} \longrightarrow \mathbf{H}_{/X}$
  (proposition \ref{BaseChangeAdjunctions})
  is conservative, definition~\ref{conservativeFunctor}.
\end{proposition}
\begin{proposition}
  \label{ConservativeFunctorReflectsLimitsWhichItPreserves}
  A conservative functor reflects all the limits and colimits which it preserves.
\end{proposition}

\begin{proposition}[comonadic descent, e.g.,\ {\cite[2.4]{JanelidzeTholen94}}]
  \label{comonadicdescent}
  Given an epimorphism $\xymatrix{ X \ar@{->>}[r]^{f} & Y}$ in a
  category of sheaves $\mathbf{H}$, with induced base change comonad
  $$
    J := f^\ast f_\ast : \mathbf{H}_{/X} \to \mathbf{H}_{/X}
  $$
  (via proposition~\ref{BaseChangeAdjunctions} and proposition~\ref{CoMonadsFromAdjunctions}),
  then there is an equivalence of categories
  $$
    \mathbf{H}_{/Y} \stackrel{\simeq}{\longrightarrow} \mathrm{EM}(J)
  $$
  between the slice category $\mathbf{H}_{/Y}$ (definition \ref{SliceCategory})
  and the Eilenberg-Moore category of $J$-coalgebras in $\mathbf{H}_{/X}$, definition~\ref{EMOverComonad}.

  Moreover, under this identification the comonadic adjunction $(U_J \dashv F_J)$ from proposition~\ref{CofreeForgetAdjunction}
  coincides with the base change adjunction $(f^\ast \dashv f_\ast)$ of proposition~\ref{BaseChangeAdjunctions}:
  $$
    (U_J \dashv F_J) \simeq (f^\ast \dashv f_\ast)
    \,.
  $$
\end{proposition}
\begin{proof}
	Since $f$ is assumed to be epi,
	proposition~\ref{PullbackAlongEpisIsConservative} says that $f^\ast$
	is conservative. Moreover, since $f^\ast$ is right adjoint to $f_!$ by
	proposition~\ref{BaseChangeAdjunctions}, it preserves all small
	limits (proposition~\ref{rightadjointpreserveslimits}). Therefore
	the conditions in the monadicity theorem \ref{BeckMonadicityTheorem}
	are satisfied:
  $$
    \begin{gathered}
    \xymatrix{
      \mathbf{H}_{/X}
        \ar@<-4pt>[rr]_{f_\ast}
        \ar@{<-}@<+4pt>[rr]^{f^\ast}
        \ar@{=}[d]
        &&
      \mathbf{H}_{/Y}
      \ar[d]^\simeq
      \\
      \mathbf{H}_{/X}
        \ar@<-4pt>[rr]_{F_J}
        \ar@{<-}@<+4pt>[rr]^{U_J}
        &&
      \mathrm{EM}(J)
    }
    \end{gathered}
    \, .
  $$
\end{proof}

\newpage
\bibliographystyle{utphys-alpha}
\bibliography{toposjets}

\end{document}